\documentclass[11pt,a4paper]{amsart}

\usepackage{fullpage}
\usepackage{amsmath}
\usepackage{amsfonts}
\usepackage{amssymb}
\usepackage{mathrsfs}
\usepackage{mathtools}
\usepackage{bbm}
\usepackage{tikz-cd}
\usepackage[english]{babel}
\usepackage[utf8]{inputenc}
\usepackage{csquotes}
\usepackage{enumitem}
\usepackage{multicol}
\usepackage{calc}
\usepackage{extpfeil}

\setlist[itemize,1]{label={--\,}}
\setlist[enumerate,1]{label=(\roman*)}

\usepackage[backend=biber, maxbibnames=50, style=alphabetic, giveninits=true]{biblatex}
\renewbibmacro{in:}{}

\PassOptionsToPackage{hyphens}{url}\usepackage{hyperref}

\DeclareRobustCommand{\gobblefour}[4]{}

\makeatletter
\renewcommand{\tocsection}[3]{%
  \indentlabel{\@ifnotempty{#2}{\ignorespaces#1 \makebox[\widthof{00.}][l]{#2.}\quad}}#3}
\renewcommand{\tocsubsection}[3]{%
  \indentlabel{\@ifnotempty{#2}{\ignorespaces#1 \makebox[\widthof{00.0.}][l]{#2.}\quad}}#3}
\makeatother

\setlist[itemize]{leftmargin=*}
\setlist[enumerate]{leftmargin=*}

\newtheorem{thm}{Theorem}[section]
\newtheorem{defin}[thm]{Definition}
\newtheorem{rem}[thm]{Remark}

\newtheorem{prop}[thm]{Proposition}
\newtheorem{cor}[thm]{Corollary}

\newtheorem{lemma}[thm]{Lemma}

\newtheorem{manualtheoreminner}{Theorem}
\newenvironment{manthm}[1]{%
	\renewcommand\themanualtheoreminner{#1}%
	\manualtheoreminner
}{\endmanualtheoreminner}

\newcommand{\1}{\mathbbm{1}}
\newcommand{\A}{\mathbb A}

\newcommand{\C}{\mathbb C}

\newcommand{\G}{\mathbb G}

\newcommand{\N}{\mathbb N}
\newcommand{\Q}{\mathbb Q}

\newcommand{\Z}{\mathbb Z}

\newcommand{\ab}{\mathrm{ab}}

\newcommand{\alg}{\mathrm{alg}}

\newcommand{\BPKE}{{B_{\vert K}^{\otimes E}}}
\newcommand{\GBPKE}{{G\text{-}B_{\vert K}^{\otimes E}}}
\newcommand{\BPKEPairs}{{\mathcal{B}_{\vert K}^{\otimes E}}}
\newcommand{\BPKF}{{B_{\vert K}^{\otimes F}}}
\newcommand{\BPQE}{{B_{\vert\Q_p}^{\otimes E}}}

\newcommand{\BPKKEE}{{B_{\vert K_1}^{\otimes E_1}}}

\newcommand{\cat}{\mathrm{cat}}
\newcommand{\cHom}{\mathcal{H}om}
\newcommand{\ci}{\mathrm{ci}}

\newcommand{\cont}{\mathrm{cont}}
\newcommand{\cris}{\mathrm{cris}}
\newcommand{\crys}{\mathrm{cris}}

\newcommand{\cyc}{\mathrm{cyc}}

\newcommand{\dR}{{\mathrm{dR}}}
\newcommand{\End}{\mathrm{End}}

\newcommand{\Fil}{\mathrm{Fil}}
\DeclareMathOperator{\Frac}{Frac}

\newcommand{\Gal}{{\mathrm{Gal}}}
\newcommand{\GL}{\mathrm{GL}}
\newcommand{\GSp}{\mathrm{GSp}}

\newcommand{\Hom}{\mathrm{Hom}}
\newcommand{\HT}{\mathrm{HT}}
\newcommand{\id}{\mathrm{id}}
\newcommand{\Id}{\mathrm{Id}}

\newcommand{\Ind}{\mathrm{Ind}}

\newcommand{\length}{\mathrm{length}}

\newcommand{\Nm}{\mathrm{Nm}}

\newcommand{\Ob}{\mathrm{Ob}\,}

\newcommand{\PGL}{\mathrm{PGL}}

\newcommand{\ptri}{{\mathrm{ptri}}}

\newcommand{\red}{\mathrm{red}}
\newcommand{\reg}{\mathrm{reg}\,}
\newcommand{\qreg}{\mathrm{qreg}}
\newcommand{\Rep}{\mathrm{Rep}}
\DeclareMathOperator{\res}{res}

\newcommand{\rig}{\mathrm{rig}}
\DeclareMathOperator{\rk}{rk}
\newcommand{\scn}{{\mathrm{scn}}}

\newcommand{\SL}{\mathrm{SL}}

\newcommand{\sms}{{\mathrm{ss}}}
\newcommand{\st}{{\mathrm{st}}}

\newcommand{\Sym}{{\mathrm{Sym}}}

\newcommand{\stri}{{\mathrm{stri}}}
\newcommand{\pstri}{{\mathrm{pstri}}}
\newcommand{\tri}{{\mathrm{tri}}}

\newcommand{\Vect}{\mathrm{Vect}}

\newcommand{\xto}{\xrightarrow}
\newcommand{\into}{\hookrightarrow}
\newcommand{\onto}{\twoheadrightarrow}
\newcommand{\xonto}{\xtwoheadrightarrow}
\newcommand{\ovl}{\overline}
\newcommand{\wtl}{\widetilde}

\newcommand{\ccirc}{\kern0.5ex\vcenter{\hbox{$\scriptstyle\circ$}}\kern0.5ex}

\newcommand{\second}{{\prime\prime}}

\newcommand{\calA}{\mathcal{A}}
\newcommand{\calB}{\mathcal{B}}
\newcommand{\cC}{\mathcal{C}}
\newcommand{\calC}{\mathcal{C}}
\newcommand{\cD}{\mathcal{D}}
\newcommand{\calD}{\mathcal{D}}

\newcommand{\calE}{\mathcal{E}}

\newcommand{\cG}{\mathcal{G}}

\newcommand{\calP}{\mathcal{P}}

\newcommand{\calS}{\mathcal{S}}
\newcommand{\cS}{\mathscr{S}}

\newcommand{\cW}{\mathcal{W}}

\newcommand{\calH}{\mathcal{H}}

\newcommand{\bB}{{\mathbf{B}}}
\newcommand{\bD}{{\mathbf{D}}}

\newcommand{\bG}{{\mathbf{G}}}

\newcommand{\bP}{{\mathbf{P}}}
\newcommand{\bS}{{\mathbf{S}}}

\newcommand{\udelta}{{\underline{\delta}}}
\newcommand{\ueta}{{\underline{\eta}}}

\newcommand{\unu}{{\underline{\nu}}}

\newcommand{\ui}{{\underline{i}}}

\newcommand{\uu}{{\underline{u}}}
\newcommand{\uv}{{\underline{v}}}

\newcommand{\Qp}{\overline{\Q}_p}

\newcommand{\fm}{{\mathfrak{m}}}

\title{Lifting trianguline Galois representations \\ along isogenies}
\author{Andrea Conti}

\begin{document}

\begin{abstract}
	Given a central isogeny $\pi\colon G\to H$ of connected reductive $\Qp$-groups, and a local Galois representation $\rho$ valued in $H(\Qp)$ that is trianguline in the sense of Daruvar, we study whether a lift of $\rho$ along $\pi$ is still trianguline. We give a positive answer under weak conditions on the Hodge--Tate--Sen weights of $\rho$, and the assumption that the trianguline parameter of $\rho$ can be lifted along $\pi$. This is an analogue of the results proved by Wintenberger, Conrad, Patrikis, and Hoang Duc for $p$-adic Hodge-theoretic properties of $\rho$. We describe a Tannakian framework for all such lifting problems, and we reinterpret the existence of a lift with prescribed local properties in terms of the simple connectedness of a certain pro-semisimple group. While applying this formalism to the case of trianguline representations, we extend a result of Berger and Di Matteo on triangulable tensor products of $B$-pairs.
\end{abstract}

\maketitle

\setcounter{tocdepth}{1}

\section*{Introduction}

Fix a prime $p$ and a number field $F$. According to the Langlands conjectures, algebraic automorphic representations of the adelic points of a connected reductive $F$-group $G$ should provide us with a large class of representations of the absolute Galois group $\Gal(\ovl F/F)$, valued in the $p$-adic points of the Langlands dual of $G$. 
The Fontaine--Mazur conjecture and its generalizations predict, roughly, that such representations are those that are almost everywhere unramified and potentially semistable at the $p$-adic places. In the case of the group $\GL_{2/\Q}$ this is a theorem of Kisin and Emerton, building on the work of many people.


The following notation is in place throughout the introduction. Let $G$ and $G^\prime$ be two connected reductive groups over $F$ and let $H^\prime=({}^LG^\prime)^\circ$ and $H=({}^LG)^\circ$ be the neutral connected components of their Langlands duals, that we see as groups over $\Qp$. 
Given a morphism $S\colon H^\prime\to H$, one can compose a representation $\wtl\rho\colon\Gal(\ovl F/F)\to H^{\prime}(\Qp)$ with $S$ to obtain a representation $\rho\colon\Gal(\ovl F/F)\to H(\Qp)$. 
When $\wtl\rho$ is of automorphic origin, the Langlands functoriality conjectures predict the existence of a transfer of automorphic representations of $G^\prime(\A_F)$ to automorphic representations of $G(\A_F)$. The characterization of Galois representations via $p$-adic Hodge theory is compatible with such a transfer: if $\wtl\rho$ is potentially semistable at the $p$-adic places, then the same is true for $\rho=S\ccirc\wtl\rho$. One can ask whether the converse is true; admitting that the characterization suggested by the Fontaine--Mazur conjecture holds, this would amount to asking whether $\wtl\rho$ is of automorphic origin whenever $\rho$ is. 
In this direction one has the following result of Wintenberger and Conrad. Let $K$ and $E$ be two finite extensions of $\Q_p$ and let $S\colon H^\prime\to H$ be an isogeny of connected reductive $\Qp$-groups. Let $I_K$ be the inertia subgroup of $\Gal(\ovl K/K)$ and $\rho\colon I_K\to H(\Qp)$ a semistable representation, meaning that it is semistable, in the usual sense, after composition with a faithful (hence with any) representation of $H$. 
By a \emph{lift} of $\rho$ to $H^\prime$ we mean a representation $\wtl\rho\colon I_K\to H^\prime(E)$ that satisfies $S\ccirc\wtl\rho\cong\rho$.

\begin{manthm}{A}[{\cite[Théorème 1.1.3]{winteniso},\cite[Theorem 6.2]{conradlift}}]\label{wintint}\mbox{ }\\
Assume that the Hodge--Tate cocharacter $\G_{m,\C_p}\to H_{\C_p}$ attached to $\rho$ can be lifted along $S$ to a cocharacter $\G_{m,\C_p}\to H^\prime_{\C_p}$. 
Then $\rho$ admits a crystalline lift $I_K\to H^\prime(\Qp)$. 
	\end{manthm}

Given the Tannakian nature of the definition of crystalline representation, it is not surprising that the proof of Theorem \ref{wintint} involves Tannakian arguments. However, one cannot deduce the statement in a purely abstract way, and concrete manipulation of filtered $\varphi$-modules is essential to the proof. 

In the same spirit of Theorem \ref{wintint}, we have the more recent results of Hoang Duc \cite{hoangth} and Conrad \cite{conradlift}, where $I"K$ is replaced by either a local or a global Galois group and the possible ramification of a lift is studied in more detail. In \cite{nootlift}, Noot studies the analogue lifting problem for a compatible system of Galois representations attached to an abelian variety. The work of Di Matteo \cite{dimatadm} can also be interpreted in the above setting: he shows that if instead of an isogeny $H^\prime\to H$ one considers a (non one-dimensional) representation $S\colon\GL_m\to\GL_n$, described by a Schur functor, then any lift along $S$ of a Hodge--Tate (de Rham, semi-stable, crystalline) representation into $\GL_n(\Qp)$ is Hodge--Tate (de Rham, semi-stable, crystalline) up to a twist. 

It is known that, for many choices of a connected reductive groups $G$ over a number field $F$, algebraic automorphic representations of finite slope of $G(\A_F)$ live in $p$-adic families: these are rigid analytic (or adic) spaces equipped with global functions that specialize on a Zariski-dense set to the Hecke eigensystems of automorphic representations of $G(\A_F)$. 
By specializing such functions at an arbitrary point of a $p$-adic family one almost never obtains the Hecke eigensystem of an automorphic representation of $G(\A_F)$. However, one can often interpret such a specialization as the Hecke eigensystem of a $p$-adic automorphic form for $G$, and attach to it an $H(\Qp)$-valued Galois representation that will not be potentially semi-stable at the $p$-adic places. The correct notion describing the local behavior at $p$ of representations arising this way seems to be that of triangulinity, introduced by Colmez and inspired by earlier work of Kisin: if $K$ is a $p$-adic field, a continuous representation 
\[ \Gal(\ovl K/K)\to\GL_n(\Qp) \] 
is \emph{trianguline} if the corresponding $(\varphi,\Gamma)$-module, or equivalently $B$-pair, can be obtained by successive extensions of $(\varphi,\Gamma)$-modules, or $B$-pairs, of rank 1. The ordered list of 1-dimensional subquotients appearing in a triangulation is called its \emph{parameter}, and we say that the triangulation of a $B$-pair is \emph{strict} if it is the only one with a given parameter. The definition of triangulinity for a representation $\Gal(\ovl K/K)\to H(\Qp)$, with $H$ not equal to $\GL_n$, is more subtle and has been the object of V. Daruvar's recent Ph.D. thesis \cite{daruvarth}.

Roughly speaking, one conjectures that representations $\Gal(\ovl F/F)\to H(\Qp)$ that are almost everywhere unramified and trianguline at the $p$-adic places are attached to a $p$-adic automorphic form for $G$. Such a conjecture has been made precise only for those $G$ for which all of the ingredients are in place, that include $\GL_{2/\Q}$ and the definite unitary groups studied in \cite{brehelsch}, and proved only for $\GL_{2/\Q}$ (Emerton's ``overconvergent Fontaine--Mazur conjecture'' \cite{emerton}).

Our goal for this paper is to show that the trianguline condition is compatible with the $p$-adic Langlands transfer, in other words, to give an analogue of Theorem \ref{wintint} in the context of $p$-adic variation. Our main result has the following form; we point the reader to the main text for the precise statement. Let $E$ be a $p$-adic field, and let $S\colon H^\prime\to H$ be a morphism of connected reductive $E$-groups with finite central kernel. Let 
\[ \rho\colon\Gal(\ovl F/F)\to H(E) \]
be a continuous Galois representation. The quasi-regularity condition appearing in the statement below is a condition on the Hodge--Tate--Sen weights, that is for instance implied by the weights being all distinct for every embedding of the coefficient field into $\Qp$.

\begin{manthm}{B}[Corollary \ref{liftisogcor}]\label{mainint}
Let $E$ be a $p$-adic field and $\rho\colon\Gal(\ovl F/F)\to H(E)$ a continuous representation that is unramified outside of a finite set of places $\Sigma$, and quasi-regular and strictly trianguline at the $p$-adic places of $F$. Assume that:
\begin{enumerate}[label=(\roman*)]
\item for every $v\in\Sigma$, the restriction of $\rho$ at a decomposition group at $v$ admits a lift to $H^\prime(E)$;
\item the ``$H$-parameters'' of the triangulations of $\rho$ at the $p$-adic places can be lifted to ``$H^\prime$-parameters''.
\end{enumerate}
Then $\rho$ admits a lift $\wtl\rho\colon\Gal(\ovl F/F)\to H^\prime(E)$ that is unramified almost everywhere and trianguline at the $p$-adic places of $F$.
\end{manthm}

The intended application of Theorem \ref{mainint} is to the study of congruences between $p$-adic families of different kinds in purely Galoisian terms: if $\calE_{G^\prime}$ and $\calE_G$ are eigenvarieties associated with $G^\prime$ and $G$, respectively, and $S_\calE\colon\calE_{G^\prime}\to\calE_G$ is the rigid analytic map associated with the Langlands transfer along $S$, then one can hope to prove, by identifying $\calE_{G^\prime}$ and $\calE_G$ with spaces of deformations of trianguline Galois representations and applying Theorem \ref{mainint}, that a point of $\calE_G$ belongs to the image of $S_\calE$ if and only if its associated representation $\Gal(\ovl F/F)\to H(E)$ comes from a representation $\Gal(\ovl F/F)\to H^\prime(E)$ via $S$. This plan has been carried out in \cite{contigsp4} in the special case of the symmetric cube transfer from $\GL_2$ to $\GSp_4$.

Under the assumptions of Theorem \ref{mainint}, the existence of an arbitrary lift follows from a result of Conrad \cite{conradlift}, so all of our work is aimed at checking that such a lift is trianguline at the $p$-adic places. Condition (ii) of Theorem \ref{mainint} can be seen as an analogue of the assumption in Theorem \ref{wintint} that the Hodge--Tate cocharacter can be lifted.

We explain in more detail the structure of the paper and the results that lead to the proof of Theorem \ref{mainint}.

In Sections \ref{tpsub} and \ref{pbschur}, we give an abstract Tannakian description of the problem of lifting Galois representations with prescribed local properties along an isogeny. 
Consider a field $E$ of characteristic 0, an $E$-linear, neutral Tannakian category $\calC$ and a full tensor subcategory $\calD$ of $\calC$. The category $\calC$ should be thought of as the ambient category, for instance that of $\Q_p$-linear representations of $\Gal(\ovl F/F)$, while the objects of its subcategory $\calD$ are those that satisfy a condition we are interested in, for instance the representations that possess some desirable local properties. We then build a new category $\ovl\calD$, sitting between $\calC$ and $\calD$, generated under direct sum by all of the objects $V\in\calC$ such that
\[ V\otimes W \in\calD \]
for some $W\in\calC$. 
In the abstract setting, we can study the discrepancy between $\calD$ and $\ovl\calD$ by means of Tannakian duality. If
\begin{equation}\label{gpint} G_\calC\xonto{\phantom{g}} G_{\ovl\calD}\xonto{\phantom{g}} G_\calD 
	\end{equation} 
is the sequence of Tannakian fundamental groups dual to $\calD\subset\ovl\calD\subset\calC$, then we show that $G_{\ovl\calD}$ is a kind of universal covering of $G_\calD$ ``inside of $G_\calC$''. Under reasonable assumptions on $\calD$ (see condition \eqref{potcond} and the discussion following it) we can assume that the groups of connected components are constant along \eqref{gpint}, so we focus on the neutral components.

\begin{manthm}{C}\label{tannint}\mbox{ }
\begin{enumerate}	
	\item (Proposition \ref{schur}) The objects of $\ovl\calD$ are precisely the $V\in\calC$ for which there exists a Schur functor $S\colon\calC\to\calC$ such that $S(V)$ is an object of $\calD$ of dimension strictly larger than 1.
	\item (Proposition \ref{ovlDuni}) The group $G_{\ovl\calD}^\circ$ is the inverse limit of all pro-algebraic groups $H$ fitting into a diagram
\[ G_\calC^\circ\xonto{\phantom{g}} H\xonto{g} G_\calD^\circ \]
where $g$ is a central isogeny.
\end{enumerate}
	\end{manthm}

The fact that representations $\GL_m\to\GL_n$ are described by Schur functors allows us to reinterpret results of the type of Theorems \ref{wintint} and \ref{mainint} as stating that, for certain choices of $\calC$ and $\calD$, the inclusion $\calD\subset\ovl\calD$ is an equality. 

Unfortunately, Theorem \ref{tannint} by itself is not sufficient to deduce that $\ovl\calD=\calD$ in some concrete interesting example. However, it plays an important role in the proof of the following local result. Here $K$ and $E$ are again two $p$-adic fields, and we write \emph{$\BPKE$-pair} to emphasize what base and coefficient field we are working with; $\BPKE$-pairs of slope 0 correspond to $E$-linear representations of $\Gal(\ovl K/K)$. 

\begin{manthm}{D}[Theorems \ref{ptri} and \ref{thsttri}]\label{bpint} 
	Let $W$ be a $\BPKE$-pair and $S$ a Schur functor. If $S(W)$ is quasi-regular and triangulable, then $W$ is potentially triangulable. If moreover $S(W)$ admits a strict triangulation whose parameter ``lifts to a candidate parameter for $W$'', then $W$ admits a strict triangulation of this candidate parameter.
	\end{manthm}

By replacing $\BPKE$-pairs with modifications of slope 0, we can reduce Theorem \ref{bpint} to the case of $\BPKE$-pairs attached to actual Galois representations; these form a Tannakian category that is neutral, contrary to that of all $\BPKE$-pairs. We are then in a position to apply Theorem \ref{tannint}(i), that allows us to reduce the statement of Theorem \ref{bpint} to a single of Schur functor of length $n$, for each $n$. The choice $S=\Sym^n$ presents some symmetries that we can exploit. Section \ref{galois} is devoted to the actual manipulation of $B$-pairs leading to the proof of Theorem \ref{bpint}. 

In \cite{dimatadm} Di Matteo proved a statement similar to Theorem \ref{bpint} with ``triangulable'' replaced by Hodge--Tate, de Rham, potentially semi-stable, or crystalline. 
Within the Tannakian framework introduced above, we can reformulate his result, in the special case of $\BPKE$-pairs of slope 0, as follows: if $\calC$ is the category of continuous, finite-dimensional $E$-linear representations of $\Gal(\ovl K/K)$ and $\calD$ is the full subcategory tensor generated by those that are potentially semi-stable up to a twist, then $\calD=\ovl\calD$. Our Theorem \ref{bpint} corresponds instead to the choice of $\calD$ as the category tensor generated by the quasi-regular, potentially trianguline representations.

Furthermore, Berger and Di Matteo \cite{berdimtri} proved that, if $V$ and $W$ are two $\BPKE$-pairs such that $V\otimes W$ admits a triangulation whose 1-dimensional subquotients are restrictions to $G_K$ of $\BPQE$-pairs, then both $V$ and $W$ are potentially triangulable. We could combine this result with Theorem \ref{tannint} to show Theorem \ref{bpint} under some additional assumptions on the triangulation. Our technique allows us to work with the weaker condition of quasi-regularity. 

The proof of Theorem \ref{mainint} consists in constructing, for an arbitrary $n$, a crystalline period of $W$ from a crystalline period of $\Sym^nW$: a triangulable $\BPKE$-pair always admits such a period up to a twist, and on the other hand such a period determines a rank 1 sub-$\BPKE$-pair, allowing us to work by induction on the rank of $W$. 

Finally, Sections \ref{secsttri} and \ref{seclift} deal with going from Theorem \ref{bpint} to Theorem \ref{mainint}. The main obstacle here are the subtleties in the definition of the ``trianguline'' condition for a representation 
\[ \rho\colon\Gal(\ovl K/K)\to H(\Qp), \]
$K$ a $p$-adic field, when $H$ is not a general linear group. This problem has been studied in depth in the Ph.D. thesis of V. Daruvar \cite{daruvarth}, who gives a Tannakian definition of triangulable $H$-$(\varphi,\Gamma)$-module that turns out to be practical for studying, for instance, deformation spaces of such objects. We restate his definition in terms of $B$-pairs and specialize it to the case when the coefficients are a field, rather than an affinoid algebra, but we allow for a quasi-split group $H$ rather than just a split group as he does (Definition \ref{deftridar}). 
Alternatively, one could give a ``naive'' definition of triangulinity, in Wintenberger's style, saying that $\rho$ is trianguline if and only if $S\ccirc\rho$ is trianguline for a faithful (hence for any) representation $S$ of $H$. 
Daruvar's definition allows us to speak naturally of parameters, while the naive definition allows us to apply Theorem \ref{bpint}. We bridge the gap by proving that the two definitions are equivalent:

\begin{manthm}{E}[Proposition \ref{Sstrict}]
An $H$-$\BPKE$-pair is triangulable if and only if there exists a faithful representation $S\colon H\to\GL_n$ of $H$ such that the $\BPKE$-pair of  rank $n$ attached to $S(W)$ is triangulable.
\end{manthm}

\smallskip

\subsection*{Notation and terminology.} 
All categories we work with are assumed to be essentially small. 
We denote by $\Ob(\calC)$ the class of objects of a category $\calC$; however, when this does not cause any ambiguity, we may write $X\in\calC$ rather than $X\in\Ob(\calC)$ for an object $X$ of $\calC$. 
For all the tensor categories under consideration the tensor product will be denoted with $\otimes$. We denote by $\Vect_E$ the category of vector spaces over a field $E$. If $V\in\Vect_E$, we write $\GL(V)$ for the group scheme over $E$ of automorphisms of $V$. Given an affine group scheme $G$ and a field $F$, we write $\Rep_F(G)$ for the category of algebraic $F$-representations of $G$, equipped with the usual structure of neutral Tannakian category where the fiber functor is the forgetful one. By a \emph{Tannakian subcategory} $\calD$ of $\calC$ we mean a strictly full (i.e., $\calD$ is full and if $X\cong Y$ in $\calC$ and $Y\in\calD$, then $X\in\calD$) subcategory of $\calC$ closed under the formation of subquotients, direct sums, tensor products, and duals. If $\calS$ is a set of objects of $\calC$, by \emph{tensor category generated by $\calS$} we mean the smallest Tannakian subcategory of $\calC$ containing all the objects in $\calS$ (in particular, it will contain all the duals of the objects in $\calS$).

If $\calC$ is a neutral Tannakian category, we write $G_\calC$ for its Tannakian fundamental group. If $V$ is an object of a $\calC$, we still write $V$ for its image under a specified fiber functor when this does not create confusion.

Throughout the text $p$ will denote a fixed prime number. Given a field $K$, we write $\ovl K$ for an algebraic closure of $K$ (fixed once we use it for the first time) and $G_K$ for the absolute Galois group $\Gal(\ovl K/K)$, equipped with the profinite topology. We fix once for all an extension of the $p$-adic valuation of $\Q_p$ to $\Qp$, and denote by $\C_p$ a completion of $\Qp$ for this valuation. By a ``$p$-adic field'' we will always mean a finite extension of $\Q_p$. 

For every positive integer $m$, we write $\mu_m$ for the group scheme over $\Z$ of $m$-th roots of unit. We do not bother to add specifications for when we are looking at a base change of it to an obvious base (typically a fixed base field). 
When $K$ is a $p$-adic field we write $K^\Gal$ for the Galois closure of $K/\Q_p$ in $\Qp$, $K_0$ for the largest unramified extension of $\Q_p$ contained in $K$, and we set $K_n=K(\mu_{p^n}(\ovl K))$, $K_\infty=\bigcup_{n\ge 1}K_n$, 
$\Gamma_K=\Gal(K_\infty/K)$ and $H_K=\Gal(\ovl K/K_\infty)$. We write $\chi_K^\cyc$ for the cyclotomic character, both $\Gamma_K\to\Z_p^\times$ and $G_K\to\Z_p^\times$, since this will not cause any ambiguity. We pick the Hodge--Tate weight of the cyclotomic character $\chi_{\Q_p}^\cyc$ to be $-1$.

With the hope to make it clearer to the reader when the group representations under consideration are linear or semilinear, we will write the coefficients on the right and as a lower index for linear representations (such as in $\Rep_E(G_K)$) and on the left for semilinear representations (such as in $\bB_\dR\Rep(G_K)$).

For every $n\geq 1$ we write $\1_n$ for the $n\times n$ unit matrix.

By ``image'' of a morphism of (group) schemes over a field of characteristic 0 we always mean the scheme-theoretic image (in the case of group schemes, we equip it with the structure of group scheme induced by that of the target). 

By a \emph{line} in a free module over an arbitrary ring we mean a free submodule of rank $1$. By a \emph{saturated line} we mean a line that is not properly contained in any other line.

\medskip

\textbf{Acknowledgments.} This paper was prepared while I was a postdoctoral fellow at a series of institutions: the Max Planck Institute for Mathematics in Bonn, Heidelberg University, and currently the University of Luxembourg. I am grateful to all of them for their financial support and the excellent working environment. I thank Laurent Berger and Giovanni Di Matteo for interesting discussions about this paper, Giovanni's thesis and their subsequent joint work. I thank Vincent Daruvar for a careful reading of a very early version of this paper. I also thank Gebhard B\"ockle, Eugen Hellmann, Adrian Iovita, Cédric Pépin, Vincent Pilloni, Benjamin Schraen and Jacques Tilouine for their comments.

\tableofcontents

\medskip

\numberwithin{thm}{section}
\numberwithin{equation}{section}

\section{Tensor products landing in a subcategory}\label{tpsub}

Let $E$ be a field of characteristic $0$.
Let $\calC$ be a neutral Tannakian category over $E$. For a (necessarily neutral) Tannakian subcategory $\calD$ of $\calC$, we define another category $\ovl\calD$ as the full subcategory of $\calC$ whose objects are the $V\in\Ob(\calC)$ having the following property: there exists a positive integer $m$ and a collection of objects $V_i$, $i=1,\ldots,m$, such that
\begin{enumerate}
\item $V$ is isomorphic to $\bigoplus_iV_i$ in $\calC$, and
\item for every $i\in\{1,\ldots,m\}$ there exists $W_i\in\Ob(\calC)$ satisfying $V_i\otimes W_i\in\Ob(\calD)$.
\end{enumerate}
We call \emph{basic objects} of $\ovl\calD$ the objects $V$ of $\ovl\calD$ for which there exists $W\in\Ob(\calC)$ such that $V\otimes W\in\Ob(\calD)$. Such a $W$ will automatically be an object of $\ovl\calD$. By definition, an object of $\ovl\calD$ is a direct sum of basic objects, but a non-trivial direct sum of basic objects or an extension of such can still be basic. Clearly all irreducible objects of $\ovl\calD$ are basic.

The category $\ovl\calD$ is a Tannakian subcategory of $\calC$. Indeed:
\begin{itemize}
\item It is clearly stable under direct sums.
\item It is stable under subquotients: Consider an exact sequence $0\to V\to V^\prime\to V^\second$ in $\calC$, such that $V^\prime\in\Ob(\ovl\calD)$. Then there exists $W^\prime\in\Ob(\ovl\calD)$ such that $V^\prime\otimes W^\prime\in\Ob(\calD)$. The sequence $0\to V\otimes W\to V^\prime\otimes W^\prime\to V^\second\otimes W^\second\to 0$ is exact in $\ovl\calD$ (because all objects are $E$-vector spaces), and the central object belongs to $\Ob(\calD)$. Since $\calD$ is Tannakian, it is stable under subquotients, so $V\otimes W$, $V^\second\otimes W^\second$ are objects of $\Ob(\calD)$, and $V,V^\second$ are objects of $\ovl\calD$.
\item It is stable under tensor products: If $V,V^\prime\in\Ob(\ovl\calD)$, then there exist $W,W^\prime\in\Ob(\ovl\calD)$ such that $V\otimes W$, $V^\prime\otimes W^\prime\in\Ob(\calD)$, so $(V\otimes V^\prime)\otimes (W\otimes W^\prime)\in\Ob(\calD)$.
\item It is stable under duals: If $V\in\Ob(\ovl\calD)$, then there exists $W\in\Ob(\ovl\calD)$ such that $V\otimes W\in\Ob(\calD)$, so $V^\vee\otimes W^\vee\cong(W\otimes V)^\vee$ is the dual of an object of $\calD$, hence also an object of $\calD$.
\end{itemize}

\begin{rem}
If $X$ and $Y$ are two objects of $\ovl\calD$ and $Z$ is an extension of $X$ by $Y$ in $\calC$, then $Z$ is not necessarily an object of $\ovl\calD$.
	\end{rem}

We prove that applying the above construction a second time produces no new category. Let $\ovl{\ovl\calD}$ be the category obtained by applying the construction to the inclusion $\ovl\calD\subset\calC$. 

\begin{lemma}
The categories $\ovl{\ovl\calD}$ and $\ovl\calD$ coincide.
	\end{lemma}

\begin{proof}
If $W$ is a basic object of $\ovl{\ovl\calD}$, then there exists $W^\prime\in\Ob(\ovl{\ovl\calD})$ such that $W\otimes W^\prime\in\Ob(\ovl\calD)$. We decompose $W=\bigoplus_{i=1}^mW_i$ and $W^\prime=\bigoplus_{j=1}^nW_j^\prime as sums of basic objects of \ovl\calD$. Let $V_i, 1\le i \le n$, and $V_j^\prime, 1\le j\le m$, be objects of $\ovl\calD$ satisfying $W_i\otimes V_i\in\Ob(\calD)$ and $W_i^\prime\otimes V_i^\prime\in\Ob(\calD)$ for every $i,j$. Then, for each $i$ and $j$, $V_i\otimes (V_j^\prime\otimes W_j\otimes W_j^\prime)$ is an object of $\calD$, hence all of the $V_i$ are objects of $\ovl\calD$, and so is their direct sum $W$.
\end{proof}

Let $G_\calC$, $G_\calD$, $G_{\ovl\calD}$ be the Tannakian fundamental groups of $\calC$, $\calD$, $\ovl\calD$, respectively. They are pro-algebraic groups over $E$. By \cite[Proposition 2.21(a)]{delmiltan}, the inclusions $\calD\into\ovl\calD\into\calC$ give faithfully flat morphisms of affine group schemes over $E$:
\begin{equation}\label{embed} G_\calC\onto G_{\ovl\calD}\to G_\calD. \end{equation}

\begin{rem}\label{1dim}
The category $\ovl\calD$ contains all $1$-dimensional objects of $\calC$, since duals exist in $\ovl\calD$ and the evaluation map $X\otimes X^\vee\to\1_\calD$ is an isomorphism when $X$ is $1$-dimensional. By Tannakian duality, we obtain that the algebraic characters of $G_\calC$ all factor through the morphism $G_\calC\to G_{\ovl\calD}$ of \eqref{embed}. 
\end{rem}

Recall that kernels exist in the category of pro-algebraic groups over $E$. Let $I=\ker{(G_{\ovl\calD}\onto G_\calD)}$.
For an object $V$ of $\ovl\calD$, we denote by $\rho_V\colon G_{\ovl\calD}\to\GL(V)$ the representation associated with $V$ by Tannakian duality, and by $I_V$ and $G_V$ the scheme-theoretic images of $I$ and $G_{\ovl\calD}$, respectively, in $\GL(V)$.

\begin{lemma}\label{Iscalar}
If $V$ is basic in $\ovl\calD$, then $I_V$ is contained in the center of $\GL(V)$. 
\end{lemma}

\begin{proof}
By definition of $\ovl\calD$, there exists an $E$-vector space $W$ and a representation $\rho_W\colon G_{\ovl\calD}\to\GL(W)$ such that $\rho_{V\otimes W}$ factors through $G_{\ovl\calD}\onto G_\calD$, that is, $\rho_V\otimes\rho_W(I)$ is a direct sum of copies of the trivial representation. 
Now $\rho_{V\otimes W}=\rho_V\otimes\rho_W$, and the only way a tensor product of two $E$-representations of $I$ can be a direct sum of copies of the trivial representation is if the two of them factor through characters of $I$. This means precisely that $I_V$ and $I_W$ are contained in the center of $\GL(V)$.
\end{proof}

\begin{cor}\label{dual}
If $V$ is an object of $\ovl\calD$, then $V\otimes V^\vee$ is an object of $\calD$.
\end{cor}

\begin{proof}
Since $I$ is central in $\GL(V)$ by Lemma \ref{central}, it acts trivially on $V\otimes V^\vee$.
\end{proof}


\begin{lemma}\label{central}
The pro-algebraic subgroup $I$ of $G$ is contained in the center of $G$.
\end{lemma}

\begin{proof}
Write $G_{\ovl\calD}$ as an inverse limit $\varprojlim_{i\in\N}G_i$ of algebraic group schemes, that is, group schemes whose Hopf algebras are finite-dimensional as $E$-vector spaces. Fix $i\in\N$. 
By \cite[Corollary 2.5]{delhodge} the group scheme $G_i$ has a faithful, 
finite-dimensional $E$-representation $\rho_i\colon G_i\to\GL(V_i)$. The the projection $G_{\ovl\calD}\to G_i$ composed with the representation $\rho_i$ gives a representation of $G_{\ovl\calD}$ on $V_i$, that is the Tannakian dual of an object of $\calD$ that we still denote by $V_i$. 
Let $I_i$ be the image of $I$ under $G_{\ovl\calD}\to G_i$. 
For every $i$, write $V_i$ as a direct sum $\bigoplus_jV_{ij}$ of basic objects. 
By Lemma \ref{Iscalar}, $\rho_i(I_i)$ acts via scalar endomorphisms on $V_{ij}$ for every $j$, so it is central in $\GL(V_i)$ and in particular in $\rho_i(G_i)$. Since $\rho_i$ is faithful, this means that $I_i$ is central in $G_i$. By taking a limit over $i\in\N$ and using the fact that $I=\varprojlim_{i\in\N}I_i$ because $I$ is a closed subgroup scheme of $G_{\ovl\calD}$, we conclude that $I$ is central in $G_{\ovl\calD}$.
\end{proof}

For every positive integer $m$ and every object $V$ of $\calC$, we embed $\mu_m$ into $\GL(V)$ in the usual way, by letting it act on $V$ via scalar automorphisms. For the rest of this section and throughout the next one we make the following assumption:
\begin{equation}\label{assum}\tag{$1$-dim} \calD\text{ contains all $1$-dimensional objects of }\calC. 
\end{equation}

\begin{lemma}\label{Ifinite}
Let $V$ be a basic object of $\ovl\calD$ and let $n=\dim_EV$. Then $I_V$ is contained in $\mu_n$.
\end{lemma}

\begin{proof}
Assumption \eqref{assum} implies that every algebraic character of $G_{\ovl\calD}$ factors through $G_{\ovl\calD}\to G_\calD$, that is, is trivial on $I$.
Lemma \ref{Iscalar}, $I_V$ is central in $\GL(V)$, so it is contained in the group $\G_m$ embedded in $\GL(V)$ as the subgroup of scalar endomorphisms. The restriction of the determinant of $\GL(V)$ to $G_V$ gives an algebraic character of $G_V$; by our previous observation, such a character has to be trivial on $I_V$. This implies that $I_V$ is contained in the subgroup $\mu_n$ of $\G_m$.
\end{proof}

\begin{rem}\mbox{ }
\begin{enumerate}
\item The group $I$ can be non-trivial, that is, the categories $\calD$ and $\ovl\calD$ can be different: as we will show in Section \ref{galois}, it is the case when $\calC$ is the category $\Rep_E(G_K)$ for some $p$-adic fields $K$ and $E$, and $\calD$ is the full subcategory of trianguline representations. We will see that in this example $\ovl\calD$ is the category of potentially trianguline representations, and there exist for every $K$ and $E$ potentially trianguline representations that are not trianguline (pick any semistabelian, non-semistable representation of $G_K$).
\item It can happen that for some object $V$ of $\ovl\calD$ the group $I_V$ has no $E$-points: by Lemma \ref{Ifinite}, it is always the case if $E$ does not contain any $n$-th roots of unity. Nevertheless, $I_V$ can be a non-trivial subgroup scheme of $\GL(V)$, hence such a $V$ is not necessarily an object of $\calD$.
\end{enumerate}
\end{rem}

\begin{cor}\label{Iprof}
The group $I$ is profinite. 
\end{cor}

\begin{proof}
In the last paragraph of the proof of Lemma \ref{central} we showed that $I=\varprojlim_{i\in\N}I_i$ where $I_i$ is isomorphic to $I_{V_i}$ for some object $V_i$ of $\ovl\calD$ (since we chose the representation $\rho_i$ in the proof of Lemma \ref{central} to be faithful). Writing $V_i$ as a direct sum of basic objects and applying Lemma \ref{Ifinite} we obtain that $I_{V_i}$ is finite, hence $I$ is profinite. 
\end{proof}

For later use, we prove a simple lemma. Let $\cS$ be a tensor generating set of $\ovl\calD$.

\begin{lemma}\label{Ivanishgen}
	If the image of $I$ in $\GL(V)$ is trivial for every $V\in\cS$, then $I$ is trivial.
	\end{lemma}

\begin{proof}
	Since $\cS$ is a tensor generating set of $\ovl\calD$, the assumption implies that the image of $I$ in $\GL(V)$ is trivial for every object $V$ of $\ovl\calD$, which implies that $\ovl\calD=\calD$.
	\end{proof}

\medskip

\section{Pullback via Schur functors}\label{pbschur}

Let $E,\calC,\calD,\ovl\calD$ be as in the previous section. 
Throughout the rest of the paper, underlined, Roman lower-case letters will always denote tuples of finite length whose entries are positive integers. 
Given a tuple $\uu$, we denote by $\length(\uu)$ the number of entries of $\uu$ and by $\ell(\uu)$ the sum of the entries of $\uu$.
Following \cite[Section 1.4]{deltens}, we recall the definition of the Schur functor $\bS^\uu$ in $\calC$.
For a finite dimensional $E$-vector space $V$ and an object $X$ of $\calC$, we define objects $V\otimes X$ and $\cHom(V,X)$ of $\calC$ by asking that
\begin{align*}
\Hom_\calC(V\otimes X,Y)&=\cHom(V,\Hom_\calC(X,Y)) \\
\Hom_\calC(Y,\cHom(V,X))&=\Hom_\calC(V\otimes Y,X)
\end{align*}
for every object $Y$ of $\calC$.

Let $V$ be an object of $\calC$. The symmetric group $\calS_{\ell(\uu)}$ on $\ell(\uu)$ elements acts on the object $V^{\otimes\ell(\uu)}$ of $\calC$ by permuting its factors. We index isomorphism classes of simple representations of $\calS_{\ell(\uu)}$ by tuples of sum $\ell(\uu)$: with each such tuple one associates a Young tableau with $\ell(\uu)$ entries, and we attach a representation to a tableau as in \cite[Lecture 4]{fulhar}. 
For every $\uu$ let $V_\uu$ be a representative of the isomorphism class indexed by $\uu$. By functoriality of $\cHom$ in the two arguments, the group $\calS_{\ell(\uu)}$ acts on $\cHom(V_\uu,V^{\otimes\ell(\uu)})$ via its actions on $V_\uu$ and $V$. 

For an object $X$ of $\calC$ carrying an action of a finite group $S$, the operator 
\[ e_S=\frac{1}{\lvert S\rvert}\sum_{s\in S}s\in\End(X) \]
is idempotent. The image of $e_S$ exists by the axioms of $E$-linear tensor categories, and is denoted by $X^S$.

\begin{defin}
We let $\bS^\uu(V)=\cHom(V_\uu,V^{\otimes\ell(\uu)})^{\calS_{\ell(\uu)}}$ (with the notation introduced just above). This defines a (non-tensor) functor from $\calC$ to itself, that we call the \emph{Schur functor} attached to $\uu$. 
\end{defin}

The Schur functor $\bS^\uu$ can be defined more explicitly by attaching to $\uu$ a suitable idempotent element in $\End(X^{\otimes\ell(\uu)})$ and taking its image, similarly to what one does in the classical theory of Schur functors in the category of vector spaces over a field.


\begin{rem}
The definition of Schur functors only requires the ambient category to be an $E$-linear tensor category (we refer to \cite[Section 1.2]{deltens} for the relevant axioms). In particular we can, and will, apply it to the category of $\BPKE$-pairs. In this case, we recover the definition from \cite[Section 1.4]{dimatadm}.
\end{rem}

\begin{rem}\label{schurcomp}\mbox{ }
\begin{enumerate}[label=(\roman*)]
\item If $V$ is a vector space over $E$ and $\uu$ a tuple, then by functoriality of $\bS^\uu\colon\Vect_E\to\Vect_E$ the $E$-linear action of $\GL(V)$ on $V$ induces an E-linear action of $\GL(V)$ on $\bS^\uu(V)$. This action defines a morphism of $E$-group schemes 
\begin{equation}\label{GLV} \GL(V)\to\GL(\bS^\uu(V)) \end{equation}
that we also denote by $\bS^\uu$. Since $E$ is of characteristic $0$, $\bS^\uu$ is simply the unique irreducible representation of $\GL(V)$ of highest weight $\uu$. The kernel of $\bS^\uu$ is the group scheme $\mu_{\ell(\uu)}$, embedded in the center of $\GL(V)$ in the usual way.
\item If $F\colon\calC\to\calC^\prime$ is an $E$-linear tensor functor, then $\bS^\uu(F(V))=F(\bS^\uu(V))$ for every object $V$ of $\calC$ and every tuple $\uu$. In particular, if $\calC$ is neutral Tannakian, the fiber functor commutes with the Schur functor $\bS^\uu$. 
\item If $V$ is an object of $\calC$ and $\rho_V\colon G_\calC\to\GL(V)$ is the representation attached to $V$ by Tannakian duality, then for every tuple $\uu$ the representation dual to $\bS^\uu(V)$ is $\bS^\uu\ccirc\rho_V$, where $\bS^\uu$ is the morphism $\GL(V)\to\GL(\bS^\uu(V))$ of part (i) of the remark.
\end{enumerate}
\end{rem}

\begin{prop}\label{schur}
Let $V\in\Ob(\calC)$ and $n=\dim_EV$. 
\begin{enumerate}
\item The object $V$ is a basic object of $\ovl\calD$ if and only if there exists a tuple $\uu$ with $\length(\uu)<n$ such that $\bS^\uu(V)\in\Ob(\calD)$.
\item If a tuple $\uu$ as in part (i) exists, then $\bS^\uv(V)\in\Ob(\calD)$ for every tuple such that $\gcd(\ell(\uu),n)\mid\ell(\uv)$.
\end{enumerate}
\end{prop}

\begin{proof}
We first prove the ``if'' of part (i). Let $\uu$ be a tuple such that $\length(\uu)<n$ and $\bS^\uu(V)\in\Ob(\calD)$. 
For every tuple $\uv$ such that $\ell(\uv)=\ell(\uu)$, the representation $\bS^\uu(V)$ of $\GL(V)$ factors through a faithful representation $\bS^\uu_0(V)$ of the reductive group $\GL(V)/\mu_{\ell(\uu)}$. 
Since $\bS_0^\uu(V)$ is faithful, \cite[Proposition 3.1]{delhodge} implies that $\bS_0^\uv(V)$ appears as a subrepresentation of $\bS_0^\uu(V)^{\otimes m}\otimes(\bS_0^\uu(V)^\vee)^{\otimes n}$ for some positive integers $m,n$. The same is true if we see these objects as representations of $\GL(V)$ via $\GL(V)\to\GL(V)/\mu_{\ell(\uu)}$, that is, if we replace $\bS_0^\uv(V)$ and $\bS_0^\uu(V)$ with $\bS^\uv(V)$ and $\bS^\uu(V)$, respectively. Since $\bS^\uu(V)$ is an object of $\calD$ and $\calD$ is stable under tensor products, duals and subquotients, $\bS^\uv(V)$ is also an object of $\calD$. We proved that $\bS^\uu(V)\in\Ob(\calD)$ for any $\uv$ with $\ell(\uv)=\ell(\uu)$. 
By the Littlewood--Richardson rule (see \cite[Appendix 8]{fulhar} for the classical version), the representation $\Sym^{\ell(\uu)-1}(V)\otimes V$ of $\GL(V)$ is a direct sum of representations of the form $\bS^\uv(V)$ with $\ell(\uv)=\ell(\uu)$, so it is an object of $\calD$. By definition of $\ovl\calD$, we conclude that $V$ is an object of $\ovl\calD$.

We now prove the ``only if'' of part (i) together with (ii). 
Let $\uu$ be an arbitrary tuple. 
With the notation introduced in Remark \ref{schurcomp}(iii), the representation $\rho_{\bS^\uu(V)}$ attached to $\bS^\uu(V)$ is $\bS^\uu\ccirc\rho_V$. 
Standard computations give that the kernel of the representation $\GL(V)\to\GL(\bS^\uu(V))$ is $\mu_{\ell(\uu)}$, so the image of $I$ under $\bS^\uu(\rho_V)$ is trivial if and only if $I$ is contained in $\mu_{\ell(\uu)}$. Assuming that $\uu$ satisfies $\bS^\uu(V)\in\Ob(\calD)$, we get that $I_V$ (the scheme-theoretic image of $I$ under $\rho_V$) is contained in $\mu_{\ell(\uu)}$. By Lemma \ref{Ifinite} $I_V$ is contained in $\mu_{n}$, hence in $\mu_{(\ell(\uu),n)}$. We conclude that $\bS^\uv(V)$ is an object of $\calD$ whenever $(\ell(\uu),n)\mid\ell(\uv)$.
\end{proof}

Since $\ovl\calD$ is tensor generated by the class of its basic objects, Proposition \ref{schur}(i) immediately gives the following.

\begin{cor}\label{ovlDgen}
The category $\ovl\calD$ is generated by the class of objects $V$ of $\ovl\calD$ for which there exists a tuple $\uu$ with $\length(\uu)<\dim_E(V)$ such that $\bS^\uu(V)\in\Ob(\calD)$.
\end{cor}

%

\begin{rem}\label{oldarg}
The proof of the ``if'' part of Proposition \ref{schur}(i) 
does not make use of Tannakian duality; therefore this statement holds even if the category $\calC$ is just an $E$-linear tensor category, and so does the following weaker version of (ii): if a tuple $\uu$ as in part (i) exists, then one has $\bS^\uv(V)\in\Ob(\calD)$ for every tuple $\uv$ such that $\ell(\uu)=\ell(\uv)$.
\end{rem}

\begin{rem}
We motivate the condition on the tuple $\uu$ appearing in Proposition \ref{schur}. When $\length(\uu)>n$ the object $\bS^\uu(V)$ is zero, so it cannot possibly give information on $V$. When $\length(\uu)=n$:
\begin{itemize}
\item If all the entries of $\uu$ are given by the same integer, $\bS^\uu(V)$ is a power of $\det(V)$, which belongs to $\Ob(\calD)$ for all $V$ by Remark \ref{1dim}.
\item Otherwise, $\uu$ can be written as $(\uv,0)+(k,\ldots,k)$ for a tuple $\uv$ with $\length(\uv)<n$ and an integer $k>0$, so that $\bS^\uu(V)=\bS^\uv(V)\otimes\det(V)^k$. Since $\ovl\calD$ contains all $1$-dimensional objects of $\calC$ by Remark \ref{1dim}, we have $\bS^\uu(V)\in\Ob(\calD)$ if and only if $\bS^\uv(V)\in\Ob(\calD)$. Therefore the restriction $\length(\uu)<n$ is irrelevant in this case.
\end{itemize}
Note that our assumption on $\uu$ is very similar to that on the partition in \cite[Sections 2.4, 3.3]{dimatadm}, the difference being that we also remove the case where $\length(\uu)=n$ but not all entries are equal; by our second comment above this allows us to simplify the assumption without the results losing strength.
\end{rem}



\subsection{Simple connectedness of fundamental groups}\label{subsecsc}

As in Section \ref{tpsub}, let $\calD$ be a Tannakian subcategory of a Tannakian category $\calC$, $\ovl\calD$ the intermediate category constructed from it, and $I$ be the kernel of the dual morphisms $G_{\ovl\calD}\to G_\calD$. 
We relate the triviality of the kernel $I$ to the simple connectedness of the ``semisimplified'' Tannakian fundamental group of $\calD$. 

If $V$ is an object of $\ovl\calD$ and $\rho\colon G_{\ovl\calD}\to\GL(V)$ its dual representation, we denote by $G_{\ovl\calD,V}$ and $I_V$ the schematic images under $\rho$ of $G_{\ovl\calD}$ and $I$, respectively, and $G_{\calD,V}$ for the quotient $G_{\ovl\calD,V}/I(V)$. 
Given an algebraic group $G$, we denote by $G^\circ$ its neutral connected component. 

For some of our results we will need to assume that $\calD$ satisfies the following condition:
\begin{equation}\tag{pot}\label{potcond}
	\text{The morphism $G_\calC\onto G_\calD$ induces an isomorphism on groups of connected components.}
	\end{equation}

The following result provides some context for condition \eqref{potcond}. Let $\Rep_E(G)$ be the category of finite-dimensional $E$-linear representations of a profinite group $G$. For every open subgroup $H$ of $G$, let $\mathbb P_H$ be a property of the restrictions $V\vert_H$ of the objects in $\Rep_E(G)$, such that 
\begin{itemize}
\item the full subcategory of $\Rep_E(G)$ whose objects are the $V$ such that $V\vert_H$ has $\mathbb P_H$ is a sub-tensor category;
\item if $H^\prime\subset H$ are two open subgroups of $G$ and $V$ an object of $\Rep_E(G)$ such that $V\vert_H$ has $\mathbb P_H$, then $V\vert_{H^\prime}$ has $\mathbb P_{H^\prime}$.
\end{itemize}
Let $\Rep_E^\bP(G)$ be the full subcategory of $\Rep_E(G)$ whose objects are the $V$ for which there exists an open subgroup $H$ of $G$ such that $V\vert_H$ has $\bP_H$. 

\begin{lemma}
The category $\calD=\Rep_E^\bP(G)$ satisfies condition \eqref{potcond} with $\calC=\Rep_E(G)$.
	\end{lemma}

As an example, one can take $G$ to be the absolute Galois group of a $p$-adic field, and $\bP_H$ be any of \{semistable, crystalline, trianguline\} over $H$. Then $\Rep_E^\bP(G)$ is the usual category of potentially \{semistable, crystalline, trianguline\} representations. More generally, one can take $\bP_H$ to be admissibility with respect to an $(E,H)$-regular ring in the sense of Fontaine.

\begin{proof}
Let $G_\calC$ be the Tannakian fundamental group of $\Rep_E(G)$. Let $\rho\colon G\to\GL(V)$ be any object of $\Rep_E(G)$. The image $G_{\calC,V}$ of $G_\calC$ in $\GL(V)$, under the representation dual to $V$, is the Zariski-closure of the image of $G$ under $\rho$. 
Let $J$ be the kernel of $G_\calC\onto G_\calD$, let $J_V$ be its image under $\rho$, and let $J_{V,0}$ be the intersection of $J_V$ with the neutral connected component of $G_{\calC,V}$. In order to show that condition \eqref{potcond} holds, it will be enough to show that $J_V=J_{V,0}$ and then pass to the limit. 

Pick a faithful, finite-dimensional $E$-linear representation $W$ of the quotient $G_{\calC,V}/J_{V,0}$, and consider it as a representation of $G_{\calC,V}$. The image of $J$ in $\GL(W)$ is the quotient $J_V/J{V,0}$ and can be identified with a subgroup of the group of connected components of the image $G_{\calC,W}$ of $G_{\calC,V}/J_{V,0}$ in $\GL(W)$. Since we quotiented by $J_{V,0}$, the map $G_\calC^\circ\to G_{\calC,W}^\circ$ of neutral connected components factors through $G_{\calC}^\circ\to G_\calD^\circ$. Since $G_{\calC,W}$ is the Zariski closure of the image of $G$ in $\GL(W)$, there exists an open subgroup $H$ of $G$ such that the image of $H$ in $\GL(W)$ is contained in the neutral connected component $G_{\calC,W}^\circ$, which admits a faithfully flat morphism from $G_\calD^\circ$. This means that the restriction of $W$ to $H$ belongs to $\Rep_E^{\bP_H}(H)$, and from the definition of $\bP$ we obtain that $W$ belongs to $\calD=\Rep_E^\bP(G)$, so that $J_V/J_{V,0}=0$.
	\end{proof}

Given an algebraic group $G$, write $G^\circ$ for the connected component of unity, $G^\red$ for the quotient of $G^\circ$ by its unipotent radical, and $G^\sms$ for the derived subgroup of $G^\red$. Clearly $G^\circ$ is connected, $G^\red$ is connected reductive, and $G^\sms$ is connected and semi-simple. 

We say that a pro-algebraic group over $E$ is connected (respectively semisimple, simply connected) if it is a projective limit of connected (respectively semisimple, simply connected) algebraic groups over $E$. If I is a small category and $G=\lim_{i\in I} G_i$ in the category of pro-algebraic $E$-groups, with the $G_i$ algebraic, then we define a connected pro-algebraic group $G^\circ=\lim_{i\in I} G_i^\circ$ (the neutral connected component of $G$), a connected pro-reductive group $G^\red=\lim_{i\in I} G_i^\red$, and a pro-semisimple group $G^\sms=\lim_{i\in I} G_i^\sms$. 
We say that $G^\sms$ is simply connected if it can be written as a projective limit of simply connected algebraic groups. 

We say that a morphism $G\to H$ of pro-algebraic groups over $E$ is a central isogeny if it is surjective and its kernel is finite and central in $G$. If $G$ is a connected and pro-semisimple pro-algebraic group, it can be written as a limit $\lim_{i\in I} G_i$ of connected semisimple algebraic groups over $E$. For every $i\in I$, let $G_i^\scn$ be the universal cover of $G_i$. The transition maps between the $G_i$ induce transition maps between the $G_i^\scn$. We let $G^\scn$ be the limit $\lim_{i\in I} G_i^\scn$ and we call it the \emph{universal cover}of $G$. It comes equipped with a morphism to $G$ and has the property that every central isogeny from a connected pro-semisimple group to $G$ is an isomorphism. The isomorphism class of $G^\scn$ as a group over $G$ is independent of the choice of $I$ and of the groups $G_i$.

\begin{rem}\label{smscon}
	A morphism $f\colon G\to H$ of pro-algebraic groups over $E$ induces a morphism 
	\[ f^\sms\colon G^\sms\to H^\sms, \]
	as follows. 
	First restrict $f$ to the neutral connected component $G^\circ$, whose image must be contained in $H^\circ$ giving a morphism
	\[ f^\circ\colon G^\circ\to H^\circ. \]
	If $U_G$ and $U_H$ are the pro-unipotent radicals of $G^\circ$ and $H^\circ$, respectively, then the composition of $f^\circ$ with $H^\circ\onto H^\circ/U_H$ factors through the quotient $G^\circ/U_G$ and a morphism
	\[ f^\red\colon G^\red\to H^\red \]
	of reductive groups. Finally, the restriction of $f^\red$ to the derived subgroup $G^\sms$ of $G^\red$ lands inside of the derived subgroup $H^\sms$ of $H^\red$, giving a morphisms
	\[ f^\sms\colon G^\sms\to H^\sms. \]
	If $f$ is a central isogeny, then a series of simple checks shows that $f^\sms$ is also a central isogeny.
\end{rem}

We prove that the group $G_{\ovl\calD}$ is a kind of ``universal cover of $G_\calD$ inside of $\calC$''. 
	For any surjection $G_1\onto G_2$ of pro-algebraic groups over $E$, consider the category $\calH(G_1,G_2)$ of triples $(H,f,g)$ fitting into a diagram
		\begin{equation}\label{Hfg} G_1\xonto{f} H\xonto{g} G_{2},
	\end{equation}
of pro-algebraic groups over $E$, with morphisms from $(H,f,g)$ to another object $(H^\prime,f^\prime,g^\prime)$ in $\calH(G_1,G_2)$ being the morphisms of pro-algebraic groups $H\to H^\prime$ that make the diagram
\begin{center}
	\begin{tikzcd}
		G_1\arrow{d}{=} \arrow{r}{f} & H\arrow{r}{g}\arrow{d} & G_2\arrow{d}{=} \\
				G_1 \arrow{r}{f^\prime} & H^\prime\arrow{r}{g^\prime} & G_2 \\
		\end{tikzcd}
	\end{center}
commute. Consider the full subcategory $\calH^\ci(G_1,G_2) of \calH(G_1,G_2)$ consisting of the triples $(H,f,g)$ such that the kernel of $g$ is a finite central subgroup of $H^\circ$ (in other words, the restriction $g\colon H^\circ\to G_\calD^\circ$ is a central isogeny). Write $\iota(G_1,G_2)$ for the inclusion functor $\calH^\ci(G_1,G_2)\into\calH(G_1,G_2)$. 

Let $\pi_\calC^{\ovl\calD}\colon G_\calC\to G_{\ovl\calD}$, $\pi_\calC^{\calD}\colon G_{\ovl\calD}\to G_\calD$ and $\pi_\calC^{\calD}=\pi_\calC^\calD\ccirc\pi_\calC^{\ovl\calD}$ be the usual surjections.

\begin{prop}\label{ovlDuni}
Assume that condition \eqref{potcond} is satisfied. Then the triple $(G_{\ovl\calD},\pi_\calC^{\ovl\calD},\pi_\calC^{\calD})$ is the limit of the diagram $\iota(G_\calC,G_\calD)\colon\calH^\ci(G_\calC,G_\calD)\into\calH(G_\calC,G_\calD)$. In particular $G_{\ovl\calD}$ induces an isomorphism on the groups of connected components and on the pro-unipotent radicals.
	\end{prop}

In order to prove Proposition \ref{ovlDuni} we rely on the following lemma. We use, as usual, the notations of Remark \ref{smscon}.

\begin{lemma}\label{limequiv}
	The following statements are equivalent:
	\begin{enumerate}[label=(\roman*)]
	\item $(G_{\ovl\calD},\pi_\calC^{\ovl\calD},\pi_\calC^{\calD})$ is the limit of $\iota(G_\calC,G_\calD)\colon\calH^\ci(G_\calC,G_\calD)\into\calH(G_\calC,G_\calD)$;
	\item $(G_{\ovl\calD}^\circ,\pi_\calC^{\ovl\calD,\circ},\pi_\calC^{\calD,\circ})$ is the limit of $\iota(G_\calC^\circ,G_\calD^\circ)\colon  \calH^\ci(G_\calC^\circ,G_\calD^\circ)\into\calH(G_\calC^\circ,G_\calD^\circ)$;
	\item $(G_{\ovl\calD}^\red,\pi_\calC^{\ovl\calD,\red},\pi_\calC^{\calD,\red})$ is the limit of $\iota(G_\calC^\red,G_\calD^\red)\colon  \calH^\ci(G_\calC^\red,G_\calD^\red)\into\calH(G_\calC^\red,G_\calD^\red)$.
	\end{enumerate}
		Moreover, if any of (i), (ii), (iii) holds, then
			\begin{itemize}
\item[(iv)] $(G_{\ovl\calD}^\sms,\pi_\calC^{\ovl\calD,\sms},\pi_\calC^{\calD,\sms})$ is the limit of $\iota(G_\calC^\sms,G_\calD^\sms)\colon  \calH^\ci(G_\calC^\sms,G_\calD^\sms)\into\calH(G_\calC^\sms,G_\calD^\sms)$.
\end{itemize}
\end{lemma}

\begin{proof}
	In order to prove the equivalence of (i), (ii), (iii) we rely on the following simple remark: 
	\begin{itemize}
\item[($\ast$)] if $\calA$, $\calB$ are two categories, $\calA_0\subset\calA$, $\calB_0\subset\calB$ two small subcategories, and $F\colon\calA\to\calB, G\colon\calB\to\calA$ two quasi-inverse functors that induce an equivalence of categories $\calA_0\cong\calB_0$, then an object $L$ of $\calA$ is the limit of $\calA_0\into\calA$ if and only if $F(L)$ is the limit of $\calB_0\into\calB$.
\end{itemize}
		
	We prove that (i) and (ii) are equivalent by applying ($\ast$) to $\calA=\calH(G_\calC,G_\calD)$, $\calB=\calH(G_\calC^\circ,G_\calD^\circ)$ and $\calA_0$, $\calB_0$ the subcategories of ``central isogenies''. We construct the two quasi-inverse functors $F_\varnothing^\circ$ an $F_\circ^\varnothing$ that we need. 
	If $(H,f,g)$ is an object of $\calH(G_\calC,G_\calD)$, then, with the notations of Remark \ref{smscon}, the diagram
	\begin{equation}\label{funcirc}
		\begin{tikzcd}
			\calC\arrow[two heads]{r}{f} & H\arrow[two heads]{r}{g} & G_\calD \\
						\calC^\circ\arrow[two heads]{r}{f^\circ}\arrow[hookrightarrow]{u} & H^\circ\arrow[two heads]{r}{g^\circ}\arrow[hookrightarrow]{u} & G_\calD^\circ\arrow[hookrightarrow]{u} \\
		\end{tikzcd}
	\end{equation}
commutes, and moreover all squares are cartesian because $G_\calC\onto G_\calD$ induces an isomorphism on connected components by condition \eqref{potcond}. Hence $(H,f,g)\mapsto (H^\circ,f^\circ,g^\circ)$ defines a functor 
\[ F_\varnothing^\circ\colon\calH(G_\calC,G_\calD)\to\calH(G_\calC^\circ,G_\calD^\circ). \]
If the kernel of $g$ is a finite central subgroup of $G^\circ$, then the same is true for the kernel of $g^\circ$, so that $F_\varnothing^\circ$ can be restricted to a functor $\calH^\ci(G_\calC,G_\calD)\to\calH^\ci(G_\calC^\circ,G_\calD^\circ)$.

Vice versa, if we start with a triple $(H_0,f_0,g_0)$ in $\calH(G_\calC^\circ,G_\calD^\circ)$, we can construct a triple $(H,f,g)$ in $\calH(G_\calC,G_\calD)$ satisfying $H^\circ=H_0, f^\circ=f_0, g^\circ=g_0$ by taking pushouts of the second row of \eqref{funcirc} along $G_\calD^\circ\into G_\calD$. This provides us with a functor 
\[ F_\circ^\varnothing\colon\calH(G_\calC^\circ,G_\calD^\circ)\to\calH(G_\calC,G_\calD) \]
that is quasi-inverse to $F_\varnothing^\circ$. The injection $H_0\into H$ restricts to an isomorphism between the kernels of $g$ and $g^\circ$, hence $F_\circ^\varnothing$ restricts to a functor $\calH^\ci(G_\calC,G_\calD)\to\calH^\ci(G_\calC^\circ,G_\calD^\circ)$, as desired.

We prove the equivalence between (ii) and (iii) by applying ($\ast$) again, after constructing a pair of quasi-inverse functors $F_\circ^\red, F_\red^\circ$. 
Let $(H,f,g)$ be an object of $\calH^\ci(G_\calC^\circ,G_\calD^\circ)$. By quotienting out unipotent radicals, we obtain an object $(H^\red,f^\red,g^\red)$ as in the second row of the commutative diagram
	\begin{center}
	\begin{tikzcd}
G_\calC^\circ\arrow{r}{f}\arrow[two heads]{d} \arrow[two heads]{r} & H\arrow[two heads]{r}{g}\arrow[two heads]{d} & G_\calD^\circ\arrow[two heads]{d} \\
G_\calC^\red\arrow[two heads]{r}{f^\red} & H^\red\arrow[two heads]{r}{g^\red} & G_\calD^\red
		\end{tikzcd}
\end{center}
If $g$ is a central isogeny then the quotient map $H^\circ\onto H^\red$ induces an isomorphism between $\ker(g^\circ)$ and $\ker(g^\red)$, so that $g^\circ$ is also a central isogeny. Then $(H,f,g)\mapsto (H^\red,f^\red,g^\red)$ defines the required functor 
\[ F_\circ^\red\colon\calH(G_\calC^\circ,G_\calD^\circ)\to\calH(G_\calC^\red,G_\calD^\red). \]

Vice versa, starting from a triple $(H_0,f_0,g_0)$ in $\calH(G_\calC^\red,G_\calD^\red)$, we construct a triple $(H,f,g)$ in $\calH(G_\calC^\circ,G_\calD^\circ)$ that satisfies $H^\red=H_0$, $f^\red=f_0$ and $g^\red=g_0$: Define $H$ as the pullback of $G_\calC^\red\onto G_\calD^\red \twoheadleftarrow H_0$ and $g$ as the map it comes equipped with. Since the kernel of $G_\calD^\circ\onto G_\calD^\red$ is pro-unipotent, the same is true of the kernel of $H\onto H_0$. In particular $H^\red=H_0$ and $g^\red=g_0$. 
The group $G_\calC^\circ$ admits compatible maps to $H$ and $G_\calD^\circ$, hence a map $f$ to the pullback $H$. Commutativity of all the diagrams involved gives $f^\red=f_0$. The projection $H\onto H_0$ induces an isomorphism between $\ker(g)$ and $\ker(g_0)$, so that if $g_0$ is a central isogeny then $g$ is also one. Therefore we obtain a functor 
\[ F_\red^\circ\calH(G_\calC^\red,G_\calD^\red)\to\calH(G_\calC^\circ,G_\calD^\circ) \]
that is quasi-inverse to $F_\circ^\red$ and has the desired properties.

We conclude the proof by showing that (iii) implies (iv). 
We construct two functors
\[ F_\red^\sms\colon\calH(G_\calC^\red,G_\calD^\red)\to\calH(G_\calC^\sms,G_\calD^\sms) \]
and 
\[ F_\sms^\red\colon\calH(G_\calC^\sms,G_\calD^\sms)\to\calH(G_\calC^\red,G_\calD^\red) \]
such that 
\begin{enumerate}[label=(\arabic*)]
\item $F_\sms^\red\circ F_\red^\sms$ is naturally isomorphic to the identity functor on $\calH(G_\calC^\sms,G_\calD^\sms)$,
\item $F_\red^\sms(G_\calD^\red)=G_\calD^\sms$, 
\item the essential image of the restriction of $F_\red^\sms$ to $\calH^\ci(G_\calC^\red,G_\calD^\red)$ is contained in $\calH^\ci(G_\calC^\sms,G_\calD^\sms)$, and
\item every object in the essential image of the restriction of $F_\sms^\red$ to $\calH^\ci(G_\calC^\sms,G_\calD^\sms)$ is a limit of a diagram in $\calH^\ci(G_\calC^\red,G_\calD^\red)$.
\end{enumerate}
This will be enough to prove the desired statement: if $L_\sms$ and $L_\red$ are the limits of $\calH^\ci(G_\calC^\sms,G_\calD^\sms)$ and $\calH^\ci(G_\calC^\red,G_\calD^\red)$, respectively, then conditions (1) and (3) give us the existence of a morphism from $L_\sms$ to $F_\red^\sms(L^\red)$. Condition (4), on the other hand, gives us a morphism $L_\red\to F_\sms^\red(L_\sms)$, that is mapped by $F_\red^\sms$ to a morphism from $F_\red^\sms(L_\red)\to F_\red^\sms(F_\sms^\red(L_\sms))$, whose target is isomorphic to $L_\sms$ by (1). The universal property of the limit forces the two morphisms we constructed between $L_\sms$ and $F_\red^\sms(L^\red)$ to be isomorphisms, and combining this with (2) gives (iv).

In order to construct the functor $F_\red^\sms$, simply start with any $(H,f,g)\in\calH(G_\calC^\red,G_\calD^\red)$ and apply the construction of Remark \ref{smscon} to the first row of
	\begin{center}
	\begin{tikzcd}
		G_\calC^\red\arrow{r}{f} \arrow[two heads]{r} & H\arrow[two heads]{r}{g} & G_\calD^\red \\
		G_\calC^\sms\arrow[two heads]{r}{f^\sms}\arrow[hook]{u} & H^\sms\arrow[two heads]{r}{g^\sms}\arrow[hook]{u} & G_\calD^\sms\arrow[hook]{u}
	\end{tikzcd}
\end{center}
in order to obtain the second row, hence an object $(H^\sms,f^\sms,g^\sms)$ of $\calH(G_\calC^\sms,G_\calD^\sms)$. If $g$ is a central isogeny, then so is $g^\sms$, since the kernel of $g^\sms$ injects into that of $g$ via $H^\sms\into H$. 

Vice versa, pick an object $(H_0,f_0,g_0)$ of $\calH(G_\calC^\sms,G_\calD^\sms)$. We write $Z(G)$ for the center of a pro-reductive group $G$. Since $G_\calC^\red$ and $G_\calD^\red$ are pro-reductive, there are exact sequences
\begin{equation}\label{redstr}
	\begin{gathered}
		0\to Z(G_\calC^\red)\cap G_\calC^\sms\to Z(G_\calC^\red)\times G_\calC^\sms\xto{\pi_\calC}G_\calC^\red\to 0, \\
0\to Z(G_\calD^\red)\cap G_\calD^\sms\to Z(G_\calD^\red)\times G_\calD^\sms\xto{\pi_\calD}G_\calD^\red\to 0, 
\end{gathered}
\end{equation}
where the injection is the diagonal one. 
The morphism $\pi_\calC^{\calD,\red}\colon G_\calC^\red\onto G_\calD^\red$ restricts to a morphism $\pi_Z\colon Z(G_\calC^\red)\to Z(G_\calD^\red)$. 
Consider the morphisms
\begin{align*}
 f_1\colon Z(G_\calC^\red)\times G_\calC^\sms&\to Z(G_\calD^\red)\times H_0 \\
(z,h)&\mapsto (\pi_Z(z),f_0(h))
\end{align*}
	and
\begin{align*}
	\wtl g_1\colon Z(G_\calD^\red)\times H_0&\to Z(G_\calD^\red)\times G_\calD^\sms \\
(z,h)&\mapsto(z,g_0(h))
\end{align*}
Let $H_1$ be the image of $f_1$, and write $\wtl f_1$ for the map $Z(G_\calC^\red)\times G_\calC^\sms\to H_1$ induced by $f_1$, and $\wtl g_1$ for the restriction of $g_1$ to $H_1$. 
We fit these maps into a diagram
	\begin{center}
	\begin{tikzcd}
		Z(G_\calC^\red)\times G_\calC^\sms\arrow[two heads]{r}{\wtl f_1} \arrow[two heads]{d}{\pi_\calC} & Z(G_\calD^\red)\times H_0\arrow{r}{\wtl g_1}\arrow[two heads]{d}{\pi_H} & Z(G_\calD^\red)\times G_\calD^\sms\arrow[two heads]{d}{\pi_\calD} \\
		G_\calC^\red\arrow{r}{f} \arrow[two heads]{r} & H\arrow[two heads]{r}{g} & G_\calD^\red \\
		G_\calC^\sms\arrow[two heads]{r}{f_0}\arrow[hook]{u} & H_0\arrow[two heads]{r}{g_0}\arrow[hook]{u} & G_\calD^\sms\arrow[hook]{u}
	\end{tikzcd}
\end{center}
where 
\begin{itemize}
\item the maps $\pi_\calC$, $\pi_\calD$ come from \eqref{redstr},
\item $H$ and the maps $f$ and $\pi_H$ are defined as the pushout of the top left square,
\item the map $g$ comes from the universal property of the pushout, after checking that the maps $\pi_\calC^{\calD,\red}\ccirc\pi_\calC$ and $(\pi_\calD\ccirc g)\ccirc f$ coincide,
\item the map $H_0\to H$ is obtained as the composite of $H_0\into Z(G_\calD^\red)\times H_0, h\mapsto (1,h)$ with the projection $\pi_H$,
\item $f$ is surjective because $\wtl f_1$ is, and $g$ is surjective because the composite $\pi_\calC^{\calD,\red}=g\ccirc f$ is surjective.
\end{itemize}
We define $F_\sms^\red$ as the functor $(G_\calC^\sms,G_\calD^\sms)\to\calH(G_\calC^\red,G_\calD^\red$ that maps $(H_0,f_0,g_0)$ to the triple $(H,f,g)$ from the above diagram. 

By construction of $g_1$, the kernel of $g_1$ is $\{1\}\times\ker(g_0)$, and that of $\pi_\calD\ccirc g_1\colon Z(G_\calD^\red)\times H_0\to G_\calD^\red$ is an extension of $\ker(\pi_\calD)=Z(G_\calD^\red)\cap G_\calD^\sms$ by $\ker(g_1)$. If $g_0$ is an isogeny then $\ker(g_1)$ is finite. Since $\ker(\pi_\calD)$ is profinite, $\ker(\pi_\calD\ccirc g_1)$ is also profinite. The kernel of $\wtl g_1$ is a subgroup of that of $g_1$, and $\ker(g)$ is a quotient of $\ker(\wtl g_1)$. Therefore $\ker(g)$ is also profinite, and we can write g as a limit of isogenies onto $G_\calD^\red$. Since $G_\calC$ is connected, so is its quotient $H$ and so is any quotient of $H$ admitting an isogeny onto $G_\calD^\red$. Since the base field $E$ is of characteristic 0, every isogeny out of a connected pro-reductive $E$-group is central, hence $g\colon H\onto G_\calD^\sms$ is a limit of objects in $\calH^\ci(G_\calC^\red,G_\calD^\red)$. Therefore $F_\sms^\red$ has property (4). Conditions (1-3) are easily checked. 
\end{proof}

\begin{proof}[Proof of Proposition \ref{ovlDuni}]
	It is enough to prove statement (iii) in Lemma \ref{limequiv}.
	Consider a triple $(H,f,g)$ in $\calH^\ci(G_\calC^\red,G_\calD^\red)$. Observe that $H$ is necessarily reductive, being a quotient of $G_\calC^\red$. 
	Let $\rho_V\colon H\to\GL(V)$ be any irreducible representation of $H$, and let $H_V$ and $\ker(g)_V$ be the schematic images of $H$ and $\ker(g)$, respectively, in $\GL(V)$. Since $\ker(g)_V$ is a central subgroup of $H$ and $V$ is irreducible, by Schur's lemma $\ker(g)_V$ must be contained in the center of $\GL(V)$. Given that $\ker(g)_V$ is finite, it must be contained in $\mu_n$, acting on $V$ via scalar endomorphisms, for a sufficiently large $n$. Now pick any tuple $\uu$ with $\ell(\uu)=n$ and $\length(\uu)<n$. By Remark \ref{schurcomp}(i), the kernel of $\bS^\uu\colon\GL(V)\to\GL(\bS^\uu(V))$ is $\mu_n$, hence $\bS^\uu\ccirc\rho_V$ factors through $H_V/\ker(g)_V$. Since the morphism $H\to H_V/\ker(g)_V$ factors through $H\onto H/\ker(g)\cong G_\calD^\red$, the representation $V$, seen as an object of $\calC$ via $f\colon G_\calC^\red\onto H$, satisfies $\bS^\uu(V)\in\calD$. Thanks to Proposition \ref{schur}(i), we conclude that $V$ is an object of $\ovl\calD$, or in other words, that the representation 
\[ \rho_V\ccirc f\colon G_\calC^\red\to\GL(V) \]
factors through $G_\calC^\red\onto G_{\ovl\calD}^\red$. Since this holds for every irreducible representation $V$ of $H$ and every representation of the reductive group $H$  is semisimple, we conclude that $f$ itself factors as the composition of $G_\calC^\red\onto G_{\ovl\calD}^\red$ and a map $f_0G_{\ovl\calD}^\red\to H^\red$, providing us with a morphism
\begin{center}
	\begin{tikzcd}
		G_\calC^\red\arrow[two heads]{r}{\pi_\calC^{\ovl\calD,\red}}\arrow{d}{=} & G_{\ovl\calD}^\red\arrow[two heads]{r}{\pi_\calC^{\calD,\red}}\arrow{d}{f_0} & G_\calD^\red\arrow{d}{=} \\
		G_\calC^\red\arrow[two heads]{r}{f^\red} & H^\red\arrow[two heads]{r}{g^\red} & G_\calD^\red
		\end{tikzcd}
	\end{center}
from $(G_{\ovl\calD}^\red,\pi_\calC^{\ovl\calD,\red},\pi_\calC^{\calD,\red})$ to $(H,f,g)$. 

Since $H$ was chosen arbitrarily, we obtain a morphism from $(G_{\ovl\calD}^\red,\pi_\calC^{\ovl\calD,\red},\pi_\calC^{\calD,\red})$ to the limit of $\iota\colon\calH^\ci(G_\calC^\red,G_\calD^\red)\into\calH(G_\calC^\red,G_\calD^\red)$. In order to prove that it is an isomorphism, it is sufficient to write $(G_{\ovl\calD}^\red,\pi_\calC^{\ovl\calD,\red},\pi_\calC^{\calD,\red})$ as a limit of some subdiagram of $\iota(G_\calC^\red,G_\calD^\red)$. For this, consider a finite-dimensional $E$-linear representation $V$ of $G_{\ovl\calD}^\red$, and let $G_{\ovl\calD,V}^\red$ and $I_V$ be the images of $G_{\ovl\calD}^\red$ and $I$, respectively, in $\GL(V)$. Here $I$ denotes, as usual, the kernel of $\pi_\calC^{\ovl\calD,\red}\colon G_{\ovl\calD}^\red\to G_\calD^\red$. Since $I_V$ is finite and central by Lemma \ref{central} and Corollary \ref{Iprof}, the surjection $\pi_V\colon G_{\ovl\calD,V}^\red\to G_{\ovl\calD,V}^\red/I_V$ is a central isogeny. Pulling back $\pi_V$ along $G_\calD^\red\to G_{\ovl\calD,V}^\red/I_V$, we obtain a diagram
\begin{center}
	\begin{tikzcd}
		G_\calC^\red\arrow[two heads]{r}{\pi_\calC^{\ovl\calD,\red}}\arrow[two heads]{d} & G_{\ovl\calD}^\red\arrow[two heads]{r}{f} & H_V \arrow[two heads]{r}{\wtl\pi_V}\arrow[two heads]{d} & G_\calD^\red\arrow[two heads]{d}{=} \\
		G_{\calC,V}^\red\arrow[two heads]{rr}{\pi_\calC^{\ovl\calD,\red}} & & G_{\ovl\calD,V}^\red\arrow[two heads]{r}{\pi_V} & G_{\ovl\calD,V}^\red/I_V \\
	\end{tikzcd}
\end{center}
where $H_V=G_\calD^\red\times_{G_{\ovl\calD,V}^\red/I_V}G_{\ovl\calD,V}^\red$ and the first line gives an object $(H_V,f_V\ccirc\pi_\calC^{\ovl\calD,\red},\wtl\pi_V)$ of $\calH^\ci(G_\calC^\red,G_\calD^\red)$. From $G_{\ovl\calD}^\red=\lim_{V\in\Rep_E(G_{\ovl\calD}^\red)}G_{\ovl\calD,V}^\red$, we deduce that $(G_{\ovl\calD}^\red,\pi_\calC^{\ovl\calD,\red},\pi_\calC^{\calD,\red})$ is the limit of the full subcategory of $\calH^\ci(G_\calC^\red,G_\calD^\red)$ consisting of the triples of the form $(H_V,f_V\ccirc\pi_\calC^{\ovl\calD,\red},\wtl\pi_V)$ for some $V\in\Rep_E(G_{\ovl\calD}^\red)$.
	\end{proof}

\begin{cor}\label{ovlDsc}
Assume that $G_\calC^\sms$ is simply connected and that condition \eqref{potcond} holds. Then $G_{\ovl\calD}^\sms$ is the universal cover of $G_\calD^\sms$. 
In particular:
\begin{enumerate}
\item for every object $V$ of $\ovl\calD$, $G_{\calD,V}^\sms$ is the universal cover of $G_{\calD,V}^\sms$;
\item if $\ovl\calD=\calD$, then $G_{\ovl\calD}^\sms$ is simply connected;
\item if $G_\calC=G_\calC^\sms$, then $\calD=\ovl\calD$ if and only if $G_\calD$ is simply connected.
\end{enumerate}
\end{cor}

\begin{proof}
This follows immediately from Proposition \ref{ovlDuni} and the equivalence between (i) and (iv) in Lemma \ref{limequiv}. Indeed, if $G_\calC^\sms$ is simply connected, for every central isogeny $g\colon H\to G_\calD^\sms$, with $H$ connected, there exists a surjection $f\colon G_\calC^\sms\to H$ such that $g\ccirc f=\pi_\calC^\calD\colon G_\calC^\sms\onto G_\calD^\sms$. In particular every such $g$ defines an object $(H,f,g)$ of $\calH^\ci(G_\calC,G_\calD)$, so that the limit of $\calH^\ci(G_\calC,G_\calD)\to\calH(G_\calC,G_\calD)$ is also the limit of all $g$, that is, $G_{\ovl\calD}^\sms\onto G_\calD^\sms$ is the universal cover of $G_\calD^\sms$.
\end{proof}

\begin{rem}
	Assume that the hypotheses of Corollary \ref{ovlDsc} are satisfied. 
	Similarly to \cite[\S 3.1, Remarque]{wintentann}, we observe that the reverse implication to (ii) does not hold: If $V$ is an object of $\calD$, then by "" there exists an object $W$ of $\calD$ such that $G_{\ovl\calD,W}^\sms$ is the universal cover of $G_{\calD,V}^\sms$. In particular, if $G_{\ovl\calD}^\sms$ is simply connected, then 
	\[ G_{\ovl\calD}^\sms\to G_{\calD,V}^\sms \]
	factors through $G_{\ovl\calD,W}^\sms$. However, if the kernel of the central isogeny 
	\begin{equation}\label{WtoV} G_{\ovl\calD,W}\to G_{\calD,V}
		\end{equation}
	is not contained in $G_{\ovl\calD,W}^\sms$, then \eqref{WtoV} does not factor through $G_{\calD,W}$, so that $W$ is an object of $\ovl\calD$ but not of $\calD$.
	\end{rem}

\medskip

\section{Application to categories of \texorpdfstring{$B$}{B}-pairs}\label{galois}

We recall some definitions from the theory of $B$-pairs, as one can find for instance in \cite{berconstr}. Let $K$ be a $p$-adic field, and let $\bB$ be a topological ring equipped with a continuous action of $G_K$. We call \emph{semilinear $\bB$-representation of $G_K$}, or in short \emph{$\bB$-representation of $G_K$}, a free $\bB$-module $M$ of finite rank equipped with a semilinear action of $G_K$, that is, such that $g(bm)=g(b)g(m)$ for every $b\in\bB$, $m\in M$ and $g\in G_K$. We denote by $\bB\Rep(G_K)$ the category whose objects are the semilinear $\bB$-representations of $G_K$ and whose morphisms are the $G_K$-equivariant morphisms of $\bB$-modules. We call \emph{rank} of an object of $\bB\Rep(G_K)$ its rank as a $\bB$-module. We say that a $\bB$-representation $M$ of $G_K$ is \emph{trivial} if $M$ admits a $\bB$-basis consisting of $G_K$-invariant elements. We call \emph{eigenvector} in a semilinear $\bB$-representation $M$ a vector that \emph{belongs} to a $G_K$-stable $\bB$-line in $M$ (recall that by a line in a free $\bB$-module we mean a free rank 1 submodule). An eigenvector does not necessarily generate a $G_K$-stable line as a $\bB$-module; some of its eigenvalues may belong to the total fraction ring of $\bB$.

When $\bB$ has a structure of $E$-algebra with respect to which the action of $G_K$ is $E$-linear, and $\eta$ is an $E$-valued character of $G_K$, we write $\bB(\eta)$ for the rank $1$ $\bB$-representation $\bB\otimes_EE(\eta)$, where $G_K$ acts diagonally. 

Let $\bB$ be an $(E,G_K)$-regular ring in the sense of \cite[Definition 2.8]{fonouy}; it is in particular a topological $E$-algebra equipped with a continuous action of $G_K$. Let $V$ be an $E$-linear representation of $G_K$. We define a $\bB$-semilinear representation of $G_K$ by letting $G_K$ act diagonally on $\bB\otimes_EV$. 
We say that $V$ is \emph{$\bB$-admissible} if the $\bB$-semilinear representation $\bB\otimes_EV$ is trivial.

We use the standard notation for Fontaine's rings of periods $\bB_\HT, \bB_\dR, \bB_\dR^+, \bB_\cris, \bB_\st$, as defined in \cite{fontaine}. Each of these objects is a $(\Q_p,G_K)$-regular ring. We denote by $\varphi$ the Frobenius endomorphism of both $\bB_\cris$ and $\bB_\st$, and follow the standard notation again in setting $\bB_e=\bB_\cris^{\varphi=1}$. We write $t$ for Fontaine's choice of a generator of the maximal ideal of the complete discrete valuation ring $\bB_\dR^+$.

Let $E$ be a $p$-adic field. We set $\bB_{?,E}=\bB_?\otimes_{\Q_p}E$ for $?\in\{\HT,\dR,\st,\crys,e\}$, and also $\bB_{\dR,E}^+=\bB_\dR^+\otimes_{\Q_p}E$. Each of these rings is a topological $E$-algebra, that we equip with the continuous action of $G_K$ obtained by extending \emph{$E$-linearly} the action of $G_K$ on the original $\Q_p$-algebra.

\begin{defin}
A \emph{$\BPKE$-pair} is a pair $(W_e,W_\dR^+)$ where:
\begin{itemize}
\item $W_e$ is an object of $\bB_{e,E}\Rep(G_K)$;
\item $W_\dR^+$ is a $G_K$-stable $\bB_{\dR,E}^+$-lattice of $\bB_{\dR,E}\otimes_{\bB_{e,E}}W_e$. 
\end{itemize}
We write $W_\dR$ for the $\bB_\dR$-representation $\bB_{\dR,E}\otimes_{\bB_{e,E}}W_e$.
We define the \emph{rank} of $(W_e,W_\dR^+)$ as the common rank of $W_e$ and $W_\dR^+$. 

Given two $\BPKE$-pairs $(W_e,W_\dR^+)$ and $(W_e^\prime,W_\dR^{+,\prime})$, a \emph{morphism of $\BPKE$-pairs} $(W_e,W_\dR^+)\to (W_e^\prime,W_\dR^{+,\prime})$ is a pair $(f_e,f_\dR^+)$ where:
\begin{itemize}
\item $f_e\colon W_e\to W_e^\prime$ is a morphism in $\bB_{e,E}\Rep(G_K)$,
\item $f_\dR^+$ is a morphism in $\bB_\dR^+\Rep(G_K)$,
\item the two morphisms $W_\dR\to W_\dR^\prime$ in $\bB_\dR\Rep(G_K)$ obtained by extending $\bB_\dR$-linearly $f_e$ and $f_\dR^+$ coincide.
\end{itemize}
\end{defin}

Given two $\BPKE$-pairs $W=(W_e,W_\dR^+)$ and $X=(X_e,X_\dR^+)$, we say that $X$ is a \emph{modification} of $W$ if $X_e\cong W_e$ \cite[Définition 2.1.8]{berconstr}. If $X_e\subset W_e$ and $X_\dR^+\subset W_\dR^+$, then we say that $X$ is a sub-$\BPKE$-pair of $W$. We say that such an $X$ is a \emph{saturated} sub-$\BPKE$-pair of $W$ if the lattice $X_\dR^+$ is saturated in $W_\dR^+$, that is, if $X_\dR^+=X_\dR\cap W_\dR^+$. The quotient of$W$ by a sub-$\BPKE$-pair $X$ admits a natural structure of $\BPKE$-pair if and only if $X$ is saturated in $W$. Given a sub-$\BPKE$-pair $X$ of $W$, we can always find a unique saturated modification of it in $W$ by replacing $X_\dR^+$ with $X_\dR\cap W_\dR^+$; we will call this modification the \emph{saturation} of $X$ in $W$.

Berger proved that the category of $\BPKE$-pairs is equivalent to that of $(\varphi,\Gamma_K)$-modules over the Robba ring over $E$. This allows one to transport the theory of slopes from $\varphi$-modules to $\BPKE$-pairs, and in particular to speak of pure (or isoclinic) $\BPKE$-pairs and of Dieudonné--Manin filtrations for $\BPKE$-pairs. We refer to \cite{kedmon} for the relevant definitions.

Given a $\BPKE$ pair $W$ and finite extensions $L/K$ and $F/E$, we can define a $B\vert_L^{\otimes F}$-pair as $(F\otimes_EW)\vert_{G_L}$, with the obvious notations. Given a property $\mathbb P$ of a linear or semilinear representation of $G_K$, or of a $\BPKE$-pair, we say that one such object $W$ has $\mathbb P$ \emph{potentially} if there is a finite extension $L/K$ such that $W\vert_{G_L}$ has $\mathbb P$.

We denote by $\Rep_E(G_K)$ the category of continuous, $E$-linear, finite-dimensional representation $V$ of $G_K$. For an object $V$ of $\Rep_E(G_K)$ we denote by $W(V)$ the $\BPKE$-pair $(\bB_{e,E}\otimes_EV,\bB_{\dR,E}^+\otimes_EV)$. The rank of $W(V)$ is equal to the $E$-rank of $V$. Given two objects $V$, $V^\prime$ of $\Rep_{E}(G_K)$ and a morphism $f\colon V\to V^\prime$, we define a morphism $W(f)\colon W(V)\to W(V^\prime)$ by $\bB_{e,E}$-linearly extending $f$ to the first element of $W(V)$ and $\bB_{\dR,E}^+$ linearly to the second. The functor $W(\cdot)$ defined this way is fully faithful and identifies $\Rep_{E}(G_K)$ with the full tensor subcategory of the category of $\BPKE$-pairs whose objects are the pure $\BPKE$-pairs of slope $0$). This is \cite[Théorème 3.2.3]{berconstr} when $E=\Q_p$ and an immediate consequence of it for general $E$. 

\begin{defin}\label{Bpadm}
A $\bB_{e,E}$-representation $W_e$ is crystalline, semistable, or de Rham if $\bB_{?,E}\otimes_{\bB_{e,E}}W_e$ is trivial for $?=\cris$, $\st$, or $\dR$, respectively. A $B_\dR^+$-representation is Hodge--Tate if $\bB_\HT\otimes_{\C_p}(W_\dR^+/tW_\dR^+)$ is trivial. 

A $\BPKE$-pair $(W_e,W_\dR^+)$ is crystalline, semistable, or de Rham if $W_e$ is crystalline, semistable, or de Rham, respectively. It is Hodge--Tate if $W_\dR^+$ is Hodge--Tate.

An $E$-linear representation $V$ of $G_K$ is crystalline, semistable, Hodge--Tate or de Rham if $\bB_{?,E}\otimes_EV$ is trivial for $?=\cris,\st,\HT$ or $\dR$, respectively. 
\end{defin}

It is clear that an $E$-linear representation of $G_K$ is crystalline, semistable, Hodge--Tate or de Rham if and only if the associated $\BPKE$-pair has the same property. 
Recall that the properties of being de Rham and potentially semistable are equivalent for a $\BPKE$-pair by the $p$-adic monodromy theorem \cite[Théorème 2.3.5]{berconstr}. 

To a continuous character $\eta\colon K^\times\to E^\times$, Nakamura attaches a $\BPKE$-pair 
\[ W(\eta)=(\bB_{e,E}(\eta),\bB_{\dR,E}(\eta)), \]
and proves that every $\BPKE$-pair of rank $1$ is isomorphic to $W(\eta)$ for some $\eta$ \cite[Theorem 1.45]{nakclass}. Via Berger's equivalence between the categories of $\BPKE$-pairs and $(\varphi,\Gamma_{K})$-modules over the Robba ring over $E$, Nakamura's classification is a natural generalization to arbitrary coefficients of that given by Colmez in the case $K=\Q_p$ \cite[Proposition 3.1]{colmez}. Note that the $\BPKE$-pair $W(\eta)$ is of slope $0$ if and only if the character $\eta$ can be extended to a Galois character $G_K\cong\widehat{K^\times}\to E^\times$, where the first isomorphism is given by the reciprocity map of local class field theory. In such a case, $W(\eta)$ is simply $(\bB_{e,E}\otimes_EE(\eta),\bB_{\dR,E}\otimes_EE(\eta))$. In particular, this notation is compatible with the notation $\bB(\eta)$ introduced in the beginning of the section. 

We introduce the standard terminology for $\BPKE$-pairs that can be obtained via successive extensions of $\BPKE$-pairs of rank $1$.


\begin{defin}\label{Bptri}
A $\BPKE$-pair $W$ is \emph{split triangulable} if there exists a filtration 
\[ 0=W_0\subset W_1\subset\ldots\subset W_n=W \]
where, for every $i\in\{0,\ldots,n\}$, $W_i$ is a saturated sub-$\BPKE$-pair of $W$ of rank $i$. 
If $W_i/W_{i-1}\cong W(\delta_i)$ for $i\in\{1,\ldots,n\}$ and characters $\delta_i\colon K^\times\to E^\times$, then we say that $W$ is \emph{split triangulable with ordered parameter} $\udelta=(\delta_1,\delta_2,\ldots,\delta_n)\colon K^\times\to (E^\times)^n$. 

A $\BPKE$-pair $W$ is triangulable if there is a finite extension $F$ of $E$ such that the $B\vert_K^{\otimes F}$ pair $F\otimes_EW$ is split triangulable. 

An object $V$ of $\Rep_E(G_K)$ is (split) trianguline if $W(V)$ is (split) triangulable.
\end{defin}

\noindent We will use the adjective ``potentially'' in front of the above properties with its usual meaning. Note that some references call ``triangulable'' what we call ``split triangulable''.

The condition about the $W_i$ being saturated in $W$ is not very serious: one can replace each $W_i$ with its saturation in $W$ and obtain this way a filtration where each step is saturated.

\smallskip

\subsection{Main result on potentially trianguline \texorpdfstring{$B$}{B}-pairs}

Let $K$ and $E$ be two $p$-adic fields. 
Let $\bB$ be an $(E,G_K)$-regular ring in the sense of \cite[Definition 2.8]{fonouy} (for instance, $\bB=\bB_{?,E}$ with $?\in\{\HT,\dR,\st,\crys\}$). 

\begin{lemma}\label{tannlist}
The following full subcategories of $\Rep_E(G_K)$ are neutral Tannakian:
\begin{enumerate}
\item the category $\Rep_E^\bB(G_K)$ of $E$-linear representations of $G_K$ that are (potentially) $\bB$-admissible up to twist by a character of $G_K$; 
\item the categories ($\Rep_E^{\stri}(G_K),\Rep_E^{\ptri}(G_K),\Rep_E^{\pstri}(G_K)$) $\Rep_E^{\ptri}(G_K)$  of (split, potentially, potentially split) trianguline $E$-linear representations of $G_K$. 
\end{enumerate}
\end{lemma}

Note that the categories in (ii) are all stable under twisting by $E$-linear characters of $G_K$.

\begin{proof}
Since $\Rep_E(G_K)$ is a neutral Tannakian category, it is enough to check that the categories in (i) and (ii) are stable under direct sums, taking subquotients, tensor products and duals, where all these operations are intended in $\Rep_E(G_K)$.
Proving this for $\Rep_E^\bB(G_K)$ a minor variation on \cite[Theorem 2.13(2)]{fonouy}. 
As for the categories of trianguline representations, one can check easily their stability under all the operations listed above. 
\end{proof}

If $W$ is a trianguline representation and $\langle W \rangle$ is the subcategory of $\Rep_E(G_K)$ tensor generated by $W$, we can define in a canonical way a triangulation of an object in $\langle W\rangle$ starting from a triangulation of $W$. For this we refer to Remark \ref{remsem} and the discussion preceding it.

\begin{rem}\mbox{ }
\begin{itemize}
\item Lemma \ref{tannlist}(2) is a special case of (1): if $\cD\subset\cC$ is an inclusion of neutral Tannakian categories with $\cD$ full in $\cC$ and associated morphism of fundamental groups $\pi\colon\cG_\cC\to\cG_\cD$, then following \cite[Section 2.2]{fontainepst} we can consider the affine algebra $B_{\cD,\alg}$ of $\cG_\cD$, that carries an action of $\cG_\cD$ by left translation, hence of $\cG_\cC$ via $\pi$. Then the objects of $\cD$ are the $B_{\cD,\alg}$-admissible ones in $\cC$. 
We still write (1) and (2) separately because such a period ring has not been studied in the literature as far as the author knows. 
\item In (ii) one can fix the extension of $K$, respectively $E$, over which the $\BPKE$-pairs become triangulable, respectively split, and still get a neutral Tannakian category by the same argument as the given one.
\item One could think of defining a category of couples consisting of a trianguline $E$-linear representation of $G_K$ and a triangulation of the associated $\BPKE$-pair, and make it into a tensor category by means of the argument in the proof of Lemma \ref{tannlist}(2); however one runs into the same problems that make the category of filtered vector spaces non-abelian.
\end{itemize}
\end{rem}

\begin{rem}\label{notann}
The Tannakian categories of Lemma \ref{tannlist} are categories of Galois representations rather than $B$-pairs. For this reason, in the proof of Theorem \ref{ptri} we will first reduce the statements to the case when $W$ is pure of slope $0$: the category of such $\BPKE$-pairs is neutral Tannakian because it is equivalent to that of continuous $E$-representations of $G_K$. The reason we cannot work directly with $B$-pairs is that the Tannakian category of $\BPKE$-pairs is not neutral in general (see for instance \cite[Section 10.1.2]{thecurve}), so it does not fit in the framework of the previous section. This is also the reason why we can only prove statements (ii) and (iii) when $W$ is pure, since then we can, up to extending $E$, find a slope $0$ modification of $W$. Note however that (2)$\implies$(1) and part (iii) with the divisibility condition replaced by $\ell(\uu)=\ell(\uv)$ hold for arbitrary $W$ (not necessarily pure) by virtue of Remark \ref{oldarg}.
\end{rem}

For later use, we introduce some more neutral Tannakian categories. As usual, if $\textbf P$ is a property of $\BPKE$-pairs, we say that an $E$-linear representation of $G_K$ has property $\textbf P$ if the associated $\BPKE$-pair has it.

\begin{defin}\label{qreg}\mbox{ }
	\begin{enumerate}
	\item Let $\sigma$ be an embedding of $E$ into $\Qp$. We say that a $\BPKE$-pair is 
	\emph{$\sigma$-regular} if its $\sigma$-Hodge--Tate--Sen weights are all distinct.
	\item Let $\Rep_E^{\tri,\sigma-\reg}(G_K)$ (respectively $\Rep_E^{\ptri,\sigma-\reg}(G_K)$) be the smallest full Tannakian subcategory of  $\Rep_E^{\tri}(G_K)$ (respectively $\Rep_E^{\ptri}(G_K)$) containing all the (respectively, potentially) trianguline, $\sigma$-regular representations. 
	\item We say that a split triangulable $\BPKE$-pair, 
	of parameters $\delta_1,\ldots,\delta_n$, is \emph{quasi-regular} if there exists a triangulation of $V$ and an embedding $\sigma\colon E\to\Q_p$ for which the following holds: if the $\sigma$-Hodge--Tate--Sen weights of $\delta_i$ and $\delta_j$ coincide for some $i,j\in\{1,\ldots,\dim V\}$, then the $\tau$-Hodge--Tate--Sen weights of $\delta_i$ and $\delta_j$ coincide for every embedding $\tau\colon E\into\Qp$. 
	We say that a potentially trianguline $\BPKE$-pair 
	is quasi-regular if it becomes split triangulable and quasi-regular after replacing $K$ and $E$ by finite extensions.
	\item Let $\Rep_E^{\tri,\qreg}(G_K)$ (respectively $\Rep_E^{\ptri,\qreg}(G_K)$) be the smallest full Tannakian subcategory of  $\Rep_E^{\tri}(G_K)$ (respectively $\Rep_E^{\ptri}(G_K)$) containing 
	all the (respectively, potentially) trianguline quasi-regular representations. 
\item As in \cite[Introduction]{berdimtri}, we say that a $\BPKE$-pair is \emph{split $\Delta(\Q_p)$-triangulable} if it admits a triangulation whose rank 1 subquotients are all restrictions of $\BPQE$-pairs to $G_K$. 
	\end{enumerate}
	\end{defin}

\begin{rem}\mbox{ }
	\begin{enumerate}
	\item All of the categories of Galois representations (not $B$-pairs) introduced in Definition \ref{qreg} are automatically neutral. 
	\item For every $\sigma$, $\Rep_E^{\tri,\sigma-\reg}(G_K)$ (respectively, $\Rep_E^{\ptri,\sigma-\reg}(G_K)$) is a Tannakian subcategory of $\Rep_E^{\tri,\qreg}(G_K)$ (respectively, $\Rep_E^{\ptri,\qreg}(G_K)$): any $\sigma$ for which all of the $\sigma$-Hodge--Tate--Sen weights of a (respectively, potentially) trianguline $E$-representation $V$ of $G_K$ are distinct makes the quasi-regularity condition empty.
	\item Not all objects of $\Rep_E^{\ptri,\sigma-\reg}(G_K)$ (respectively, $\Rep_E^{\ptri,\qreg}(G_K)$) are $\sigma$-regular (respectively, quasi-regular). For instance, a direct sum of two copies of the same $\sigma$-regular representation is not $\sigma$-regular. However, the $\sigma$-regular representations (respectively, the quasi-regular representations) form by definition a tensor generating set of $\Rep_E^{\ptri,\sigma-\reg}(G_K)$ (respectively, $\Rep_E^{\ptri,\qreg}(G_K)$).
	\item Every potentially $\Delta(\Q_p)$-triangulable $\BPKE$-pair is quasi-regular: since all the rank 1 subquotients appearing in the triangulation are restrictions of $\BPQE$-pairs, their Hodge--Tate--Sen weights are uniquely determined by their $\sigma$-Hodge--Tate--Sen weight for a single embedding $\sigma\colon E\into\Qp$; in particular, when two of them share the same $\sigma$-Hodge--Tate--Sen weight for one $\sigma$, they share it for every $\sigma$.
	\item Every potentially trianguline $E$-representation of $G_K$ of dimension at most 3 belongs to $\Rep_E^{\ptri,\qreg}(G_K)$. The author does not have an explicit example of an $E$-representation of $G_K$ that does not belong to $\Rep_E^{\ptri,\qreg}(G_K)$. 
	\end{enumerate}
	\end{rem}

We apply the abstract Tannakian results of Section \ref{pbschur} to some of the categories we introduced above, in order to prove the following theorem. 

\begin{thm}\label{ptri}
	Let $W$ be a quasi-regular $\BPKE$-pair and let $n=\rk W$. 
	\begin{enumerate}[label=(\roman*)]
		\item Assume that either
		\begin{enumerate}[label=(\arabic*)]
			\item there exists a $\BPKE$-pair $W^\prime$ such that $W\otimes_EW^\prime$ is triangulable, or
			\item there exists a tuple $\uu$ with $\length(\uu)<n$ such that $\bS^\uu(W)$ is triangulable.
		\end{enumerate}
		Then $W$ is potentially triangulable. 
	\end{enumerate}
	Moreover, if $W$ is pure (in the sense of the theory of slopes) then:
	\begin{enumerate}[resume]
		\item Conditions (1) and (2) of part (i) are equivalent.
		\item If condition (2) of part (i) holds for some tuple $\uu$, then it holds for all tuples $\uv$ satisfying $\gcd(\ell(\uu),n)\mid\ell(\uv)$. 
	\end{enumerate}
\end{thm}

\begin{rem}
	One can obviously weaken ``triangulable'' in assumptions (1) and (2) of part (i) to ``potentially triangulable''.
\end{rem}

The first part of our result contains as a special case a theorem 
of Berger and Di Matteo \cite[Theorem 5.4, Remark 5.6]{berdimtri}, where it is shown that $W$ is potentially triangulable under the following strengthened form of assumption (1): there exists a $\BPKE$-pair $W^\prime$ such that $W\otimes_EW^\prime$ admits a triangulation whose rank $1$ subquotients have an associated $\bD_\dR$ which is free over $K\otimes_{\Q_p}E$ (in other words, these subquotients cannot be ``partially de Rham'', in the sense of \cite{dingpart}, without being de Rham; this is the case, for instance, if they are restrictions to $G_K$ of $\BPQE$-pairs). After their Remark 5.6, they also provide a counterexample showing that the ``triangulable'' in the conclusion of part (i) of Theorem \ref{ptri} cannot be removed.

Thanks to the results of Section \ref{subsecsc}, we can deduce the following result by specializing Theorem \ref{ptri} to the case of $\BPKE$-pairs of slope 0. Let $\calC$ be the neutral Tannakian category $\Rep_E(G_K)$, and let $\calD$ be the subcategory consisting of the representations that are quasi-regular and potentially trianguline.
Let $\ovl\calD$ be the intermediate category constructed from the inclusion $\calD\subset\calC$, as in Section \ref{tpsub}.

\begin{cor}
	The categories $\calD$ and $\ovl\calD$ coincide.
	\end{cor}

Unfortunately we cannot apply Corollary \ref{ovlDsc} to obtain that $G_{\calD}^\sms$ is simply connected, since $G_\calC^\sms$ is not. For instance, if the residue field of $E$ has $q$ elements, take a multiple $n$ of $q-1$ and a tamely ramified, non-unramified continuous character $\chi\colon G_K\to E^\times$. Consider the injection $f\colon E^\times\to\GL_n(E)$ that maps $e\in E$ to the diagonal element $(e,1,\ldots,1)$, and let $P\chi$ be the representation $G_K\to\PGL_n$ obtained by composing $f\ccirc\chi$ with the projection $\GL_n\to\PGL_n$. If $G_\calC^\sms$ were simply connected, then the representation $P\chi$ would admit a lift along the central isogeny $\SL_n\to\PGL_n$, which is not the case.

Observe that the fact that $\chi$ as above cannot be lifted along $\G_m\to\G_m$, $t\mapsto t^n$, also shows that $G_\calC^\circ$ admits non-trivial central isogenies from connected pro-algebraic groups that are trivial on $G_\calC^\sms$.

\smallskip

\subsection{Crystalline \texorpdfstring{$B$}{B}-pairs}\label{pcrys}

Assume from now on that $E=E^\Gal\subset K$, so that $E_0\subset K_0$. Note that this is the opposite inclusion as one usually asks for in $p$-adic Hodge theory, and we will assume later that $E=E^\Gal=K$. The inclusion $E\subset K$ will guarantee that all of the morphisms of period rings over $E$ that we look at are $G_K$-equivariant, instead of some of them being only $G_E$-equivariant (such as the maps \eqref{pisigma}). One could probably avoid making the assumption and adapt the action of $G_K$ in order to make everything $G_K$-equivariant; however, since in this section we only want to deal with properties of $\BPKE$-pairs being \emph{potentially} true, we are not worried about having to replace $K$ with a finite extension.

Let $\bB$ be an $E$-algebra carrying a $G_K$-action. Set $\bB_E=\bB\otimes_{\Q_p}E$ and extend \emph{$E$-linearly} the action of $G_K$ from $\bB$ to $\bB_E$. 
For every $\sigma\in\Gal(E/\Q_p)$, we denote by $\bB^\sigma$ the ring $\bB$ equipped with the $E$-algebra structure obtained by pre-composing the inclusion $E\subset\bB$ with $\sigma$. Such a structure map is $G_K$-equivariant because $E\subset K$. There is an $E$-linear isomorphism $\bB^\sigma=\bB_E\otimes_{K\otimes_{\Q_p}E,\pi^\sigma}K$, and we denote again by $\pi^\sigma$ the resulting morphism $\bB_E\to\bB^\sigma$. 
We put them together to obtain an $E$-linear, $G_K$-equivariant isomorphism
\begin{equation}\label{decB}
\bigoplus_{\sigma\colon E\to K}\pi^\sigma\colon\bB_E\xto{\sim}\bigoplus_{\sigma\colon E\to K}\bB^\sigma.
\end{equation}

Given a semilinear $\bB_E$-representation $W_{\bB_E}$ of $G_K$, tensoring \eqref{decB} with $W_{\bB_E}$ we obtain an isomorphism
\begin{equation}\label{decW}
\bigoplus_{\sigma\in\Gal(E/\Q_p)}\pi^\sigma\colon W_{\bB_E}\xto{\sim}\bigoplus_{\sigma\in\Gal(E/\Q_p)}W_{\bB_E}\otimes_{\bB_E,\pi^\sigma}\bB ,
\end{equation}
where each factor on the right is a semilinear $\bB$-representation of $G_K$. We write $W_\bB^\sigma=W_{\bB_E}\otimes_{\bB_E,\pi^\sigma}\bB$; it is a $\bB^\sigma$-representation of $G_K$. 
When applying decomposition \eqref{decB} we will write $\pi^\sigma$ for the maps there without specifying the relevant $\bB$ or $W_{\bB_E}$; it will always be evident what we are referring to. The notation $\pi^\sigma$ will be used for quite a few morphisms in the following, all related to decomposition \eqref{decW}. We believe this will avoid adding burdens to the notation without creating any confusion.

\begin{defin}\label{sigmadRdef}
We say that a $\BPKE$-pair $(W_e,W_\dR^+)$ is \emph{$\sigma$-$\C_p$-admissible}, respectively \emph{$\sigma$-Hodge--Tate}, if $\C_p^\sigma\otimes_{\C_p\otimes_{\Q_p}E,\pi^\sigma}(W_\dR^+/tW_\dR)$, respectively $\bB_\HT^\sigma\otimes_{\bB_{\HT,E},\pi^\sigma}(W_\dR^+/tW_\dR^+)$, is trivial.

We say that a $\bB_{e,E}$-representation $W_e$ of $G_K$ is \emph{$\sigma$-de Rham} if $\bB_\dR^\sigma\otimes_{\bB_{\dR,E},\pi^\sigma}(\bB_{\dR,E}\otimes_{\bB_{e,E}}W_e)$ is trivial. 
 
We say that a $\BPKE$-pair $(W_e,W_\dR^+)$ is \emph{$\sigma$-de Rham} if $W_e$ is.

We say that a continuous $E$-linear representation of $G_K$ has one of the above properties if the associated $\BPKE$-pair does.
\end{defin}

It is equivalent to the last part of the definition to say that an $E$-linear representation $V$ of $G_K$ is $\sigma$-$\C_p$-admissible, $\sigma$-Hodge--Tate or $\sigma$-de Rham if and only if it is $\C_p^\sigma$, $\bB_\HT^\sigma$, or $\bB_\dR^\sigma$-admissible, respectively. When $K=E$ these notions coincide with those introduced in \cite{dingpart}; in the general case they are still completely analogous to those in \emph{loc. cit.} apart from the fact that our $\sigma$ is an automorphism of $E$, whereas in there it is an embedding $K\into\C_p$ (in some sense, we are decomposing our semilinear objects in different directions).

Let $f$ be the inertial degree of $E_0$ over $\Q_p$. Define an endomorphism $\varphi_E$ of $E\otimes_{E_0}\bB_\st$ as $1\otimes\varphi^f$, and denote again with $\varphi_E$ its restriction to $E\otimes_{E_0}\bB_\crys$. We extend $E$-linearly the action of $G_K$ on $\bB_\st$ and $\bB_\cris$ to $E\otimes_{E_0}\bB_\st$ and $E\otimes_{E_0}\bB_\crys$ (recall that $E_0\subset K_0$). The actions of $G_K$ and $\varphi_E$ commute on both rings, and they can be extended to their fields of fractions in the obvious way.

We choose once and for all an extension $\log$ of the $p$-adic logarithm from a map $1+\fm_{\C_p}\to\fm_{\C_p}$ to a map $\C_p^\times\to\C_p$, setting in particular $\log(p)=0$. This choice determines an embedding $\bB_\st\into \bB_\dR$, that is fixed throughout the text. We denote by $E\bB_\crys$ the subring of $\bB_\dR$ generated by $E$ and $\bB_{\crys}$, and by $E\bB_\st$ the subring of $\bB_\dR$ generated by $E$ and $\bB_\st$. 
Similarly to \cite[Section 2]{berdimtri}, we attach to every $\sigma\in\Gal(E/\Q_p)$ two $G_K$-equivariant isomorphisms
\[ \sigma\otimes\varphi^{n(\sigma)}\colon E\otimes_{E_0}\bB_{\crys,E}\to E\bB_\crys \]
and
\[ \sigma\otimes\varphi^{n(\sigma)}\colon E\otimes_{E_0}\bB_{\st,E}\to E\bB_\st, \]
where $n(\sigma)$ is the element of $\{0,\ldots,f-1\}$ such that $\sigma=\varphi^{n(\sigma)}$ on $E_0$. We use again the notation $\pi^\sigma$ for these isomorphisms; it will not create any ambiguity. For every $\sigma\in\Gal(E/\Q_p)$ we denote by $t_\sigma$ the element of $E\bB_\crys$ constructed in \cite[Section 5]{bermult} (see also \cite[Proposition 2.4]{berdimtri}). One has $t_\sigma=\pi^\sigma(t_\Id)$ for every $\sigma$. 
We define Frobenius operators on $E\bB_\crys$ and $E\bB_\st$ by transporting $\varphi_E$ via the above isomorphisms, and we still denote them by $\varphi_E$.

Observe that $(E\bB_\crys)\otimes_{\Q_p}E\cong E\otimes_{E_0}(\bB_\crys\otimes_{\Q_p}E)$, where the tensor product over $E_0$ is taken with respect to the $E_0$-vector space structure of $\bB_\crys$. 

We recall the following observation of Berger and Di Matteo.

\begin{lemma}\label{BeEphi}\cite[Proposition 2.2]{berdimtri} 
The composite map $\bB_{e,E}\into \bB_{\crys,E}\onto E\otimes_{E_0}\bB_\crys$ gives an identification
\begin{equation}\label{eEcrys} \bB_{e,E}=(E\otimes_{E_0}\bB_\crys)^{\varphi_E=1}. \end{equation}
\end{lemma}

We will always consider $\bB_{e,E}$ as a subring of $E\bB_\crys$ and of $E\bB_\st$ via \eqref{eEcrys}. In particular, for every $\sigma\in\Gal(E/\Q_p)$ there are maps
\begin{equation}\label{pisigma} \bB_{e,E}\into E\otimes_{E_0}\bB_\crys\xto{\sigma\otimes\varphi^{n(\sigma)}}E\bB_\crys\into E\bB_\st\into \bB_\dR. 
\end{equation}
Given $\bB\in\{E\bB_\crys, E\bB_\st, \bB_\dR\}$, we write $\bB^\sigma$ for the ring $\bB$ equipped with the $\bB_{e,E}$-module structure arising from the maps in \eqref{pisigma}. The resulting $E$-algebra structure map, given by the composite $E\to \bB_{e,E}\to\bB$, is the same as that introduced before Equation \eqref{decB}, so that there are no conflicts in the notation. 

\begin{defin}\label{defsigmacris}
We say that a $\bB_{e,E}$-representation $W_e$ of $G_K$ is \emph{$\sigma$-crystalline}, respectively \emph{$\sigma$-semistable}, if $E\bB_\crys^\sigma\otimes_{\bB_{e,E}}W_e$, respectively $E\bB_\st^\sigma\otimes_{\bB_{e,E}}$, is trivial.

We say that a $\BPKE$-pair $(W_e,W_\dR^+)$ is \emph{$\sigma$-crystalline} or \emph{$\sigma$-semistable} if $W_e$ has the respective property.

We say that a continuous $E$-linear representation of $G_K$ has one of the above properties if the associated $\BPKE$-pair does.
\end{defin}

We remark that Ding defines in \cite{dingpder} a notion of $B_\sigma$-pair for every embedding $\sigma\colon K\into\C_p$, and attaches to a $B$-pair $W=(W_e,W_\dR^+)$ a $B_\sigma$-pair $W_\sigma=(W_{e,\sigma},W_{\dR,\sigma}^+)$ for each $\sigma$. When $K=E$, a $\BPKE$-pair $W$ is $\sigma$-crystalline in our sense if and only if Ding's $W_{e,\sigma}$ becomes trivial after extending its scalars to $E\bB^\sigma_\crys$.

We extend the monodromy operator $N$ on $\bB_{\st}$ to an $E$-linear nilpotent operator $N_E$ on $E\bB_\st$.
Since $E\bB_\crys^\sigma=(E\bB_\st^\sigma)^{N_E=0}$, a $\bB_{e,E}$-representation $W_e$ of $G_K$ is $\sigma$-crystalline if and only if it is $\sigma$-semistable and the operator induced on $(\bB_\st^\sigma\otimes_{\bB_{e,E}}W_e)^{G_K}$ by $N_E$ is identically zero.

The filtration $(\Fil^i\bB_\dR)_{i\in\Z}$ on $\bB_\dR$ defined by $\Fil^i\bB_\dR=t^i\bB_\dR^+$ for $i\in\Z$ induces filtrations $(\Fil^iE\bB_\crys)_{i\in\Z}$ and $(\Fil^iE\bB_\st)_{i\in\Z}$ on $E\bB_\crys$ and $E\bB_\st$, respectively. The graded ring associated with $E\bB_\crys$, $E\bB_\st$ and $\bB_\dR$ is the same, $\bB_\HT$.

\smallskip

\subsection{Reminders on Fontaine's classification of \texorpdfstring{$\bB_\dR$}{BdR}-representations}

We recall Fontaine's classification of $\bB_\dR$-representations from \cite{fontarith} (recall that $\bB_\dR=\bB_{\dR,\Q_p}$, so that we are working with $E=\Q_p$ here). 

Let $C(\ovl K)$ (respectively $C(\ovl K/\Z)$) be the set of $G_K$-orbits in the additive group $\ovl K$ (respectively $\ovl K/\Z$). For $A\in C(\ovl K)$, let $K_A$ be the extension of $K$ generated by the elements of $A$. Let $d_A$ be the degree of $K_A/K$. Let $a$ be any element of $A$ and let $r_A$ be the smallest integer such that
\[ v_p(a\log(\chi_K^\cyc(\gamma))>\frac{1}{p-1} \]
for all $\gamma\in\Gamma_{K,r_A}$. Thanks to the previous inequality we can define a $1$-dimensional $K_A$-linear representation $\rho_A\colon\Gamma_{K,r_A}\to K_A^\times$ by setting
\begin{equation}\label{rhoa} \rho_A(\gamma)=\exp(a\log\chi_K^\cyc(\gamma)) \end{equation}
for every $\gamma\in\Gamma_{K,r_A}$. 
Now the induction
\[ N[A]=\Ind_{\Gamma_{K,r_A}}^{\Gamma_K}\rho_A. \]
is a $K_A$-linear representation of $\Gamma_K$ of dimension $p^{r_A}$. We see it as a $K$-linear representation of dimension $d_Ap^{r_A}$. We define a semilinear $K_\infty$-representation of $\Gamma$ by
\[ N_\infty[A]=K_\infty\otimes_KN[A], \]
where $\Gamma$ acts via its natural action on $K_\infty$ and diagonally on $N_\infty[A]$. It is not always the case that $N_\infty[A]$ is a simple object in the category of semilinear $K_\infty$-representations of $\Gamma$, but all of its simple factors are isomorphic. We choose one and denote it by $K_\infty[A]$. As proved in \cite[Proposition 2.13]{fontarith}, the dimension of $K_\infty[A]$ is $d_Ap^{s_A}$ for some integer $s_A$ with $0\le s_A\le r_A$. 

There exists no $G_K$-equivariant section of the projection $\bB_\dR\to\C_p$, but one can define a $G_K$-equivariant homomorphism $s\colon\ovl K\into \bB_\dR$ such that $\theta\ccirc s=\id_{\ovl K}$, as in \cite[Section 3.1]{fontarith} (what is noted $\ovl P$ there always contains $\ovl K$). In particular we have a $G_K$-equivariant section $K_\infty\to \bB_\dR$. We define a semilinear $\bB_\dR$-representation of $G_K$ by setting
\[ \bB_\dR[A]=\bB_\dR\otimes_{K_\infty}K_\infty[A], \]
where the tensor product is taken via the aforementioned section, $G_K$ acts via the projection $G_K\to\Gamma_K$ on $K_\infty[A]$ and diagonally on $\bB_\dR[A]$. By \cite[Proposition 3.18]{fontarith}, the isomorphism class of the $\bB_\dR$-representation $\bB_\dR[A]$ only depends on the image of $A$ in $C(\ovl K/\Z)$. For this reason we will also speak unambiguously of $\bB_\dR[A]$ when $A$ is an orbit in $C(\ovl K/\Z)$ rather than $C(\ovl K)$.

The construction above already gives all the simple objects in the category of semilinear $\bB_\dR$-representations of $G_K$. There exist however non-semisimple objects, that Fontaine also describes. Let $d\in\Z_{>0}$. Following \cite[Section 2.6]{fontarith}, denote by $\Z_p\{0;d\}$ the $\Z_p$-vector space of polynomials of degree smaller than $d$ in one variable $X$, equipped with the unique $\Z_p$-linear action of $G_K$ satisfying
\[ g(X)=X+\log\chi_K^\cyc(g) \]
for all $g\in G_K$. Note that this is the same as the action one would get by identifying $X$ with $\log{t}$, where $t$ is the usual generator of $\Fil^1\bB_\dR$. It is clear that $\Z_p\{0;d\}$ is given by successive extensions of $d$ trivial $1$-dimensional $\Z_p$-linear representations of $G_K$. 
Given $A\in C(\ovl K/\Z)$ and $d\in\Z_{>0}$, we define a semilinear $\bB_\dR$-representation of $G_K$ by
\[ \bB_\dR[A;d]=\bB_\dR[A]\otimes_{\Z_p}\Z_p\{0;d\}, \]
on which $G_K$ acts diagonally. This representation has dimension $dd_Ap^{s_A}$, and its simple subquotients are all isomorphic to $\bB_\dR[A]$.

By \cite[Théorème 3.19]{fontarith}, every semilinear $\bB_\dR$-representation of $G_K$ can be written in a unique way, up to permutation of the factors, as a direct sum of representations of the form $\bB_\dR[A;d]$ for some $A\in C(\ovl K/\Z)$ and $d\in\Z_{>0}$.


\smallskip

\subsection{Reducing Theorem \ref{ptri} to the case of slope \texorpdfstring{$0$}{0}}

We reduce Theorem \ref{ptri}(i) to the case where $W$ is pure. 

\begin{prop}
Assume that (i) and (ii) of Theorem \ref{ptri} are true whenever $W$ is pure. Then Theorem \ref{ptri}(i) holds. 
\end{prop}

Note that the $\BPKE$-pair $W^\prime$ appearing in statement (ii) is not assumed to be pure of slope $0$.

\begin{proof}
Assume that $W$ is pure. 
We will use the following lemma.

\begin{lemma}\label{triext}
Let $0\to W_1\to W\to W_2\to 0$ be an exact sequence of $\BPKE$-pairs. Suppose that the statements of Theorem \ref{ptri}(i) and (ii) are true for $W_1$ and $W_2$. Then the statement of Theorem \ref{ptri}(i) is true for $W$.
\end{lemma}

\begin{proof}
We prove the statement of Theorem \ref{ptri}(i) for $W$. Assume first that there exists a $\BPKE$-pair $W^\prime$ such that $W\otimes_EW^\prime$ is triangulable. The sequence
\[ 0\to W_1\otimes_EW^\prime\to W\otimes_EW^\prime\to W_2\otimes_EW^\prime\to 0 \]
is exact (since the underlying sequence $E$-vector spaces is exact). Since the category of split triangulable $\BPKE$-pairs is stable under subquotients, $W_1\otimes_EW^\prime$ and $W_2\otimes_EW^\prime$ are triangulable. Then Theorem \ref{ptri}(ii) for $W_1$ and $W_2$ implies that there exist (not too long) tuples $\uu$ and $\uv$ such that $\bS^\uu(W_1)$ and $\bS^\uv(W_2)$ are triangulable, and Theorem \ref{ptri}(i) implies that $W_1$ and $W_2$ are potentially triangulable. It follows immediately that $W$ is also potentially triangulable.

Now let $n=\rk W$ ans assume that there exists $\uu$, with $\length(\uu)<n$, such that $\bS^\uu(W)$ is triangulable. Then $\bS^\uu(W_1)$ and $\bS^\uu(W_2)$ are also triangulable. This is clear from the fact that these two $\BPKE$-pairs are sub-$\BPKE$-pairs of $\bS^\uu(W)$, as one sees from the definition of the Schur functor, and the category of triangulable $\BPKE$-pairs is stable under taking subobjects.
\end{proof}

Let $W$ be an arbitrary $\BPKE$-pair. 
By \cite[Théorème 2.1]{bergchen} (which is a translation to the language of $B$-pairs of \cite[Theorem 6.10]{kedmon}) $W$ admits a Dieudonné--Manin filtration, that is, an increasing filtration in sub-$\BPKE$-pairs whose graded pieces are pure of increasing slopes. Then, by Lemma \ref{triext}, if Theorem \ref{ptri} is true when $W$ is pure, it is also true for an arbitrary $W$. 
\end{proof}

Next we reduce all statements of Theorem \ref{ptri} to the case when $W$ is pure of slope $0$. Recall that a modification of a $\BPKE$-pair $W$, in the sense of \cite[Définition 2.1.8]{berconstr}, 
is a $\BPKE$-pair $W_0$ satisfying $W_{0,e}=W_e$, that is, modifying $W$ amounts to replacing $W_\dR^+$ with a different $\bB_\dR^+$-lattice in $W_\dR$. 

\begin{prop}
Assume that the statements of Theorem \ref{ptri} hold for every $E$ and every $W$ that is pure of slope $0$. Then they hold for every pure $W$.
\end{prop}

\begin{proof}
We rely on the following lemma.

\begin{lemma}\label{trimod}
Let $W$ be a $\BPKE$-pair and $W_0$ be a modification of $W$. Then $W$ is (split) triangulable if and only if $W_0$ is.
\end{lemma}

\begin{proof}
By \cite[Corollary 3.2]{berdimtri} a $\BPKE$-pair is (split) triangulable if and only if the associated $\bB_{e,E}$-representation is (split) triangulable. The conclusion follows from the fact that $W_e=W_{0,e}$.
\end{proof}

Let $W$ be a $\BPKE$-pair, pure of slope $s=d/h$ with $d$ and $h$ coprime integers, $h>0$, and let $E_h$ is the unique unramified extension of $E$ of degree $h$.

\begin{lemma}\label{modEh}
There exists a modification $W_0$ of $W\otimes_EE_h$ of slope $0$.
\end{lemma}

\begin{proof}
The $B\vert_K^{\otimes E_h}$-pair $W\otimes_EE_h$ has slope $d$.
Let $D$ be the $(\varphi,\Gamma_K)$-module over $E\otimes_{\Q_p}\bB_{\rig,K}^\dagger$, of slope $d$, associated with $W\otimes_EE_h$ by Berger's dictionary. 
The $(\varphi,\Gamma_K)$-module over $E\otimes_{\Q_p}\bB_{\rig,K}^\dagger$ associated with the $\BPKE$-pair $W^\prime=(W_e\otimes_EE_h,t^{-d}(W_\dR^+\otimes_EE_h))$ is $t^{-d}D$ (here we are improperly writing $t$ for the element $t\otimes 1\in\bB_\dR\otimes_{\Q_p}E)$. Since $t^{-d}D$ is isoclinic of slope $0$, we conclude that $W^\prime$ is pure of slope $0$. 
\end{proof}

Obviously $W$ is triangulable if and only if $W\otimes_EE_h$ is triangulable (we are not asking for splitness). Hence it is enough to deduce the statements of Theorem \ref{ptri} after (implicitly) replacing $W$ with $W\otimes_EE_h$. Thanks to Lemma \ref{modEh}, we can then pick a modification $W_0$ of $W$ of slope $0$. 

Assume now that Theorem \ref{ptri} is known for $\BPKE$-pairs of slope $0$; in particular, it holds for $W_0$. We deduce the statements of Theorem \ref{ptri} for $W$. Observe that:
\begin{itemize}
\item[(a)] For every $\BPKE$-pair $W^\prime$, $W_0\otimes_{E}W^\prime$ is a modification of $W\otimes_EW^\prime$, hence it is triangulable if and only if $W\otimes_EW^\prime$ is. 
\item[(b)] For every tuple $\uu$, the $\BPKE$-pair $\bS^\uu(W_0)$ is a modification of $\bS^\uu(W)$, hence it is triangulable if and only if $\bS^\uu(W_0)$ is. 
\end{itemize}
By remarks (a) and (b), if (1) or (2) in Theorem \ref{ptri}(i) holds for $W$, then it holds for $W_0$. Hence Theorem \ref{ptri}(i) applied to $W_0$ gives that $W_0$ is potentially triangulable, which in turn implies that $W$ is potentially triangulable. Moreover, Theorem \ref{ptri}(ii) gives that conditions (1) and (2) are equivalent for $W_0$, hence they are also equivalent for $W$. In alternative, Theorem \ref{ptri}(ii) for $W$ follows immediately from Theorem \ref{ptri}(ii) for $W_0$ and remark (b) above.
\end{proof}


\smallskip

\subsection{Extending the base and coefficient fields}\label{fieldext}

Before continuing with the proof of Theorem \ref{ptri}, we give here a procedure for replacing our base and coefficient fields $K$ and $E$ with a common finite extension. We will refer to it a couple of times in the following. We keep working under the assumption $E^\Gal\subset K$. Let $L$ be a Galois extension of $\Q_p$ containing $K$, and let $\sigma$ be an element of $\Gal(E/\Q_p)$. Let $W$ be a $\BPKE$-pair. We:
\begin{enumerate}[label=(\arabic*)]
\item replace both $E$ and $K$ with $L$ and the $\BPKE$-pairs $W$, 
with their extension of scalars to $L$ and the restriction of the Galois action to $G_L$, and
\item replace the automorphism $\sigma$ of $E$ with an arbitrarily chosen extension of it to an automorphism $\wtl\sigma_0$ of $L$.
\end{enumerate}
There are isomorphisms of $\bB_\dR$-representations of $G_L$
\begin{align*}
W_{\dR}^\sigma&\xto{x\mapsto x\otimes 1}\bigoplus_{\substack{\wtl\sigma\colon L\to L \\ \wtl\sigma\vert_E=\sigma}}(L\otimes_EW)_{\dR}^{\wtl\sigma}\xto{\pi^{\wtl\sigma_0}} (L\otimes_EW)_{\dR}^{\wtl\sigma_0},
\end{align*}
and of $E\bB_\crys$-representations of $G_L$
\begin{align*}
W_{\crys}^\sigma&\xto{x\mapsto x\otimes 1}\bigoplus_{\substack{\wtl\sigma\colon L\to L \\ \wtl\sigma\vert_E=\sigma}}(L\otimes_EW)_{\crys}^{\wtl\sigma}\xto{\pi^{\wtl\sigma_0}} (L\otimes_EW)_\crys^{\wtl\sigma},
\end{align*}
that induce isomorphisms between the leftmost and rightmost objects in each of the two lines.
In any given application, we replace all the elements that have been chosen in $W_{\dR}^\sigma$ and $W_\cris^\sigma$ with their images in $(L\otimes_EW)_{\dR}^{\wtl\sigma}$ and $(L\otimes_EW)_{\cris}^{\wtl\sigma}$ via the isomorphisms above. Remark that, if $f_{L/E}$ is the inertia degree of $L/E$, then $\varphi_{L}=(1\otimes\varphi_E)^{f_{L/E}}$.

\smallskip

\subsection{Proof of Theorem \ref{ptri} for \texorpdfstring{$B$}{B}-pairs of slope \texorpdfstring{$0$}{0}}\label{proof0}

We now prove Theorem \ref{ptri} assuming that $W$ is pure of slope $0$. 
We will apply the results of Sections \ref{tpsub}-\ref{pbschur} by choosing:
\begin{itemize}
\item as $\calC$ the category of $\BPKE$-pairs pure of slope $0$ (this category is neutral Tannakian since it is equivalent to the category of continuous $E$-representations of $G_K$ by \cite[Proposition 2.2.9]{berconstr}), 
\item as $\calD$ the category $\Rep^{\ptri,\qreg}_E(G_K)$ introduced in Definition \ref{qreg}.
\end{itemize}

\begin{rem}
Recall for a moment the notation of Lemma \ref{Ifinite}. One could hope that, since $I_V$ is finite by Lemma \ref{Ifinite}, it is possible to find an open subgroup of $G_K$ so that the image in $\GL_V$ of the Tannakian fundamental group of $\Rep^\tri_E(G_K)$ intersects $I_V$ trivially. Unfortunately this is in general impossible if $I_V$ is non-trivial: For instance, for a two-dimensional trianguline $E$-representation $V$ of $G_K$ and for every finite extension $K^\prime$ of $K$, the image of $G_{\Rep_E^\tri(G_{K^\prime})}$ in $\GL(V)$ is the Zariski closure of the image of $G_{K^\prime}$ in $\GL(V)$. If $V$ does not admit an abelian subgroup of index $1$ or $2$, then for every $K^\prime$ the above Zariski closure contains $\SL(V)$. We know from Lemma \ref{Ifinite} that $I_V$ is contained in $\SL(V)$, hence, if it is not trivial, there is no way to make it trivial by replacing $K$ with a finite extension. In other words, one cannot prove by an abstract Tannakian argument that replacing our current $\calD$ with the category of potentially trianguline $E$-representations of $G_K$ makes the kernel $I$ of $G_{\ovl\calD}\to G_{\calD}$ trivial.
\end{rem}

By Lemma \ref{Ivanishgen}, it is enough to prove Theorem \ref{ptri} for all the $\BPKE$-pairs in a tensor generating set of $\ovl\calD$. By Corollary \ref{ovlDgen}, such a set is provided by the $\BPKE$-pairs $W$ of slope 0 for which there exists a tuple $\uu$ such that $\length(\uu)<\rk(W)$ and $\bS^\uu(W)\in\Rep_E^{\ptri,\qreg}(W)$.
 
\begin{lemma}\label{uqreg}
	Let $W$ be a $\BPKE$-pair, $n=\rk(W)$, and $\uu$ a tuple with $\length(\uu)<n$. 
	\begin{itemize}
	\item The Hodge--Tate weights of $W$ are uniquely determined by those of $\bS^\uu(W)$.
	\item If $W$ is a $\BPKE$-pair such that, for some tuple $\uu$ with $\length(\uu)<n$, $\bS^\uu(W)$ is quasi-regular, then $W$ is also quasi-regular.
	\end{itemize}
\end{lemma}

\begin{proof}
	Let $\uu$ be as in the statement.
Let $T_n$ be a maximal split torus in $\GL_n(E)$. 
	For an arbitrary embedding $\tau\colon K\into\C_p$, consider the Hodge--Tate cocharacter 
	\[ \HT_{W,\tau}\colon\bG_m(\C_p)\to T_n(\C_p) \] 
	associated with $W$ and $\tau$. The representation $\bS^\uu\colon\GL_n\to\GL_m$ can be restricted to a representation $\bS_T^\uu\colon T_n\to T_m$, where m is the rank of $\bS^\uu(W)$ and $T_m$ is a maximal split torus in $\GL_m(E)$. The Hodge--Tate cocharacter associated with $\bS^\uu(W)$ is then $\bS_T^\tau\ccirc\HT_\tau$. Since the kernel of $\bS_T^\tau$ is the group $\mu_{\ell(\uu)}$ of $\ell(\uu)$-th roots of unity, two Hodge--Tate cocharacters $\bG_m(\C_p)\to T_n(\C_p)$ coincide after composition with $\bS_T^\uu$ if and only if their quotient takes values in $\mu_{\ell(\uu)}$. This is only possible if such a quotient is trivial, since all Hodge--Tate cocharacters $\bG_m(\C_p)\to T_n(\C_p)$ take values in the image of the exponential map $(\C_p)^n\to T(\C_p)$.
	\end{proof}

Thanks to Lemma \ref{uqreg}, we can assume from now on that $W$ is quasi-regular. 

We observe that if a $\BPKE$-pair $W^\prime$ as in condition (1) of Theorem \ref{ptri}(i) exists, then there exists, up to extending $E$, a $\BPKE$-pair pure of slope $0$ satisfying the same property: The first non-zero step $\Fil^1W^\prime$ in the Dieudonné--Manin filtration of $W^\prime$ is a pure sub-$\BPKE$-pair of $W^\prime$, and $W\otimes\Fil^1W^\prime$ is a sub-$\BPKE$-pair of $W\otimes_EW^\prime$. Since $W\otimes_EW^\prime$ is triangulable, the same is true for $W\otimes_E\Fil^1W^\prime$. Hence, up to replacing $W^\prime$ with $\Fil^1W^\prime$, we can assume that $W^\prime$ is pure. Then, up to implicitly extending $E$, we can modify $W^\prime$ to a $\BPKE$-pair $W_0^\prime$ which is pure of slope $0$. Since $W\otimes_EW_0^\prime$ is a modification of the triangulable $\BPKE$-pair $W\otimes_EW^\prime$, it is triangulable.

Given that we can harmlessly strengthen condition (1) of part (i) by requiring $W^\prime$ to be pure of slope $0$, parts (ii) and (iii) of the theorem are an immediate consequence of Proposition \ref{schur} applied to our choice of $\calC$ and $\calD$.

We prove part (i). Assume that one of the equivalent conditions in part (i) holds. Thanks to part (iii), we can assume that $\Sym^nW$ is triangulable. 
We proceed by induction on the rank of $W$. If the rank is $1$ there is nothing to prove. If the rank is larger than $1$, we prove that there exist finite extensions $E_1$ of $E$ and $K_1$ of $K$ such that $(W\otimes_EE_1)\vert_{G_{K_1}}$ contains a saturated sub-$\BPKKEE$-pair $W^\prime$ of rank $1$. This is enough: let $(W\otimes_EE_1)\vert_{G_{K_1}}/W^\prime$ be the cokernel of the inclusion of $W^\prime$ in $W\otimes_EE_1$; it is again a $\BPKKEE$-pair because $W_1$ is saturated in $(W\otimes_EE_1)\vert_{G_{K_1}}$. Moreover $\Sym^{n}((W\otimes_EE_1)\vert_{G_{K_1}}/W^\prime)$ can be easily seen to be a quotient of the split trianguline $\BPKKEE$-pair $\Sym^{n}(W\otimes_EE_1)\vert_{G_{K_1}}$, hence it is also split trianguline and we can use the inductive hypothesis in rank $\rk W-1$.

We proceed to prove the sufficient claim from the previous paragraph. We implicitly replace both $K$ and $E$ with finite extensions such that $E=E^\Gal=K$ and that $\Sym^nW$ becomes split triangulable over $E$, and then extend scalars in $W$ to the new $E$. 
We still denote $K$ and $E$ with distinct letters in order to emphasize the different roles played by the base and coefficient field, and to make it clear at which point we are using the fact that they coincide. 
For $X\in\{W_e,W_\dR^+,W_\dR\}$ and elements $f_1,\ldots,f_n$ of $X$, we will abuse of notation and always write $f_1\otimes\ldots\otimes f_n$ for the image of this element of $X^{\otimes n}$ under the natural projection $X^{\otimes n}\to\Sym^nX$.


\begin{lemma}\label{qr}
The $\BPKE$-pair $\Sym^nW$ is a direct factor of $(W\otimes_EW^\vee)^q\otimes_E\det^r{W}$ for some $q\in\Z_{\ge 0}$ and $r\in\Z$.
\end{lemma}

\begin{proof}
Since all the $\BPKE$-pairs involved in the statement are pure of slope $0$, it is enough to prove the analogous result after replacing $W$ with an $E$-representation $V$ of $G_K$. The action of $\GL(V)$ on $V\otimes_EV^\vee\otimes_E\det{V}$ factors through $\GL(V)/\mu_n$ and makes $V\otimes_EV^\vee\otimes_E\det{V}$ into a faithful representation of the latter group. By \cite[Theorem 3.1(a)]{delhodge} $V\otimes_EV^\vee\otimes_E\det{V}$ is a tensor generator of the Tannakian category of $E$-representations of $\GL(V)/\mu_n$, meaning that every irreducible $E$-representation of $\GL(V)/\mu_n$ appears as a direct factor of
\[ (V\otimes_EV^\vee\otimes_E\det{V})^a\otimes_E((V\otimes_EV^\vee\otimes_E\det{V})^{\vee})^b \]
for some non-negative integers $a$ and $b$. Then $q=a+b$ and $r=a-b$ meet the requirements of the theorem.
\end{proof}

\begin{rem}
With the notations of the above proof, it is immediate that the relations defining $\Sym^nV$ as a direct factor of $(V\otimes_EV^\vee)^q\otimes_E\det^r{V}$ carry over when one replaces $V$ with either $W_e$ or $W_\dR^+$, so that Lemma \ref{qr} holds even when $W$ is not of slope $0$. More generally, the proof can be rephrased in terms of Schur functors in any $E$-linear tensor category.
\end{rem}

Take $q$ and $r$ as in Lemma \ref{qr}. By Lemma \ref{dual}, the $\BPKE$-pair $W\otimes_EW^\vee$ is triangulable. Fix a triangulation of $W\otimes_EW^\vee$. The triangulation of $W\otimes_EW^\vee$ induces triangulations of $(W\otimes_EW^\vee)^q$, of the twist $(W\otimes_EW^\vee)^q\otimes_E\det^r{W}$ and of its direct factor $\Sym^nW$.
We write in short $N$ for the rank $\binom{2n-1}{n-1}$ of $\Sym^nW$. Let
\begin{equation}\label{sym2tri} 0=\Fil^0\Sym^nW\subset\Fil^1\Sym^nW\subset\ldots\subset\Fil^{N}\Sym^nW \end{equation}
be the above triangulation of $\Sym^nW$, and let $W_i^\prime$ be the rank $1$ quotient $\Fil^i\Sym^nW/\Fil^{i-1}\Sym^nW$ for $1\le i\le N$. 

By \cite[Theorem 3.4]{berdimtri} the triangulation of $(W\otimes_EW^\vee)_e$ induced by the triangulation of $W\otimes_EW^\vee$ splits as a direct sum of $\bB_{e,E}$-representations of rank $1$. Since the triangulation of $\Sym^nW$ is constructed from that of $W\otimes_EW^\vee$ via Lemma \ref{qr}, the triangulation of the $\bB_{e,E}$-representation $(\Sym^nW)_e$ induced by \eqref{sym2tri} also splits as a direct sum
\[ (\Sym^nW)_e\cong\bigoplus_{i=1}^{N}W_{i,e}^\prime, \]
for some $\bB_{e,E}$-representations $W_{1,e},\ldots,W_{N,e}$ of rank $1$. 
Tensoring with $\bB_{\dR,E}$ over $\bB_{e,E}$ we obtain a decomposition
\begin{equation}\label{symdirdR} (\Sym^nW)_\dR\cong\bigoplus_{i=1}^{N}W_{i,\dR}^\prime. 
\end{equation}

Since we are assuming that $E^\Gal\subset K$, we can apply the decomposition \eqref{decW} to the $\bB_\dR$-representations $W_\dR$, $\Sym^nW_\dR$ and $W_{i,\dR}^\prime$, $1\le i\le N$, to write
\begin{align*} W_\dR&\cong\bigoplus_{\sigma\colon E\into K}W_{\dR}^\sigma , \\
(\Sym^nW)_\dR&\cong\bigoplus_{\sigma\colon E\into K}(\Sym^nW)_{\dR}^\sigma ,  \\
W_{i,\dR}^\prime&\cong\bigoplus_{\sigma\colon E\into K}W_{i,\dR}^{\prime,\sigma}.  \end{align*}
The decompositions above are obviously compatible in the sense that
\begin{equation}\label{compdR} \Sym^n(W_{\dR,\sigma})\cong(\Sym^nW)_{\dR,\sigma}\cong\bigoplus_{i=1}^NW_{i,\dR,\sigma}^\prime
\end{equation}
as $\bB_\dR$-representations; we used the direct sum decomposition \eqref{symdirdR} for the second isomorphism. 

Recall that we assumed $W$ to be quasi-regular, as in Definition \ref{qreg}. Let $\sigma$ be an embedding of $E$ into $K$ that, seen as en embedding $E\into\Qp$, verifies the quasi-regularity condition for $W$. 
We can write the isomorphism class of $W_{\dR,\sigma}$ following Fontaine's classification: 
\begin{equation}\label{WdRsigma} W_{\dR}^\sigma\cong\bigoplus_{\substack{A\in C(\ovl K/\Z) \\ d\in\Z_{\ge 0}}}\bB_\dR[A;d]^{h_{A,d,\sigma}} 
\end{equation}
for some non-negative integers $h_{A,d,\sigma}$, almost all zero. By the compatibilities \eqref{compdR}, the symmetric $n$-th power of the right-hand side of \eqref{WdRsigma} must decompose as a direct sum of $1$-dimensional $\bB_\dR$-representations. 

Let $A\in C(\ovl K/\Z)$ be an orbit appearing in the decomposition \ref{WdRsigma}. Since $\Z_p\{0;d\}$ contains a trivial rank 1 $\Z_p$-subrepresentation, $\Sym^n(\bB_\dR[A;d])$ contains a sub-$\bB_\dR$-representation of the form $\Sym^n(\bB_\dR[A])$, that has to be decomposable as a direct sum of 1-dimensional $\bB_\dR$-representations. 

\begin{lemma}\label{KAdec}
	There exists a finite extension $K^\prime$ of $K$ such that the $\bB_\dR$-representation $\bB_\dR[A]$ of $G_{K^\prime}$ decomposes as a direct sum of 1-dimensional representations.
	\end{lemma}

\begin{proof}
We abuse of notation and still write $A$ for an arbitrary element of $C(\ovl K)$ lifting $A\in C(\ovl K/\Z)$. 
Up to implicitly replacing $K$ by $K_{\mu_{r_A}}$ and $A$ by a $G_{K_{\mu_{r_A}}}$-suborbit of it, we can assume that $r_A=1$, hence that $K_\infty[A]\cong K_\infty\otimes_KK_A[\rho_A]$ for $\rho_A$ as in \eqref{rhoa}. Then  $\Sym^n(\bB_\dR[A])\cong\bB_\dR\otimes_{K_A}\Sym_K^n(K_A[\rho_A])$, where we emphasize in the notation that we are taking the symmetric power of $K_A[\rho_A]$ seen as a $K$-linear representation. By \cite[Proposition 3.18]{fontarith}, the $\bB_\dR$-representation $\bB_\dR\otimes_{K_A}\Sym_K^n(K_A[\rho_A]$ of $G_K$ can be decomposed as a direct sum of 1-dimensional representations if and only if the same is true for the $K$-linear representation $\Sym_K^n(K_A[\rho_A])$.

Now if $v$ is an element of the $K$-vector space $K_A$, then $G_K$ acts on the direct factor $K_Av^{\otimes n}$ of $\Sym_K^n(K_A[\rho_A])$ via $\rho_A^{\otimes_{K_A}n}$, that is, the $n$-th power of $\rho_A$ seen as a $K_A^\times$-valued character.  The only elements of $K_A^\times$ that act diagonally on any given $K$-basis of $K_A$ are the elements of $K^\times$, so $K_Av^{\otimes n}$ admits a decomposition into 1-dimensional $K$-factors if and only if $\rho_A^{\otimes_{K_A}n}$ is $K^\times$-valued. With the notations of \eqref{rhoa}, if 
\[ \rho_A^{\otimes_{K_A}n}=\exp(na\log\chi_K^\cyc) \] 
is definable over $K$ then $na\in K$, which implies that $a\in K$, as desired. 
\end{proof}

Thanks to Lemma $\ref{KAdec}$, we can implicitly replace $K$ by a finite extension such that $\bB_\dR[A]$ is 1-dimensional for every $A$ appearing in \eqref{WdRsigma}, that is, every such $A$ is a singleton $\{a\}$ for some $a\in\ovl K/\Z$.

If $a\in\ovl K/\Z$ and $d>0$ then $\Sym^n(\bB_\dR[\{a\};d])$ is never semisimple: it is enough to observe that
\[ \Sym^n(\Z_p\{0;d\})=\Sym^n(\Sym^{d-1}\Z_p\{0;2\}) \]
always contains a direct factor isomorphic to 
\[ \Sym^{2d}\Z_p\{0;2\}\cong\Z_p\{0;2d-1\}, \] 
so that $\Sym^n(\bB_\dR[\{a\};d])$ contains a direct factor isomorphic to $\bB_\dR[\{na\};2d-1]$, which is not semisimple by Fontaine's classification \cite[Théorème 3.19]{fontarith}.

We conclude that $d=1$ whenever $h_{A,d,\sigma}>0$, so that $W_{\dR}^\sigma$ is isomorphic to a direct sum of $1$-dimensional $\bB_\dR$-representations. For each such 1-dimensional factor we pick a generator, and build this way a basis $\{f_{\dR,i}\}_{i=1,\ldots,m}$ of $W_\dR^\sigma$. 

To simplify the notation in the following arguments we will write $f_\dR=f_{\dR,1}$. 
Let $I_\dR$ be the set of $i\in\{1,\ldots,N\}$ such that $W_{i,\dR,\sigma}^{\prime}$ has the same Hodge--Tate--Sen weight as the $\bB_\dR$-representation $\bB_\dR f_\dR$. 
Since our chosen $\sigma$ verifies the quasi-regularity condition of Definition \ref{qreg} for $W$, the tuple of Hodge--Tate--Sen weights of $W_i^\prime$ is independent of $i\in I_\dR$. Choose an arbitrary $i_0\in I_\dR$. By Lemma \ref{modEh} there exists, up to implicitly replacing $E$ and $K$ by a common finite extension, a slope 0 modification $W_{i_0}^\second$ of $W_{i_0}^\prime$. Write $W_{i_0}^\second=W(\delta)$ for a character $\delta\colon G_K\to E^\times$. 

Up to replacing $K$ and $E$ implicitly by a common finite extension when $p\mid n$, we define an $n$-th root $\delta^{1/n}\colon G_K\to E^\times$ of $\delta$ by 
\[ \delta^{1/n}(g)=\exp{\big(\frac{1}{n}\log{\delta}(g)\big)}. \]
Since $W$ and $\Sym^nW$ are potentially triangulable if and only if $W\otimes_EW(\delta^{-1/n})$ and $\Sym^n(W\otimes_EW(\delta^{-1/n}))\cong\Sym^nW\otimes_EW(\delta^{-1})$ are, we can implicitly replace $W$ with $W\otimes_EW(\delta^{-1/n})$, $W_i^\prime$ with $W_i^\prime\otimes_EW^{-1}$ for $1\le i\le N$ and $f_{\dR,i}$ with $f_{\dR,i}\otimes 1$ for $1\le i\le n$, and assume from now on that:
\begin{equation}\label{asfdr}\tag{$\star$} \text{$f_\dR$ is fixed under the action of $G_K$}
\end{equation} 
hence that, for every $i\in I_\dR$, the Hodge--Tate--Sen weights of $W_i^\prime$ all vanish. In particular $f_{\dR}^{\otimes n}$ is also fixed by the action of $G_K$ and, for every $i\in I_\dR$, the rank 1 $\BPKE$-pair $W_i^\prime$ is Hodge--Tate, hence de Rham, hence potentially crystalline by the $p$-adic monodromy theorem for $B$-pairs \cite[Théorème 2.3.5]{berconstr}. By replacing once more $K$ and $E$, implicitly, by a common finite extension, we can further assume that 
\begin{equation}\label{asasfdr}\tag{$\star\star$} \text{$W_i^\prime$ is crystalline for every $i\in I_\dR$.}
	\end{equation}
This means that there exists $d_i\in E\bB_\crys^\sigma$ such that the element $d_i\otimes F_{e,i}$ of $E\bB_\crys^\sigma\otimes_{\bB_{e,E}}\bB_{e,E}F_{e,i}$ is $G_K$-invariant. 

\begin{lemma}\label{exa}
	There exists $a\in\bB_\dR^\sigma$ and $f_\crys^\sigma\in EB^\sigma_\crys\otimes_{\bB_{e,E}}W_e$ such that $f_\dR=a\otimes f_\crys^\sigma$.
	\end{lemma}

\begin{proof}
Let $W_{\dR,0}^{\sigma,\prime}=\bigoplus_{i\in I_\dR}W_{i,\dR}^{\sigma,\prime}$, that is the largest trivial sub-$\bB_\dR$-representation of $W_{\dR}^{\sigma,\prime}$. The elements $d_i\otimes F_{e,i}$, seen inside of $W_{\dR,0}^{\sigma,\prime}$ via the extension of scalars through $E\bB_\crys^\sigma\into\bB_\dR^\sigma$, form a basis of $G_K$-invariant elements of $W_{\dR,0}^{\sigma,\prime}$. In particular, since $(\bB_\dR^\sigma)^{G_K}=K$, they are a $K$-basis of the $K$-vector space of $G_K$-invariant elements in $W_{\dR,0}^{\sigma,\prime}$. In particular the $G_K$-invariant element $f_\dR^{\otimes n}$ can be written as 
\begin{equation*} f_\dR^{\otimes n}=\sum_{i\in I_\dR}k_id_i\otimes F_{e,i}
\end{equation*} 
for some $k_i\in K$. 

Observe that $k_id_i$ is an element of $E\bB_\crys^\sigma$, since $K=E$. This means that $f_\dR^{\otimes n}=1\otimes F_\crys$ for some $F_\crys\in E\bB_\crys^\sigma\otimes_{\bB_{e,E}}(\Sym^nW)_e$, which by Lemma \ref{lemmafn} implies that 
\[ f_\dR=a\otimes f_\crys^\sigma \] 
for an $f_\crys^\sigma\in E\bB_\crys^\sigma\otimes_{\bB_{e,E}}W_e$ and an $a\in\bB^\sigma_\dR$ (satisfying $a^n\in E\bB_\crys^\sigma$).
\end{proof}

Let $W_{\crys,0}^\sigma$ be the smallest $\varphi_E$-stable $E\bB_\crys^\sigma$-submodule of $E\bB_\crys^\sigma\otimes_{\bB_{e,E}}W_e$ containing $f_\crys^\sigma$, and let $h$ be its rank. 
 
 \begin{lemma}\label{crysbasis} 
 The set $\{1\otimes\varphi_E^if_\crys^\sigma\}_{i=0,\ldots,h-1}$ is a basis of $G_K$-eigenvectors for the $\Frac(E\bB^\sigma_\crys)$-representation $\Frac(E\bB^\sigma_\crys)\otimes_{E\bB^\sigma_\crys}W_{\crys,0}^\sigma$.
\end{lemma}
 	
\begin{proof} 
The $\bB_\dR^\sigma$-vector space $W_{\dR,0}^\sigma$ is generated by a finite set of elements of $W_{\crys,0}^\sigma$ of the form $\varphi_E^i(f_\crys^\sigma)$ with $i\in\N$, and since the action of $G_K$ on $\bB^\sigma_\dR\otimes_{E\bB^\sigma_\crys}W_{\crys,0}^\sigma$ fixes $f_\dR=a\otimes f_\crys^\sigma$, stabilizes $W_{\crys,0}^\sigma$ and commutes with $\varphi_E$, one has
\begin{equation}\label{gphi} g.(\varphi_E^i(1\otimes f_\crys^\sigma))=\varphi_E^i(g.(1\otimes f_\crys^\sigma))=\varphi_E^i\left(\frac{a}{g.a}\otimes f_\crys^\sigma\right). 
\end{equation}
Again since $G_K$ stabilizes $W_{\crys,0}^\sigma$, $(g.a)/a$ must belong to $\Frac(E\bB^\sigma_\crys)$: indeed, $\bB^\sigma_\dR$ is $\Frac(E\bB^\sigma_\crys)$-flat (they are both fields) and $af_\cris^\sigma$, $(g.a)f_\cris^\sigma$ generate the same line in $\bB^\sigma_\dR\otimes_{E\bB_{\crys}^\sigma}W_{\crys,0}^\sigma$, hence they generate the same line in $\Frac(E\bB^\sigma_\crys)\otimes_{E\bB_{\crys}^\sigma}W_{\crys,0}^\sigma$. In particular we can apply $\varphi_E$ to $a/g.a$, and by rewriting the rightmost member of \eqref{gphi} we obtain
\begin{equation}\label{fcryseigen} g.(\varphi_E^i(1\otimes f_\crys^\sigma))=\varphi_E^i(g.(1\otimes f_\crys^\sigma))=\varphi_E^i\left(\frac{a}{g.a}\right)\otimes\varphi_E^i(f_\crys^\sigma), 
	\end{equation}
so that $\varphi_E^i(f_\crys^\sigma)$ generates a $G_K$-stable line in $\Frac(E\bB^\sigma_\crys)\otimes_{E\bB_{\crys}^\sigma}W_{\crys,0}^\sigma$. 
\end{proof}

\begin{lemma}
The element $a$ of Lemma \ref{exa} belongs to $\Frac(E\bB^\sigma_\crys)$.
\end{lemma}

\begin{proof}
Assume by contradiction that $a\notin\Frac(E\bB^\sigma_\crys)$. 
By Lemma \ref{crysbasis}, the $\Frac(E\bB^\sigma_\crys)$-representation $\Frac(E\bB^\sigma_\crys)\otimes_{E\bB^\sigma_\crys}W_{\crys,0}^\sigma$ admits a basis of $G_K$-eigenvectors of the form $\{1\otimes\varphi_E^if_\crys^\sigma\}_{i=0,\ldots,h-1}$, where $h$ is the rank of the representation. 
The same argument as in Lemma \ref{crysbasis} gives that $1\otimes\varphi_E^hf_\crys^\sigma$ is also a $G_K$-eigenvector. Write $\varphi_E^hf_\crys^\sigma$ as a $\Frac(E\bB^\sigma_\crys)$-linear combination $\sum_{i=0}^{h-1}\alpha_i(1\otimes\varphi_E^if_\crys^\sigma)$. 
The only way for $\varphi_E^hf_\crys^\sigma$ to be a $G_K$-eigenvector is if the $\Frac(E\bB^\sigma_\crys)^\times$-valued characters giving the action of $G_K$ on $\varphi_E^hf_\crys^\sigma$ and on each of the eigenvectors $\alpha_i(1\otimes\varphi_E^if_\crys^\sigma)$ with $\alpha_i\neq 0$ all coincide. 
By comparing them for $\varphi_E^hf_\crys^\sigma$ and $\alpha_0(1\otimes f_\crys^\sigma)$ via \eqref{fcryseigen} for $i=h$ and $i=0$ (necessarily $\alpha_0\neq 0$ because $\varphi_E$ is an automorphism), we obtain for every $g\in G_K$
\[ \varphi_E^{h}\left(\frac{a}{g.a}\right)=\frac{g.\alpha}{\alpha}. \]
Since $\alpha\in\Frac(E\bB^\sigma_\crys)$, we can write
\[ \frac{a}{g.a}=\frac{g.(\varphi_E^{-h}\alpha)}{\varphi_E^{-h}\alpha}, \]
from which we get that $a\cdot\varphi_E^{-h}\alpha$ is $G_K$-invariant. Therefore $a\cdot\varphi_E^{-h}\alpha$ is an element of $(\Frac(E\bB^\sigma_\crys))^{G_K}=E$, and from $\alpha\in\Frac(E\bB^\sigma_\crys)$ we get $a\in\Frac(E\bB^\sigma_\crys)$, a contradiction.
\end{proof}

Thanks to the lemma, we can replace the basis $\{1\otimes\varphi_E^if_\crys^\sigma\}_{i=0}^{h-1}$ of $\Frac(E\bB^\sigma_\crys)\otimes_{E\bB^\sigma_\crys}W_{\crys,0}^\sigma$ with the basis of $G_K$-fixed elements $\wtl f_i=a\otimes\varphi_E^i(f_\crys^\sigma)$, $i=0,\ldots,h-1$ (this gives in particular that $\Frac(E\bB^\sigma_\crys)\otimes_{E\bB^\sigma_\crys}W_{\crys,0}^\sigma$ is a trivial, hence de Rham, representation). Since the action of $\varphi_E$ commutes with that of $G_K$, it must send a $G_K$-invariant basis to another $G_K$-invariant basis, hence it must be described in the basis $\{\wtl f_i\}_{i=0}^{h-1}$ by a matrix $\Phi$ in $\GL_h(\Frac(E\bB^\sigma_\crys))^{G_K})=\GL_h(K)$. Because of our choice of basis such a matrix will only have as non-zero entries a sub-diagonal of $1$'s and the non-zero entries of the last column, but we do not need this description. Such a matrix will admit a non-zero eigenvector over a finite extension $K^\prime$ of $K$.  

Pick a finite Galois extension $L$ of $\Q_p$ containing $K^\prime$, and let $f_{L/E}$ be the inertia degree of $L/E$. 
Let $\varphi_{L}$ be the operator on 
\[ L\otimes_{E}(\Frac(E\bB^\sigma_\crys)\otimes_{E\bB^\sigma_\crys}W_{\crys,0}^\sigma) \] 
defined by $1\otimes\varphi_E^{f_{L/E}}$, and extend $L$-linearly the action of the subgroup $G_L$ of $G_K$ from $\Frac(E\bB^\sigma_\crys)\otimes_{E\bB^\sigma_\crys}W_{\crys,0}^\sigma$ to $L\otimes_{E}(\Frac(E\bB^\sigma_\crys)\otimes_{E\bB^\sigma_\crys}W_{\crys,0}^\sigma)$. 
Since the matrix $\Phi$ admits an eigenvector defined over $L$, there exists a $\varphi_{L}$-eigenvector 
\[ \wtl f\in L\otimes_{E}(\Frac(E\bB^\sigma_\crys)\otimes_{E\bB^\sigma_\crys}W_{\crys,0}^\sigma) \]
given by an $L$-linear combination of the elements $1\otimes \wtl f_i$, $i=0,\ldots,h-1$. Since the action of $G_L$ is $L$-linear and the elements $1\otimes \wtl f_i$ are $G_K$-invariant, the element $\wtl f$ is $G_L$-invariant. 

We extend both our base and coefficient fields $K$ and $E$ to $L$ via the procedure of Section \ref{fieldext}, in order for the eigenvector to be defined over $\Frac(E\bB^\sigma_\crys)$. 
We make all the replacements implicitly and we keep writing $K$, $E$ and $\sigma$ for the relevant objects. 
We can now assume that $\Frac(E\bB^\sigma_\crys)\otimes_{E\bB^\sigma_\crys}W_{\crys,0}^\sigma$ contains an eigenvector $\wtl f$ for $\varphi_E$ that is also $G_K$-invariant. 

By multiplying $\wtl f$ with a suitable $a_0^{-1}\in E\bB^\sigma_\crys$, we obtain an element $f_0=a_0^{-1}\wtl f$ of $W_{\crys,0}^\sigma$. Since $f_0$ generates a $\varphi_E$- and $G_K$-stable line in $\Frac(E\bB^\sigma_\crys)\otimes_{E\bB^\sigma_\crys}W_{\crys,0}^\sigma$, we can apply Lemma \ref{line} to the group actions of $\varphi_E^\Z$ and $G_K$ and deduce that $f_0$ is an eigenvector for both $\varphi_E$ and $G_K$. 

All the arguments we made starting with $f_\dR$ can be repeated starting with the element $\wtl f=a_0f_0$ instead. 
In particular, we can write $(a_0f_0)^{\otimes n}$ as
\begin{equation}\label{f0} (a_0f_0)^{\otimes n}=\sum_{i\in I_\dR}d_{0,i}\otimes F_{e,i}
\end{equation}
for some $d_{0,i}\in E\bB_{\crys}^\sigma$. 
By applying $\varphi_E$ to both sides of $\eqref{f0}$, 
we get
\begin{equation}\label{f0n} \varphi_E(a_0^n)\varphi_E(f_0)^{\otimes n}=\sum_{i\in I_\dR}\varphi_E(d_{0,i})\otimes \varphi_E(F_{e,i}). 
\end{equation}

Let $\beta\in E\bB_{\crys}^\sigma$ be the $\varphi_E$-eigenvalue of $f_0$. 
The operator $\varphi_E$ acts trivially on $F_{e,i}$ for every $i$ since $F_{e,i}\in\Sym^nW_{e,E}$. This way we deduce from \eqref{f0n} that 
\begin{equation}\label{betaf0n} \varphi_E(a_0^n)\beta^nf_0^{\otimes n}=\sum_{i\in I_\dR}\varphi_E(d_{0,i})\otimes F_{e,i}. 
\end{equation}
Since $\{1\otimes F_{e,i}\}_{i\in I_\dR}$ is a $\bB^\sigma_\dR$-basis of $W_{\dR,0}^{\sigma,\prime}$, comparing \eqref{f0n} and \eqref{betaf0n} we get
\[\varphi_E(d_{0,i})a_0^n=\varphi_E(a_0^n)\beta^nd_{0,i} \] 
for every $i$. In particular $\varphi_E(d_{0,i}d_{0,1}^{-1})=d_{0,i}d_{0,1}^{-1}\in\Frac(E\bB_\crys^\sigma))$ for every $i$. By Lemma \ref{BeEphi}, $d_{0,i}d_{0,1}^{-1}\in\Frac\bB_{e,E}$ for all $i$. Set $b_i=d_{0,i}d_{0,1}^{-1}$ for each $i$, and let $b_0$ be the product of the denominators of all the $b_i$ written as quotients of elements of $\bB_{e,E}$. Then by multiplying \eqref{f0} with $d_{0,1}^{-1}b_0$ we get
\begin{equation}\label{d01} d_{0,1}^{-1}b_0a_0^nf_0^{\otimes n}=\sum_{i\in I_\dR}b_ib_0\otimes F_{e,i}, 
\end{equation}
with $b_ib_0\in\bB_{e,E}$ for every $i$. 
Define an element $F$ of $\Sym^nW_{e}$ by
\[ F=\sum_{i\in I_\dR}b_ib_0F_{e,i}. \]
By \eqref{d01}, there exists $c\in\Frac(E\bB_\cris^\sigma)$ such that $f_0^{\otimes n}=c\otimes\pi^\sigma(F)$. By Lemma \ref{lemmafn}, $f_0$ has to be of the form $c_0\otimes\pi^\sigma(F_0)$ for some $c_0\in\Frac(E\bB_\cris^\sigma)$ and $F_0\in W_e$. Since $f_0^{\otimes n}$ generates a $G_K$-stable line in $\Frac(E\bB_\crys^\sigma)\otimes_{\bB_{e,E}}\Sym^nW_e$, Lemma \ref{linetens} implies that $f_0$ generates a $G_K$-stable line in $\Frac(E\bB_\crys^\sigma)\otimes_{\bB_{e,E}}W_e$, and the same is obviously true for its $\Frac(E\bB_\crys^\sigma)$-scalar multiple $\pi^\sigma(F_0)$. 

We apply Lemma \ref{line} to $R=\bB_{e,E}$, $M=W_e$, $L=\Frac(E\bB_\cris^\sigma)$, using the fact that $\bB_{e,E}$ is a principal ideal domain \cite[Proposition 1.1]{berdimtri}, and we conclude that $W_{e,E}$ contains a $G_K$-stable saturated line $V_e$. 
By setting $V_\dR^+=(\bB_{\dR,E}\otimes_{\bB_{e,E}}V_e)\cap W_\dR^+$, we obtain a saturated sub-$\BPKE$-pair $(V_e,V_\dR^+)$ of rank $1$ of $W$.



\medskip

\section{Lifting strict triangulations}\label{secsttri}

In this section we prove that, under some extra assumptions, the potential triangulability result of Theorem \ref{ptri} can be improved to a triangulability one. 
One of these assumptions is technical: if $\bS^\uu$ and $W$ are as in condition (1) of Theorem \ref{ptri}(i), then we require that $\bS^\uu(W)$ be \emph{strictly triangulable} in the sense of Definition \ref{strictdef} below. We do not know if this condition can be removed. The second assumption, on the other hand, is a necessary and sufficient condition for $W$ to be triangulable: the ordered parameter of a triangulation of $\bS^\uu(W)$ must ``lift'' to a candidate ordered parameter for $W$. 

We give a notion of \emph{strict split triangulinity} for $B$-pairs that is slightly different from that in \cite[Definition 6.3.1]{kedpotxia}. 
Let $K$ and $E$ be two $p$-adic fields, and let $n$ be a positive integer. 

\begin{defin}\label{KEpar}
A \emph{$(K,E)$-parameter (of length $n$)} is a set of $n$ of continuous characters $K^\times\to E^\times$, while an \emph{ordered $(K,E)$-parameter (of length $n$)} an ordered $n$-tuple of continuous characters $K^\times\to E^\times$. 
\end{defin}

We will denote ordered $(K,E)$-parameters by underlined, lower-case Greek letters. Now let $W$ be a $\BPKE$-pair of rank $n$ and $\udelta$ an ordered $(K,E)$-parameter of length $n$.

\begin{rem}\label{ordbor}
Let $T$ be any split maximal torus of $\GL_{n}$. The datum of a $(K,E)$-parameter of length $n$ is the same as that of a continuous homomorphism $K^\times\into T(E)$, while giving an ordered $(K,E)$-parameter amounts to specifying, in addition, a Borel subgroup of $\GL_{n/E}$ containing $T$: indeed, the orderings of the $n$ characters making up a homomorphism $K^\times\to T(E)$ are in bijection with the possible choices of sets of positive roots of $\GL_{n}$.
	\end{rem}

\begin{defin}\label{strictdef}
We say that $W$ is \emph{split triangulable of parameter $\udelta$} if there exists a triangulation of $W$ of ordered parameter $\udelta$. 

We say that a triangulation $\cW$ of $W$ is \emph{strict} if there are no other triangulations of $W$ with the same ordered parameter as $\cW$. We say that $W$ is \emph{strictly split triangulable of parameter $\udelta$} if there exists a strict triangulation of $W$ of ordered parameter $\udelta$.

We call an $E$-linear representation $V$ of $G_K$ \emph{(strictly) split trianguline of parameter $\udelta$} if the associated $\BPKE$-pair $W(V)$ is (strictly) split triangulable of parameter $\udelta$.
\end{defin}

In this section we will not deal with non-split triangulable $B$-pairs. Though one could extend the above definitions to the non-split case, we believe that the study of non-split parameters fits more naturally in the framework of $G$-$B$-pairs that we present in the next section.

\begin{rem}
If $W$ is strictly triangulable of parameter $\udelta$ in the sense of \cite[Definition 6.3.1]{kedpotxia}, then it is strictly split triangulable of ordered parameter $\udelta$ according to our definition. 
We do not know if the converse is true: inside of a triangulable $B$-pair $W$ of some parameter one may have distinct extensions by isomorphic rank $1$ $B$-pairs (as prescribed by the parameter) of some step of the triangulation, so that $W$ is not strictly triangulable, but only be able to complete one of these extensions to a triangulation.
\end{rem}

\begin{rem}
If $\cW=(W_i)_{1\le i\le n}$ is a strict triangulation of a $\BPKE$-pair $W$, then every quotient $W_i/W_j$, $j\leq i$, inherits from $\cW$ a triangulation that is necessarily strict: if it weren't, we would be able to build in an obvious way a triangulation of $W$ distinct from $\cW$ but of the same parameter as $\cW$.
\end{rem}

\begin{rem}
If $V$ is refined trianguline in the sense of \cite[Definition 6.4.1]{kedpotxia}, then $V$ admits a strict triangulation by \cite[Lemma 6.4.2]{kedpotxia}. 
Since our motivation for the result of this paper is the application to the $p$-adic deformation of Galois and automorphic representations, we prefer to work with strictly split trianguline representations rather than with the smaller class of refined ones. Refined triangulinity is too strong a property to work generically over a $p$-adic family of Galois representations: for instance, it implies potential semistability, hence is immediately lost when the Hodge--Tate--Sen weights of a representation are not integers.
\end{rem}

\subsection{Operations on parameters}

Let $V$ be an $n$-dimensional $E$-vector space, $T(V)$ a maximal split torus in $\GL(V)$ and $B(V)$ a Borel subgroup of $\GL(V)$ containing $T(V)$. By \emph{flag} in an $E$-vector space we will always mean a \emph{complete flag}. 
Let $\Fil^\bullet V$ be the flag on $V$ whose stabilizer is $B(V)$. Each $E$-representation $W$ of $\GL(V)$ is equipped with a $B(V)$-stable flag, and such a flag is unique if $W$ is irreducible. One can construct this flag in the natural way: If $W_1,W_2$ are two objects of $\Rep_E(\GL(V))$ equipped with flags $\Fil^\bullet W_1$ and $\Fil^\bullet W_2$, we define a flag on $W_1\otimes_EW_2$ in the natural way, by setting
\[ \Fil^n(W_1\otimes_EW_2)=\bigoplus_{i+j=n}\Fil^iW_1\otimes_E\Fil^jW_2 \]
for every $n\in\Z$. Since $V$ is a tensor generator of $\Rep_E(\GL(V))$, every object of this category can be written as a direct sum of subquotients of $V_{a,b}=V^{\otimes a}\otimes_E(V^{\vee})^{\otimes b}$ for some non-negative integers $a,b$. If $V_{a,b}$ is equipped with a complete $B(V)$-stable flag, every irreducible subquotient of $V_{a,b}$ inherits a unique $B(V)$-stable flag.

For an arbitrary tuple $\uu$, let $B(\bS^\uu(V))$ be the stabilizer of the unique $B(V)$-stable flag on the irreducible representation $\bS^\uu(V)$ of $\GL(V)$. Since such a flag is $B(V)$-stable, the morphism $\bS^\uu\colon\GL(V)\to\GL(\bS^\uu(V))$, restricted to $B(V)$, gives a morphism 
\[ B(V)\to B(\bS^\uu(V)) \]
that we still denote by $\bS^\uu$. 
By restricting this map to the maximal tori $T(V)$ and $T(\bS^\uu(V))$ contained in the two sides, we obtain a morphism
\begin{equation}\label{ST} T(V)\to T(\bS^\uu(V)) 
\end{equation}
that we still denote by $\bS^\uu$. 

\begin{rem}\label{remsem}
In the above construction, we can replace $E$ with an arbitrary ring and $V$ with a free $E$-module, letting $\Rep_E(\GL(V))$ be the category of $E$-linear representations of $\GL(V)$ on finite free $E$-modules. If $B(V)$ is the stabilizer of a flag in $V$, then we can construct a unique $B(V)$-stable flag in $\GL(\bS^\uu(V))$, and the associated morphisms $B(V)\to B(\bS^\uu(V))$ and $T(V)\to T(\bS^\uu(V))$. 
\end{rem}

The following result is probably standard, but we could not find a reference for it.

\begin{lemma}\label{borelpreim}
For every tuple $\uu$ with $\length(\uu)<\dim_EV$, the preimages of $B(\bS^\uu(V))$ and $T(\bS^\uu(V))$ under $\bS^\uu\colon\GL(V)\to\GL(\bS^\uu(V))$ are $B(V)$ and $T(V)$, respectively.
\end{lemma}

\begin{proof}
Let $\cW_V$ (respectively $\cW_{\bS^\uu(V)}$) be the quotient of the normalizer of $T(V)$ (respectively $\bS^\uu(T(V))$) in $\GL(V)$ (respectively $\GL(\bS^\uu(V))$) by its centralizer. 
The morphism $\bS^\uu$ maps $\cW_V$ injectively to $\cW_{\bS^\uu(V)}$. In particular, for every $w\in\cW_V$, the Bruhat cell $B(V)wB(V)$ of $\cW_V$ is mapped to the Bruhat cell $B(\bS^\uu(V))\bS^\uu(w)B(\bS^\uu(V))$ of $\GL_{\bS^\uu(V)}$. Since Bruhat cells are disjoint, the preimage of the Bruhat cell $B(\bS^\uu(V))$ of $\GL_{\bS^\uu(V)}$ is $B(V)$.
\end{proof}

Now let $\udelta$ be an ordered $(K,E)$-parameter of length $n$ and let $T_n$ be a maximal split torus in $\GL_{n/E}$. By Remark \ref{ordbor}, the datum of $\udelta$ corresponds to that of a continuous character $K^\times\to T_n(E)$ and of a choice of a Borel subgroup $B_n$ of $\GL_n$ containing $T_n$. Let $m=\dim_E(\bS^\uu(E^n))$, and pick any basis $e_1,\ldots,e_m$ of $\bS^\uu(E^n)$ in order to attach to $\bS^\uu$ a morphism $\GL_n\to\GL_m$. Let $T_m$ and $B_m$ be the unique torus and Borel subgroup of $\GL_m$ containing $\bS^\uu(T_n)$ and $\bS^\uu(B_n)$, respectively. 

\begin{defin}
We denote by $\bS^\uu(\udelta)$ the ordered $(K,E)$-parameter defined, via Remark \ref{ordbor}, by the homomorphism $\bS^\uu\ccirc\udelta\colon K^\times\to T_m(E)$ and the choice of the Borel subgroup $B_m$ of $\GL_m$. 
\end{defin}

\begin{rem}\label{uetaell}
If $\udelta=(\delta_1,\ldots,\delta_n)$, then every character in $\bS^\uu(\udelta)$ is a monomial of degree $\ell(\uu)$ in the $\delta_i$. 
\end{rem}

Let $L$ be a finite extension of $K$ and denote by $\Nm_{L/K}\colon L^\times\to K^\times$ the norm map. 

\begin{defin}
For every $(K,E)$-parameter $\udelta=(\delta_1,\ldots,\delta_n)$, we define an $(L,E)$-parameter $\udelta_L=(\delta_{1,L},\ldots,\delta_{n,L})$ by setting
\[ \delta_{i,L}=\delta_i\ccirc\Nm_{L/K} \]
for every $i\in\{1,\ldots,n\}$.
\end{defin}

Fix local reciprocity maps $r_K\colon K^\times\to G_K^\ab$, $r_L\colon L^\times\to G_L^\ab$ making the diagram
\begin{center}
\begin{tikzcd}
L^\times\arrow{r}{r_L}\arrow{d}[swap]{\Nm_{L/K}} & G_L^\ab\arrow{d}{\iota_{L/K}} \\
K^\times\arrow{r}{r_K} & G_K^\ab
\end{tikzcd}
\end{center}
commute. 
When $\udelta$ can be extended to a homomorphism $\udelta^\Gal\colon G_K\to T(E)$ (where we identify $K^\times$ to a subgroup of $G_K$ via $r_K$), the restriction of $\delta^\Gal\vert_{G_L}\colon G_L\to T(E)$ to $L^\times$ (via $r_L$) coincides with $\udelta_L$. 

\begin{rem}\label{trirest}
For a character $\delta\colon K^\times\to E^\times$, we defined $\delta_L$ in such a way that the restriction of the $\BPKE$-pair $W(\delta)$ to $G_L$ is $W(\delta_L)$. If $W$ is a $\BPKE$-pair equipped with a triangulation of ordered parameter $\delta$, then the same triangulation is a triangulation of $W\vert_{G_L}$ of ordered parameter $\udelta_L$.
\end{rem}

\subsection{Lifting}

Under a strict split triangulinity assumption, we can improve Theorem \ref{ptri} by combining it with the following. 

\begin{thm}\label{thsttri}
Let $L$ a finite extension of $K$, and $\udelta$ and $\ueta$ be two ordered $(K,E)$-parameters of length $n$. 
Let $W$ be a $\BPKE$-pair of rank $n$, and let $\uu$ be a tuple satisfying $\length(\uu)<n$. .
\begin{enumerate}[label=(\roman*)]
\item If $W$ is triangulable of parameter $\udelta$, then $\bS^\uu(W)$ is triangulable of ordered parameter $\bS^\uu(\udelta)$. If in addition $\bS^\uu(W)$ is strictly split triangulable of ordered parameter $\bS^\uu(\udelta)$, then $W$ is strictly split triangulable of ordered parameter $\udelta$. 
\item If:
\begin{itemize}
\item $\length(\uu)<n$,
\item $\bS^\uu(W)$ is triangulable,
\item $L$ is a finite extension of $K$ such that $W$ is triangulable of ordered parameter $\ueta$, and
\item $\bS^\uu(W\vert_{G_L})$ is strictly triangulable of ordered parameter $\bS^\uu(\ueta)$,
\end{itemize}
then there exists a unique triangulation of  $W$ with the following property: the ordered parameter $\unu$ of $W$ satisfies $\bS^\uu(\unu_L)=\bS^\uu(\ueta)$. In particular, such a triangulation is strict. 
\end{enumerate}
\end{thm}

We clarify the meaning of point (ii). If we only assume that $\bS^\uu(W)$ is triangulable for some $\uu$ with $\length(\uu)<n$, then $W$ is potentially triangulable by Theorem \ref{ptri}(i). Take $L$ to be an extension of $K$ such that $W\vert_{G_L}$ is triangulable, and implicitly extend scalars to a finite extension of $E$ to assume that $W\vert_{G_L}$ is split triangulable. Let $\udelta$ be the ordered parameter of a triangulation of $W\vert_{G_L}$. Then part (i) implies that $\bS^\uu(W\vert_{G_L})$ admits a triangulation of ordered parameter $\bS^\uu(\udelta_L)$. The content of statement (ii) is that, if this triangulation is strict, then $W$, and not just its restriction to $G_L$, is strictly split triangulable. 

Note that the final equality $\bS^\uu(\unu_L)=\bS^\uu(\ueta)$ is equivalent to $\unu_L\ueta^{-1}$ taking values in the subgroup of $\ell(\uu)$-roots of unity of $\GL_n(E)$.

\begin{rem}\label{nostr?}
	We do not know whether the converse of the second statement in (i) holds: if the triangulation of $W$ of ordered parameter $\udelta$ is strict, is a triangulation of $\bS^\uu(W)$ of ordered parameter $\bS^\uu(\udelta)$ strict?
	\end{rem}

\begin{proof}[Proof of Theorem \ref{thsttri}] 
If $W$ admits a triangulation of ordered parameter $\udelta$, then by Remark \ref{remsem} $\bS^\uu(W)$ admits a triangulation of ordered parameter $\bS^\uu(\udelta)$.

As for the second statement of (i), if $W$ admits two distinct triangulation of ordered parameter $\udelta$ then the two resulting triangulations of $\bS^\uu(W)$ of ordered parameter $\bS^\uu(\udelta)$ will be distinct, hence $\bS^\uu(W)$ will not be strictly triangulable of this ordered parameter.
 
We now prove part (ii). 
Let $\uu$ be as in the statement. 
As in Remark \ref{remsem}, we write $\GL(W_e)$ for the group of $\bB_{e,E}$-linear automorphisms of $W_e$, and we use the analogous notation for $\GL(W_\dR^+)$. We let $\GL(W_e)$ and $\GL(W_\dR^+)$ act on $\GL(\bS^\uu W_e)$ and $\GL(\bS^\uu(W_\dR^+))$, respectively, via $\bS^\uu$. 

Let $\cW$ be a triangulation of $W\vert_{G_L}$; it consists of compatible triangulations (that is, complete flags) $\cW_e$ and $\cW_\dR^+$ of $W_e\vert_{G_L}$ and $W_\dR^+\vert_{G_L}$, respectively. By Remark \ref{remsem}, if $B(W_e)$ is the group of $\bB_{e,E}$-linear automorphisms of $W_e$ leaving $\cW_e$ stable, we can choose a unique complete $\bB_{e,E}$-flag $\cW_e^\prime$ in $\bS^\uu(W_e)$ that is stable under the action of $B(W_e)$. As before, we write $B(\bS^\uu(W_e))$ for the stabilizer of such a flag. 

Since the action of $G_K$ on $\bS^\uu(W_e)$ factors through its action on $W_e$, and $G_L$ leaves $\cW$ stable, the flag $\cW_e^\prime$ is a triangulation of $\bS^\uu(W_e\vert_{G_L})$, and by part (i) of this theorem we know that the ordered parameter of this triangulation is $\udelta_L$. By assumption $\bS^\uu(W_e)$ admits a ($G_K$-)triangulation $\cW^\second$ of ordered parameter $\udelta$, that is also a triangulation of $\bS^\uu(W_e)$ of ordered parameter $\udelta_L$ by Remark \ref{trirest}. Since $\bS^\uu(W_e\vert_{G_L})$ is strictly triangulable of ordered parameter $\udelta_L$ by hypothesis, we must have $\cW^\prime=\cW^\second$. This means that $\cW^\prime$ is a triangulation of $\bS^\uu(W_e)$, in other words that the action of $G_K$ on $\bS^\uu(W_e)$ factors through the stabilizer $B(\bS^\uu(W_e))$ of $\cW^\prime$, Lemma \ref{borelpreim} implies that the action of $G_K$ on $W_e$ factors through $B(W_e)$, that is, $\cW_e$ is a triangulation of $W_e$. By a completely analogous argument we obtain that $\cW_\dR^+$ is a triangulation of $W_\dR^+$, hence that $\cW$ is a triangulation of $W$. 

If $\unu$ is the ordered parameter of $\cW$, then by part (i) the ordered parameter of $\cW^\prime$ is $\bS^\uu(\unu)$. Since the ordered parameter of $\cW^\prime\vert_{G_L}$ is $\bS^\uu(\ueta)$, we deduce that
\begin{equation}\label{Sunu} \bS^\uu(\unu_L)=\bS^\uu(\ueta).
	\end{equation} 
The uniqueness statement follows from the fact that a different triangulation of $W$ of parameter $\unu$ satisfying \eqref{Sunu} would give rise to a new triangulation of $\bS^\uu(W\vert_{G_L})$ of parameter $\ueta$, contradicting the strictness of our original triangulation of $\bS^\uu(W\vert_{G_L})$.
\end{proof}

Berger and Di Matteo \cite[After Remark 5.6]{berdimtri} give an example of a $2$-dimensional, non-trianguline $\Q_p(\sqrt{-1})$-linear representation $V$ of $G_{\Q_p}$ such that $V\otimes_{\Q_p(\sqrt{-1})}V$, hence $\Sym^2V$, is trianguline. One can check that the triangulation of $\Sym^2V$ obtained in their example is strict, but its ordered parameter is not of the form $\Sym^2\udelta$ for a $2$-dimensional parameter $\udelta$. Therefore $V$ does not satisfy the assumptions of Theorem \ref{thsttri}.

\medskip

\section{Lifting $G$-trianguline representations along isogenies}\label{seclift}

We give a global application of our triangulability result, by proving an analogue of a classical result of Wintenberger about lifting geometric representations \cites[Théorème 1.1.3]{winteniso}[Théorème 2.2.2]{wintentann}. We replace the $p$-adic Hodge-theoretic conditions in his results (Hodge--Tate, de Rham, semistable, crystalline) with triangulinity. Our lifting condition for the parameter of a triangulation turns out to be the exact analogue of his lifting condition for the Hodge--Tate cocharacter.

Let $F$ be a number field, $E$ a $p$-adic field, and let $H$ be a quasisplit reductive group scheme over $E$. Pick a place $v$ of $F$. To our knowledge, there is no accepted definition of what it means for a continuous local Galois representation
\[ \rho_v\colon G_{F_v}\to H(E) \]
to be trianguline. We propose below a definition of strict triangulinity for such a $\rho_v$, modeled on the definitions in \cite{winteniso} of $\rho_v$ having various $p$-adic Hodge theoretic properties.

\subsection{$G$-trianguline representations and their parameters}

We rewrite Daruvar's definition of $G$-triangulable $(\varphi,\Gamma)$-modules \cite{daruvarth} in the context of $B$-pairs, though we only allow for our coefficients to be a field instead of an affinoid algebra as in \emph{loc. cit.}; this will be enough for our purpose. We also propose a simple extension of the definition to the case of quasisplit $G$. 
We warn the reader that we call \emph{split} $G$-triangulable $B$-pairs what Daruvar calls $G$-triangulable $B$-pairs.

Let $K$ and $E$ be two $p$-adic fields. 
Following \cite[Definition 2.2]{daruvarth}, we say that a functor $\calC\to\calD$ between two $E$-linear tensor categories is a \emph{fiber functor} if it is an $E$-linear, exact, faithful tensor functor. 

Let $G$ be a quasisplit reductive group over $E$. 
Let $(B,T)$ be a ``Borel pair" consisting of a maximal torus $T$ of $G$ and a Borel subgroup $B$ of $G$ containing $T$ (with both $T$ and $B$ defined over $E$). We denote by $\res^G_B$ the fiber functor $\Rep_E(G)\to\Rep_E(B)$ obtained by restricting representations of $G$ to $B$. 
The following definition is obtained by allowing $G$ to be quasisplit in \cite[Definition 4.9]{daruvarth}.

We denote by $\BPKEPairs$ the category of $\BPKE$-pairs, introduced in Section \ref{galois}.

\begin{defin}\label{deftridar}
	A \emph{$G$-$\BPKE$-pair} is a fiber functor
	\[ \Rep_E(G)\to\BPKEPairs. \]
	We say that a $G$-$\BPKE$-pair $W\colon\Rep_E(G)\to\BPKEPairs$ is:
	\begin{itemize}
		\item \emph{split triangulable} if there exists a fiber functor $W_B\colon \Rep_E(B)\to G$-$\BPKEPairs$ such that $W=W_B\ccirc\res^G_B$, in which case we call any such $W_B$ a \emph{triangulation} of $W$; 
		\item \emph{triangulable} if there exists a finite extension $F$ of $E$ such that $G\times_EF$-$\BPKF$-pair $W\otimes_EF$ is triangulable.
	\end{itemize}
We say that two triangulations $W_B$ and $W_B^\prime$ are equivalent if they can be obtained from one another by composition with an equivalence of categories $\Rep_E(B)\to\Rep_E(B)$. When we say that a triangulation with certain properties is unique, we always mean unique up to equivalence.
\end{defin}

To any $\BPKE$-pair $W$ of rank $n$, we attach the $\GL_{n/E}$-$\BPKE$-pair defined as the unique fiber functor $\Rep_E(\GL_{n/E})\to\BPKEPairs$ that maps the standard representation to $W$.

\begin{rem}\label{borelind}
	As is the case for \cite[Definition 4.9]{daruvarth}, Definition \ref{deftridar} is independent of the chosen Borel subgroup $B$ of $G$: since all Borel subgroups of $G$ are $G(E)$-conjugate to one another, their categories of $E$-representations are all equivalent.
\end{rem}

As usual, we will say that a $\GBPKE$-pair $W$ has property $\mathbf P$ if there exists a finite extension $K^\prime$ of $K$ such that the $G$-$B\vert_E^{\otimes K^\prime}$-pair $W\vert_{G_{K^\prime}}$ has property $\mathbf P$.

\begin{rem}\label{trigln}
	It follows from \cite[Example 3.11]{daruvarth} that Definition \ref{deftridar} is compatible with the definition of split triangulable $\BPKE$-pair: a $\BPKE$-pair $W$ of rank $n$ is split triangulable if and only if its associated $\GL_n$-$\BPKE$-pair $\wtl W$ is split triangulable. 
	
	More precisely, for every $i\in\{1,\ldots,n\}$ let $V_i$ be an $i$-dimensional representation of $B$ whose image is a Borel subgroup of $\GL(V_i)$; it is unique up to isomorphism. Then to every triangulation 
	\[ 0=W_0\subset W_1\subset\ldots\subset W_n=W \]
	of $W$, one can attach the unique triangulation $(B,W_B)$ of $\wtl W$ that maps $V_i$ to $W_i$. One checks easily that this defines a bijection between split triangulations of $W$ and $\wtl W$.
\end{rem}

To a continuous representation $\rho\colon G_K\to G(E)$ we can attach a $\GBPKE$-pair $W(\rho)$: it is the fiber functor $\Rep_E(G)\to\BPKEPairs$ that maps a representation $S\colon G\to\GL_n(E)$ to the $\BPKE$-pair associated with the $n$-dimensional representation $S\ccirc\rho\colon G_K\to\GL_n(E)$. 
We say that $\rho$ is (potentially, split, potentially split) trianguline if $W(\rho)$ is (potentially, split, potentially split) triangulable. For $G=\GL_n$, this notion agrees with the usual one of trianguline representation by Remark \ref{trigln}.

\subsection{Parameters of $G$-$B$-pairs}

We extend Daruvar's definition of parameter of a triangulation of a $\GBPKE$-pair to the case of quasisplit $G$. 
The next definition is inspired by Daruvar's notion of parameter of a $G$-$\BPKE$-pair. 
Let $G$ be a quasisplit reductive $E$-group, let $B$ be a Borel subgroup of $G$ and $T$ a maximal torus of $G$ contained in $B$.
	
\begin{defin}
	A \emph{$T$-parameter} is a fiber functor $\Rep_E(T)\to\BPKEPairs$. A \emph{$B$-parameter} is a fiber functor $\Rep_E(B)\to\BPKEPairs$ that factors through the restriction functor $\Rep_E(B)\to\Rep_E(T)$.
	\end{defin}

The distinction between $T$- and $B$-parameters is reminiscent of Remark \ref{ordbor}, with $B$-parameter being the analogue of \emph{ordered} $(K,E)$-parameters. This resemblance will be made into a precise relation after Definition \ref{Wpar}.

We denote $B$-parameters by non-underlined lowercase Greek letters in order to distinguish them from $(K,E)$-parameters, that we write as underlined lowercase Greek letters.

Let $W$ be a $\GBPKE$-pair and $W_B\colon\Rep_E(B)\to\BPKEPairs$ be a triangulation of $W$.

\begin{defin}[{cf. \cite[Definition 4.9]{daruvarth}}]\label{Wpar} 
	The \emph{$T$-parameter} of the triangulation $W_B$ is the fiber functor $\Rep_E(T)\to\BPKEPairs$ defined by pre-composing $W_B$ with the fiber functor $\Rep_E(T)\to\Rep_E(B)$ defined as pre-composition with the projection $B\onto T$. 
	
	The \emph{$B$-parameter} of $W_B$ is the fiber functor $\delta_{W_B}\colon\Rep_E(B)\to\BPKEPairs$ obtained by pre-composing the $T$-parameter of $W_B$ with the restriction functor $\Rep_E(B)\to\Rep_E(T)$.
	
	We say that $W_B$ is a \emph{strict triangulation} if it is the only triangulation of $W$ with $B$-parameter $\delta_{W_B}$.
\end{defin}


\subsubsection{From (ordered) $(K,E)$-parameters to ($B$-) $T$-parameters}

Let $n$ be a positive integer and $T$ a split $n$-dimensional torus over $E$. 
The datum of a $(K,E)$-parameter of length $n$ is equivalent to that of a continuous homomorphism $K^\times\to T(E)$. 
By specializing \cite[Example 3.13]{daruvarth} to the case when $X$ is a point, we obtain a bijection between the fiber functors $\Rep_E(T)\to\BPKEPairs$ and the continuous homomorphisms $T^\vee(K)\to E^\times$. Observe that
\begin{align*}
	\Hom_\cont(T^\vee(K),E)&=\Hom_\cont(X^\ast(T^\vee)\otimes_\Z K^\times,E)=\Hom_\cont(K^\times,X^\ast(T^\vee)\otimes_\Z E^\times)= \\
	&=\Hom_\cont(K^\times,\Hom_\Z(X^\ast(T),E^\times))=\Hom_\cont(K^\times,T(E)), 
	\end{align*}
so that elements of $\Hom_\cont(T^\vee(K),E)$ are in bijection with $(K,E)$-parameters. 
By composing the two bijections we obtain a bijection between $(K,E)$-parameters and $T$-parameters.

Now let $\udelta$ be an \emph{ordered} $(K,E)$-parameter. By Remark \ref{ordbor}, $\udelta$ is determined by its corresponding unordered $(K,E)$-parameter together with a choice of a Borel subgroup $B$ of $\GL_n$ containing $T$. 
To $\udelta$ we attach an ordered $B$-parameter $\delta\colon K^\times\to\Rep_E(B)\to\BPKEPairs$ as follows: we start with the $T$-parameter $\Rep_E(T^\vee)\to\BPKEPairs$ associated in the previous paragraph with the unordered $(K,E)$-parameter underlying $\udelta$, and we pre-compose it with the restriction functor associated with the embedding $T(E)\subset B(E)$. We obtain this way a bijection between ordered $(K,E)$-parameters and $B$-parameters.

When speaking of the ($B$-) or $T$-parameter \emph{associated} with a given (ordered) $(K,E)$-parameter, and vice versa, we refer to that given by the bijections we just defined.

\begin{rem}\label{triequiv}
Let $W$ be a $\BPKE$-pair of rank $n$, and let $\wtl W$ be the associated $\GL_n$-$\BPKE$-pair. 
The bijection of Remark \ref{trigln} maps triangulations of $W$ of ordered parameter $\udelta\colon K^\times\to T(E)$ to triangulations of $\wtl W$ of the associated $T$-parameter $\delta$. 
In particular $W$ is (strictly) split triangulable of ordered parameter $\udelta$ if and only if $\wtl W$ is (strictly) split triangulable of $B$-parameter $\delta$.
	\end{rem}

Now let $G$ and $H$ be two quasisplit reductive $E$-groups, $B_G$ a Borel subgroup of $G$ and $T_G$ a maximal torus inside of $B_G$, $S\colon G\to H$ a morphism, $T_H$ a torus of $H$ containing $S(T_G)$ and $B_G$ a Borel subgroup of $H$ containing $S(B_G)$. We keep writing $S$ for the functors $\Rep_E(B_H)\to\Rep_E(B_G)$ and $\Rep_E(T_H)\to\Rep_E(T_G)$ defined by pre-composition with $S$.

\begin{defin}
Given a fiber functor $F$ out of either $\Rep_E(G)$, $\Rep_E(B_G)$ or $\Rep_E(T_G)$, we write $S(F)$ for the functor out of $\Rep_E(H)$, $\Rep_E(B_H)$ or $\Rep_E(T_H)$, respectively, obtained by pre-composing $F$ with $S$. In particular
\begin{itemize}
\item for every $G$-$\BPKE$-pair $W\colon\Rep_E(G)\to\BPKEPairs$, we obtain a $H$-$\BPKE$-pair $S(W)$,
\item for every $T_G$-parameter $\delta\colon\Rep_E(T_G)\to\BPKEPairs$, we obtain a $T_H$-parameter $S(\delta)$,
\item for every $B_G$-parameter $\delta\colon\Rep_E(B_G)\to\BPKEPairs$, we obtain a $B_H$-parameter $S(\delta)$.
\end{itemize}
\end{defin}

\begin{rem}\label{SWtri}
Let $W$ be a $G$-$\BPKE$-pair and $\delta$ a ($T$-) $B$-parameter. If $W_B$ is a triangulation of $W$ of ($T$-) $B$-parameter $\delta$ then $S(W_B)$ is a triangulation of $S(W)$ of ($T$-) $B$-parameter $S(\delta)$.

As was the case for $\BPKE$-pairs (see Remark \ref{nostr?}), we do not know if $S(W_B)$ is a strict triangulation of $B$-parameter $S(\delta)$ whenever $W_B$ is a strict triangulation of $B$-parameter $\delta$.
\end{rem}

\begin{defin}
We say that a $G$-$\BPKE$-pair $W$ is \emph{quasi-regular} if there exists a faithful representation $S\colon G\to\GL_n$ such that the $\BPKE$-pair attached to the $\GL_n$-$\BPKE$-pair $S(W)$ is quasi-regular in the sense of Definition \ref{qreg}.	
\end{defin}

\subsubsection{From triangulable $B$-pairs to triangulable $G$-$B$-pairs}

The next proposition, combined with Remark \ref{triequiv}, relates the triangulability of a $G$-$\BPKE$-pair to the triangulability of a $\BPKE$-pair in the classical sense. In particular, it shows that Daruvar's definition of triangulable $G$-$\BPKE$-pair can be reformulated along the lines of Wintenberger's definitions of the $p$-adic Hodge theoretic properties of $G(E)$-valued representations \cite[Définition 1.1.1]{winteniso}.

\begin{prop}\label{Sstrict}\mbox{ }
	\begin{enumerate}
\item The $G$-$\BPKE$-pair $W$ is triangulable if and only if there exists a faithful $E$-representation $S\colon G\to\GL(V)$ such that the $\GL(V)$-$\BPKE$-pair $S(W)$ is triangulable. Moreover, for any triangulation $W_{B(V)}$ of $S(W)$, there exists a triangulation $W_B$ of $W$ such that $S(W_B)\cong W_{B(V)}$.
\item If $S\colon G\to\GL(V)$ is a faithful $E$-representation of $G$ such that $S(W)$ is strictly triangulable of some $B(V)$-parameter $\nu$, then there exists a unique triangulation $W_B$ of $W$, of some $B$-parameter $\mu$. such that $S(\mu)=\nu$. In particular, such a $W_B$ is strict.
\end{enumerate}
	\end{prop}

In proving Proposition \ref{Sstrict}, we will rely on the following lemma from category theory. As in \cite[\S 2.1]{brandbicat}, let $\cat_{\otimes/E}$ be the 2-category of essentially small $E$-linear tensor categories with $E$-linear tensor functors as morphisms. 
By \cite[Corollary 4.17]{brandbicat}, $\cat_{\otimes/E}$ has bicategorical pushouts. We compute such a pushout in the simple case of a diagram of neutral Tannakian categories.

\begin{lemma}\label{pushout}
	The pushout of the diagram 
	\begin{equation}\label{diagpush}
		\begin{tikzcd}
			\Rep_E(H)\arrow{r}{\alpha_1}\arrow{d}{\alpha_2}&\Rep_E(H_1) \\
			\Rep_E(H_2) & 
		\end{tikzcd}
	\end{equation}
	in $\cat_{\otimes/E}$ is isomorphic to $\Rep_E(H_1\times_HH_2)$.
\end{lemma}

\begin{proof}
	Let $\calP$ be the pushout of \eqref{diagpush}. We first observe that $\calP$ is a neutral Tannakian category. Let $F_1$ and $F_2$ be the forgetful fiber functors of $\Rep_E(H_1)$ and $\Rep_E(H_2)$, respectively; after composition with $\alpha_1$ and $\alpha_2$, respectively, they agree with the forgetful fiber functor on $\Rep_E(H)$, hence they factor through the functors $\Rep_E(H_1)\to\calP$ and $\Rep_E(H_2)\to\calP$ attached to $\calP$. The functor $\calP\to\Vect_E$ appearing in these factorizations is exact because $F_1$ is, and being a tensor functor it is also faithful. Therefore it is a fiber functor on $\calP$.
	
	Write $H_0$ for the fundamental group of $\calP$. The functors in the pushout diagram are induced, via Tannakian duality, by the morphisms in a commutative diagram of E-group schemes
	\begin{center}
		\begin{tikzcd}	
			H & H_1\arrow{l} \\
			H_2\arrow{u} & H_0\arrow{l}\arrow{u} 
		\end{tikzcd}
	\end{center}
	Since the diagram is commutative, the resulting morphism $H_0\to H$ must factor through the morphism $H_1\times_HH_2\to H$ attached to the fiber product. By Tannakian duality, such a factorization provides us with a functor $\beta\colon\Rep_E(H_1\times_HH_2)\to\Rep_E(H_0)$. 
	
	Now consider the commutative diagram of tensor categories
	\begin{center}
		\begin{tikzcd}
			\Rep_E(H)\arrow{r}{\alpha_1}\arrow{d}{\alpha_2}&\Rep_E(H_1)\arrow{d}{\iota_1} \\
			\Rep_E(H_2)\arrow{r}{\iota_2} & \Rep_E(H_1\times_HH_2)
		\end{tikzcd}
	\end{center}
	where $\iota_i$, $i=1,2$, is induced by the morphism $H_1\times_HH_2\to H_i$ attached to the fiber product. 
	By the universal property of  $\calP$ the functors $\iota_i$, $i=1,2$, factor through the functors $\Rep_E(H_i)\to\calP$ attached to the pushout. Such a factorization provides us with a functor $\gamma\colon\calP\to\Rep_E(H_1\times_HH_2)$. 
	
	The functor  $\beta\ccirc\gamma\colon\calP\to\calP$ is naturally isomorphic to the identity because of the universal property of $\calP$, hence induces via Tannakian duality an isomorphism $(\beta\ccirc\gamma)^\ast=\gamma^\ast\ccirc\beta^\ast\colon H_0\to H_0$. On the other hand, $(\gamma\ccirc\beta)^\ast\colon H_1\times_HH_2\to H_1\times_HH_2$ is an isomorphism because of the universal property of the fiber product. We conclude that $\beta^\ast$ and $\gamma^\ast$ are isomorphisms, hence that the categories $\calP$ and $\Rep_E(H_1\times_HH_2)$ are equivalent.
\end{proof}

\begin{proof}[Proof of Proposition \ref{Sstrict}]
We prove part (i).
Let $B$ be a Borel subgroup of $G$, and let $S\colon G\to\GL(V)$ be a faithful representation of $G$ as in the statement. 
The ``only if" is given by Remark \ref{SWtri}. 
Let $B(V)$ be a Borel subgroup of $\GL(V)$  and let $W_{B(V)}$ be a triangulation of $S(W)$, so that we have a diagram of $E$-linear tensor categories
\begin{equation}\label{diagtri}
	\begin{tikzcd}[column sep={3.5cm,between origins}]
		\Rep_E(\GL(V))\arrow{d}\arrow{r}\arrow[bend left=20]{rr}{S(W)} & \Rep_E(B(V))\arrow{d}\arrow[swap]{r}{W_{B(V)}} & \BPKEPairs \\		
		\Rep_E(G)\arrow{r}\arrow[bend right=33,swap]{urr}{W} & \Rep_E(B) & \\
	\end{tikzcd}
	\end{equation}
where all of the arrows in the square on the left are restriction functors. 
By Lemma \ref{pushout}, the pushout of the top left corner $\Rep_(B(V))\leftarrow\Rep_E(\GL(V))\to\Rep_E(G)$ in $\cat_{\otimes/E}$ is equivalent to $\Rep_E(S(B))$, since $B(V)\times_{\GL(V)}G\cong S(B)$.  Since $S$ is faithful, composition with $S$ induces an equivalence of categories $\Rep_E(S(B))\cong\Rep_E(B)$. In particular, the square on the left of \eqref{diagtri} is a pushout. Since the functors $W_{B(V)}$ and $W_G$ agree after composition with $\res_{\GL(V)}^{B(V)}$ and $\res_{\GL(V)}^G$, respectively, they must both factor through a (tensor) functor $W_B\colon\Rep_E(B)$, which gives a triangulation of $W$ satisfying $S(W_B)\cong W_{B(V)}$.

We prove part (ii). Let $S\colon G\to\GL(V)$ and $\nu$ be as in the statement, and let $W_{B(V)}$ be a triangulation of $S(W)$ of $B(V)$-parameter $\nu$. By part (i), there exists a triangulation $W_B$ of $W$ such that $S(W_B)\cong W_{B(V)}$. If $\mu$ is the $B$-parameter of $W_B$, then $S(\mu)=\nu$. As in the proof of part (i), $W_B$ factors through a functor $W_{B,0}\colon\Rep_E(B)\to\calP$, where $\calP$ is the pushout of $\Rep_E(B(V))\leftarrow\Rep_E(\GL(V))\to\Rep_E(G)$ in $\cat_{\otimes/E}$. Write again $H_0$ for the fundamental group of $\calP$, so that $W_{B,0}$ is induced by a morphism of $E$-group schemes $W_{B,0}^\ast\colon H_0\to B$. Now assume that a second triangulation $\wtl W_B$, of some $B$-parameter $\mu^\prime$ satisfying $S(\mu^\prime)=S(\mu)$. Let $\wtl W_{B,0}^\ast\colon H_0\to B$ be the morphism of $E$-group schemes attached to this second triangulation. By the strictness assumption, $S(W)$ admits a unique triangulation of $B(V)$-parameter $S(\mu)$, hence the triangulations $S(W_B)$ and $S(\wtl W_B)$ must be equivalent. This means that the morphisms $S\ccirc W_{B,0}^\ast$ and $S\ccirc\wtl W_{B,0}^\ast$ coincide. Since $S$ is faithful, this is only possible if $W_{B,0}^\ast\cong\wtl W_{B,0}^\ast$, which means that the triangulations $W_B$ and $\wtl W_B$ are equivalent, as desired.
\end{proof}

\subsection{Global lifting}

Let $H, H^\prime$ be two quasisplit connected reductive $E$-group schemes, and let $\pi\colon H^\prime\to H$ be a central isogeny over $E$. 
Given a continuous representation $\rho\colon G_F\to H(E)$ with some prescribed local properties, one can investigate whether there exists a representation $\rho^\prime\colon G_F\to H^\prime(E)$, with the same local properties, that ``lifts'' $\rho$, in the sense that $\pi\ccirc\rho^\prime\cong\rho$. 
When the required local properties are:
\begin{enumerate}
\item unramifiedness outside of a finite set of places containing the places above $p$;
\item a $p$-adic Hodge theoretic property at $p$, taken from the set \{Hodge Tate, de Rham, semistable, crystalline\};
\end{enumerate}
the lifting problem has been studied by Wintenberger \cite{winteniso,wintentann}, Hoang Duc \cite{hoangth}, and Conrad \cite{conradlift}. 
Furthermore, Hoang Duc and Conrad concern themselves the problem of minimizing the set of ramification primes of the lift. 

In this section we study the analogue of the problem described above when (ii) is replaced by the property that $\rho$ is strictly trianguline at $p$. For the existence and ramification locus of a lift we rely on the results of Conrad; our work comes in when trying to prove that the lift is trianguline at the right places. 

We introduce some terminology to be used in the statement of the next results. Given a Borel subgroup $B$ of a quasisplit reductive $E$-group $H$, a maximal torus $T$ of $H$ contained in $B$, and two $T$-parameters $\delta_1,\delta_2\colon\Rep_E(T)\to\BPKEPairs$, we define their product as follows: as in \cite[Example 3.13]{daruvarth}, one observes that fiber functors $\Rep_E(T)\to\BPKEPairs$ are in bijection with cocharacters $K^\times\to T^\vee(E)$, where $T^\vee$ is the dual torus of $T$. We define $\delta_1\delta_2$ as the fiber functor $\Rep_E(T)\to\BPKEPairs$ associated with the product of the two cocharacters associated with $\delta_1$ and $\delta_2$. If instead $\delta_1$ and $\delta_2$ are two $B$-parameters $\Rep_E(B)\to\BPKEPairs$, we define their product as the product of the corresponding $T$-parameters composed with the restriction functor $\Rep_E(B)\to\Rep_E(T)$.

Let $H$ and $H^\prime$ be two quasisplit connected reductive groups over $E$, and let $\pi\colon H^\prime\to H$ be a central isogeny. We denote by $Z$ the kernel of $\pi$ and by $q$ the exponent of $Z$. As usual, we denote by $\mu_q$ the $E$-group of $q$-th roots of unity. 
Let $\rho\colon G_{F_v}\to H(E)$ be a continuous representation, and write $\Sigma_1$ for the set of places of $F$ that are either archimedean or ramified for $\rho$, and $\Sigma_2$ for an arbitrary subset of the set of $p$-adic places of $L$. 
By combining \cite[Theorem 5.5]{conradlift} with Theorems \ref{ptri} and \ref{thsttri} we obtain the following.

\begin{thm}\label{liftisogtri}
Assume that:
\begin{enumerate}[label=(\roman*)]
\item $(F,\varnothing,q)$ is not in the \emph{special case (for the Grunwald--Wang theorem)} described in \cite[Definition A.1]{conradlift};
\item $\Sigma_1$ is finite;
\item for every $v\in\Sigma_1$, the representation $\rho\vert_{G_{F_v}}$ admits a lift $\rho_v^\prime\colon G_{F_v}\to H^\prime(E)$;
\item for every $v\in\Sigma_2$, there exist:
\begin{enumerate}[label=(\arabic*)]
\item a Borel subgroup $B_v$ of $H$ and a maximal torus $T_v$ contained in $B_v$,
\item a Borel subgroup $B_v^\prime$ of $H^\prime$ and a maximal torus $T_v^\prime$ contained in $B_v^\prime$, such that $\pi(B_v^\prime)\subset B_v$ and $\pi(T_v^\prime)\subset T_v$,
\item a $B_v$-parameter $\delta_v\colon\Rep_E(B_v)\to\BPKEPairs$ such that the representation $\rho\vert_{G_{F_v}}$ is quasi-regular and strictly trianguline of $B_v$-parameter $\delta_v$, 
\item a $B_v^\prime$-parameter $\delta_v^\prime\colon\Rep_E(B_v^\prime)\to\BPKEPairs$ such that $\pi(\delta_v^\prime)=\delta_v$. 
\end{enumerate}
\end{enumerate}
Then there exists a representation $\rho^\prime\colon G_K\to H^\prime(E)$ that satisfies $\pi\ccirc\rho^\prime\cong\rho$ and is unramified outside of a finite set of places, 
and any such lift is trianguline at the places in $\Sigma_2$. The $B_v^\prime$-parameter $(\delta_v^\second)_{v\in \Sigma_2}$ of a triangulation of $\rho^\prime$ at a place $v\in\Sigma_2$ can be chosen in such a way that, for every $v\in\Sigma_2$, $(\delta_v^\second)^{-1}\delta_v^\prime\colon\Rep_E(B_v^\prime)\to\BPKEPairs$ factors through $\Rep_E(\mu_{q})$. 
\end{thm}

\begin{proof}
The existence of a lift $\rho^\prime$ and the statements about its ramification follow from \cite[Theorem 5.5]{conradlift}. We prove the statement about triangulinity, for which we can restrict to the split strictly trianguline case. 
 
Let $v\in\Sigma_2$ and let $\rho_v$ and $\rho^\prime_v$ be the restrictions of $\rho$ and $\rho^\prime$, respectively, to a decomposition group at $v$. 
Let $\uu$ be a tuple satisfying $\ell(\uu)=q$. 
Let $S^\prime\colon H^\prime\to\GL_n$ be a faithful $E$-representation of $H^\prime$ and let $N=\dim_E\bS^\uu(E^n)$. Denote by $B_n$ the unique Borel subgroup of $\GL_n$ containing $S^\prime(B_v^\prime)$. We show that $S^\prime\ccirc\rho_v^\prime$ is strictly trianguline of $B_n$-parameter $S^\prime\ccirc\delta^\prime$. 
Since the kernel $Z$ of $\pi$ is central of exponent $q$, $Z$ is mapped under $S^\prime$ into the group $\mu_q$ of $q$-roots of unity embedded diagonally in $\GL_n$, and then to the trivial group by $\bS^\uu$. In particular $\bS^\uu\ccirc S^\prime$ factors as $S\ccirc\pi$ for a representation $S\colon H\to\GL_{N}$. Composition with $\rho_v^\prime$ gives
\begin{equation}\label{Srhov}
\bS^\uu\ccirc S^\prime\ccirc\rho_v^\prime\cong S\ccirc\pi\ccirc\rho_v^\prime\cong S\ccirc\rho_v. 
\end{equation}
Let $B_N$ be the unique Borel subgroup of $\GL_N$ containing $S(B_v)$. By assumption $\rho_v$ is strictly trianguline of $B_v$-parameter $\delta_v$, hence $S\ccirc\rho_v$ is trianguline of $B_N$-parameter $S(\delta_v)$ by Remark \ref{SWtri}. From the equivalence \eqref{Srhov} together with Theorems \ref{ptri}(i) and \ref{thsttri}(i) we deduce that $\bS^\uu\ccirc S^\prime\ccirc\rho_v^\prime$ is trianguline of $B_n$-parameter $S(\delta_v)$. Now $\delta_v=\pi\ccirc\delta_v^\prime$ for the $B_v^\prime$-parameter $\delta_v^\prime$ provided by condition (iv-4) of the statement, so the $B_N$-parameter of $\bS^\uu\ccirc S^\prime\ccirc\rho_v^\prime$ is $S(\pi\ccirc\delta_v^\prime)$, that coincides with $\bS^\uu\ccirc S^\prime\ccirc\delta_v^\prime$ by definition of $S$. Thanks to Theorem \ref{thsttri}(ii), we conclude that the representation $S^\prime\ccirc\rho_v^\prime$ is trianguline of a $B_n$-parameter $\delta_v^{S^\prime}$ such that  
\[ \bS^\uu((\delta_v^{S^\prime})^{-1}\cdot S^\prime(\delta_v^\prime)) \] 
is trivial. Since $\mu_q$ is the kernel of $\bS^\uu$, we deduce that 
\[ (\delta_v^{S^\prime})^{-1}\cdot S^\prime(\delta_v^\prime) \] 
factors through $\Rep_E(\mu_q)$.

Since $S^\prime$ is faithful, Proposition \ref{Sstrict}(i) implies that the triangulation of $S^\prime\ccirc\rho_v^\prime$ of $B_n$-parameter $\delta_v^{S^\prime}$ is induced by a triangulation of $\rho_v^\prime$, of some $B_v^\prime$-parameter $\delta_v^\second$ that necessarily satisfies $S^\prime(\delta_v^\second)=\delta_v^{S^\prime}$. By combining this equality with the last sentence of the previous paragraph, we obtain that the $B_n$-parameter
\[ S^\prime(\delta_v^\second)\cdot(S^\prime(\delta_v^\prime)=S^\prime((\delta_v^\second)^{-1}\delta_v^\prime) \]
factors through $\Rep_E(\mu_q)$. Since $S^\prime$ is faithful, we conclude that $(\delta_v^\second)^{-1}\delta_v^\prime$ also factors through $\Rep_E(\mu_q)$.
\end{proof}

As pointed out in Remark \ref{SWtri}, we cannot conclude that $\rho^\prime$ is \emph{strictly} trianguline at the places in $\Sigma_2$.

We give a corollary of Theorem \ref{liftisogtri}, where we relax the condition of $H^\prime\to H$ being a central isogeny, to simply having finite central kernel. 

As before, let $H$ and $H^\prime$ be two quasisplit connected reductive groups over $E$, and this time let $S\colon H^\prime\to H$ be a morphism with finite central kernel. One could take for instance as $S$ any representation $\GL_n\to\GL_m$ that is not a power of the determinant. 
We denote by $Z$ the kernel of $S$ and by $q$ the exponent of $Z$.
Let $\rho\colon G_{F_v}\to H(E)$ be a continuous representation. Let $\Sigma_1$ be the set of places of $F$ that are either archimedean or ramified for $\rho$, and $\Sigma_2$ be a subset of the set of $p$-adic places of $L$. 

\begin{cor}\label{liftisogcor}
Assume that:
\begin{enumerate}[label=(\roman*)]
	\item $(F,\varnothing,q)$ is not in the \emph{special case (for the Grunwald--Wang theorem)} described in \cite[Definition A.1]{conradlift};
	\item $\Sigma_1$ is finite;
	\item for every $v\in\Sigma_1$, the representation $\rho\vert_{G_{F_v}}$ admits a lift $\rho_v^\prime\colon G_{F_v}\to H^\prime(E)$;
	\item for every $v\in\Sigma_2$, there exist:
	\begin{enumerate}[label=(\arabic*)]
		\item a Borel subgroup $B_v$ of $H$ and a maximal torus $T_v$ contained in $B_v$,
		\item a Borel subgroup $B_v^\prime$ of $H^\prime$ and a maximal torus $T_v^\prime$ contained in $B_v^\prime$, such that $S(B_v^\prime)\subset B_v$ and $S(T_v^\prime)\subset T_v$,
		\item a $B_v$-parameter $\delta_v\colon\Rep_E(B_v)\to\BPKEPairs$ such that the representation $\rho\vert_{G_{F_v}}$ is quasi-regular and strictly trianguline of $B_v$-parameter $\delta_v$, 
		\item a $B_v^\prime$-parameter $\delta_v^\prime\colon\Rep_E(B_v^\prime)\to\BPKEPairs$ such that $S(\delta_v^\prime)=\delta_v$. 
	\end{enumerate}
\end{enumerate}
Then there exists a representation $\rho^\prime\colon G_K\to H^\prime(E)$ that satisfies $S\ccirc\rho^\prime\cong\rho$ and is unramified outside of a finite set of places, 
and any such lift is trianguline at the places in $\Sigma_2$. The $B_v^\prime$-parameter $(\delta_v^\second)_{v\in \Sigma_2}$ of a triangulation of $\rho^\prime$ at a place $v\in\Sigma_2$ can be chosen in such a way that, for every $v\in\Sigma_2$, $(\delta_v^\second)^{-1}\delta_v^\prime\colon\Rep_E(B_v^\prime)\to\BPKEPairs$ factors through $\Rep_E(\mu_{q})$.
\end{cor}


\begin{proof}
Factor $S$ as the composition of a central isogeny $\pi\colon H^\prime\to S(H^\prime)$ and the closed embedding $\iota\colon S(H^\prime)\into H$. 
If the representation $\wtl\rho\colon G_K\to S(H^\prime)$ obtained by restricting the target of $\rho$ satisfies the assumptions (i)--(iv) of Theorem \ref{liftisogtri}, then we obtain the thesis. The only non-trivial condition to be checked is that $\wtl\rho$ is strictly trianguline at the places in $\Sigma_2$, of parameters that can be lifted to $H^\prime$. 

For every $v\in\Sigma_2$, the $B_v$-parameter $\delta_v$ admits a lift
\[ \delta_v^\prime\colon\Rep_E(B_v^\prime)\to\BPKEPairs \] 
to a $B_v^\prime$-parameter, hence a lift $\pi(\delta_v^\prime)$ to a $\pi(B_v^\prime)$-parameter. 

Since the embedding $\iota\colon S(H^\prime)\into H$ is a faithful representation of $S(H^\prime)$, and $\iota\ccirc\wtl\rho_v=\rho_v$ is trianguline for every $v\in\Sigma_2$, with $B_v$-parameter $\iota\ccirc S(\delta_v^\prime)=\delta_v$, Proposition \ref{Sstrict}(i) implies that $\wtl\rho_v$ is trianguline for every $v\in\Sigma_2$ with $\pi(B_v^\prime)$-parameter $\pi(\delta_v^\prime)$.
 
Finally, the parameters $\pi(\delta_v^\prime)$ admit the lifts $\delta_v^\prime$ to $H^\prime$, hence all the conditions are fulfilled.
\end{proof}

\medskip

\section{Appendix: algebraic lemmas}

We prove here a few simple lemmas that are used in Section \ref{proof0}. 
We denote by $R$ an integral domain equipped with an action of a group $G$, and by $V$ an $R$-semilinear representation of $G$, that is, a finite free $R$-module equipped with a semilinear action of $G$. 

\begin{lemma}\label{lemmafg}
	Let $f, g$ be two elements of $V$. If there exists $a\in\Frac(R)$ such that $f^{\otimes n}=ag^{\otimes n}$, as elements of $V^{\otimes n}$, then there exists an $n$-th root $b$ of $a$ in $\Frac(R)$ such that $f=bg$.
\end{lemma}

\begin{proof}
	Let $\{v_1,\ldots,v_d\}$ be an $R$-basis of $V$, and write $f=\sum_{i\in \{1,\ldots,d\}}f_iv_i, g=\sum_{i\in \{1,\ldots,d\}}g_iv_i$ for some $f_i,g_i\in R$. Up to replacing $V$ by the $R$-span of a subset of $\{v_1,\ldots,v_d\}$, we can assume that for every $i$ at least one between $f_i$ and $g_i$ is non-zero. For $\ui=\{i_1,\ldots,i_n\}\in\{1,\ldots,d\}^n$, we write $f_\ui=\prod_{j=1,\ldots,n}f_{i_j}$, and similarly for $g_\ui$. By plugging these expansions into the equality $f^{\otimes n}=ag^{\otimes n}$, we find that $f_\ui=ag_\ui$ for every $\ui\in\{1,\ldots,d\}^n$. In particular $f_i^n=ag_i^n$ for every $i$, so that all of the $f_i$ and $g_i$ have to be non-zero. By comparing the equalities $f_\ui=ag_\ui$ for two choices of $\ui$ that differ only at one entry, we find that $f_i/f_j=g_i/g_j$ for every $i,j\in\{1,\ldots,d\}$. This implies that $f=bg$ with $b=f_1/g_1\in\Frac(R)$. A trivial computation gives that $b^n=a$.
\end{proof}

\begin{lemma}\label{lemmafn}
	Let $S$ an integral domain containing $R$, and let $f$ be an element of $V\otimes_RS$. If the tensor $f^{\otimes n}$ in $(V\otimes_RS)^{\otimes n}$ is of the form $w\otimes t$ for some $w\in V^{\otimes n}$ and $t\in S$, then $f$ is of the form $v\otimes s$ for some $v\in V$ and $s\in S$.
\end{lemma}

\begin{proof}
	Let $v_1,\ldots,v_d$ be an $R$-basis of $V$. Write $f$ as a sum $\sum_{i\in \{1,\ldots,d\}}v_i\otimes s_i$ for some $s_i\in S$. Up to replacing $V$ with the linear span of the vectors $v_i$ such that $s_i\neq 0$, we can assume that $s_i\neq 0$ for every $i$. We obtain $f^{\otimes n}=\sum_{\ui\in\{1,\ldots,d\}^n}v_\ui\otimes s_\ui$, we denote by $v_\ui$ the tensor product of the $v_i$ with the indices determined by the $n$-tuple $\ui$ and by $s_\ui$ the analogous product taken inside of $S$. 
	
	By assumption $f^{\otimes n}=w\otimes t$ for some $w\in V^{\otimes n}$ and $t\in S$. Writing $w=\sum_{\ui\in\in \{1,\ldots,d\}^n}r_\ui v_\ui$ for some $t_i\in S$ and comparing this with the expression we had for $f^{\otimes n}$, we obtain that $s_\ui=r_\ui t$ for every $\ui\in\{1,\ldots,d\}^n$. Note that $s_\ui\neq 0$ for every $\ui$ because $s_i\neq 0$ for every $i$. Comparing the last equality for two $n$-tuples that only differ at a single entry, we obtain $s_i/s_1\in\Frac R$ for every $i\in\{1,\ldots,d\}$. Write $s_i=r_{i}s_1$ for all $i$ and some $r_i\in\Frac R$. Let $I$ be the fractional ideal consisting of the $r\in R$ such that $r\sum_ir_iv_i\in V$, where we are considering $\sum_ir_iv_i$ as an element of $\Frac(R)\otimes_RV$. Since $R$ is a principal ideal domain, $I$ is of the form $yR$ for some $y\in\Frac(R)$. Write $f=(\sum_iyr_iv_i)\otimes y^{-1}s_1$. Since $f\in V\otimes_RS$, we must have $y^{-1}s\in S$, hence $f$ is of the desired form.
\end{proof}

We assume from now on that $R$ is a principal ideal domain. Recall that an \emph{$R$-line} $L$ in a finite free $R$-module is a rank 1 submodule, that we call $L$ \emph{saturated} if it is not contained in any other line, and that an \emph{eigenvector} in a semilinear $R$-representation $M$ of $G$ is an element of a $G$-stable $R$-line in $M$.

Let $F$ be a field on which $G$ acts and $h\colon R\to F$ a $G_K$-equivariant injection of rings. We set $V_F=F\otimes_RV$ and equip it with the diagonal action of $G$.

\begin{lemma}\label{line}
	If there exists $f\in V$ such that $1\otimes f$ is an eigenvector in $V_F$, then $f$ is an eigenvector in $V$. 
\end{lemma}

\begin{proof}
	Because of our assumption, for every $g\in G$ there exists $\gamma_g\in F$ such that $g.(1\otimes f=\gamma_g(1\otimes f)$. Since $V$ is a $G$-stable $R$-submodule of $V_F$, we must have $\gamma_g\in\Frac(R)$ for every $g\in G$, where we consider $\Frac(R)$ as a subfield of $F$ via $h$. Hence $1\otimes f\in\Frac(R)\otimes_RV$ generates a $G$-stable $\Frac(R)$-line, and it is enough to prove the statement when $R=\Frac(R)$.
	
	Let $I$ be the largest fractional ideal of $F=\Frac(R)$ satisfying $I(1\otimes f)\subset V$, where we consider $V$ as an $R$-submodule of $V_F$ via $v\mapsto 1\otimes v$. Since $R$ is a principal ideal domain, $I$ is of the form $bR$ for some $b\in F$. We claim that $bf$ generates a $G_K$-stable saturated $R$-line in $V$. Indeed, it is saturated by construction, and for every $g\in G$, $g.(bf)=g.b\cdot g.f=g.b\cdot\gamma_gf=(g.b\cdot\gamma_g\cdot b^{-1})(bf)$, where the coefficient of $bf$ must belong to $R$ by our choice of $b$.
\end{proof}

We equip $V^{\otimes n}$ with the action of $G$ induced by that on $V$. Recall that $R$ is assumed to be a principal ideal domain.

\begin{lemma}\label{linetens}
	If $f$ is an element of $V$, then $f$ is an eigenvector in $V$ if and only if $f^{\otimes n}$ is an eigenvector in $V^{\otimes n}$. 
\end{lemma}

\begin{proof}
	The ``only if'' is obvious. We prove the other implication. Let $g\in G$ and write $g.f^{\otimes n}=af^{\otimes n}$ for some $a\in\Frac(R)$. Since $g.f^{\otimes n}=(g.f)^{\otimes n}$, Lemma \ref{lemmafg} gives that $g.f=bf$ for some $b\in\Frac(R)$. Therefore $f$ generates a $G$-stable $\Frac(R)$-line in $V\otimes_R\Frac(R)$, and by Lemma \ref{line}, it belongs to a $G$-stable $R$-line in $V$.
\end{proof}

\bigskip

\printbibliography

\bigskip

\bigskip

{\setlength{\parindent}{0em} \small{\scshape Andrea Conti

\smallskip

Université du Luxembourg}}

\smallskip

\noindent \textit{E-mail address}: \url{andrea.conti@uni.lu}, \url{contiand@gmail.com}

\end{document}